\documentclass[a4paper,10pt]{amsart}

\usepackage[utf8]{inputenc}
\usepackage[T1]{fontenc}
\usepackage{lmodern}
\usepackage{amsmath}
\usepackage{amsfonts}
\usepackage{amssymb}
\usepackage{amsthm}
\usepackage{thmtools}
\usepackage{mathtools}
\usepackage{enumitem}
\usepackage{tikz-cd}
\usepackage{multirow}
\usepackage{comment}
\usepackage{placeins}
\usepackage{array}
\usepackage{diagbox}
\newcolumntype{P}[1]{>{\centering\arraybackslash}p{#1}}
\newcolumntype{Q}[1]{>{\raggedleft\arraybackslash}p{#1}}
\newcolumntype{R}[1]{>{\raggedright\arraybackslash}p{#1}}
\usepackage[pdftex,
            pdfauthor={Sven Möller and Nils R. Scheithauer},
            pdftitle={Dimension Formulae and Generalised Deep Holes of the Leech Lattice Vertex Operator Algebra},
            pdfsubject={},
            pdfkeywords={},
            pdfproducer={},
            pdfcreator={}]{hyperref}

\theoremstyle{plain}
\newtheorem{thm}{Theorem}[section]
\newtheorem{prop}[thm]{Proposition}
\newtheorem{cor}[thm]{Corollary}
\newtheorem{lem}[thm]{Lemma}

\newtheorem*{thm*}{Theorem}
\newtheorem*{conj*}{Conjecture}
\newtheorem*{prop*}{Proposition}

\newtheoremstyle{narrow}
  {.5em} 
  {.5em} 
  {\itshape} 
  {} 
  {\bfseries} 
  {.} 
  {.5em} 
  {} 

\theoremstyle{narrow}

\newtheorem*{thn*}{Theorem}

\newtheorem*{conjn*}{Conjecture}

\theoremstyle{definition}
\newtheorem{defi}[thm]{Definition}

\newtheorem*{nota*}{Notation}

\newcommand{\Q}{\mathbb{Q}}
\newcommand{\Z}{\mathbb{Z}}
\newcommand{\Ns}{\mathbb{Z}_{>0}}
\newcommand{\N}{\mathbb{Z}_{\geq0}}
\newcommand{\C}{\mathbb{C}}
\newcommand{\R}{\mathbb{R}}
\newcommand{\HH}{H}

\newcommand{\tr}{\operatorname{tr}}
\renewcommand{\i}{\mathrm{i}}
\newcommand{\e}{\mathrm{e}}

\newcommand{\Aut}{\operatorname{Aut}}
\newcommand{\Hom}{\operatorname{Hom}}

\newcommand{\rk}{\operatorname{rk}}

\newcommand{\sign}{\operatorname{sign}}
\newcommand{\voa}{vertex operator algebra}

\newcommand{\VOA}{Vertex Operator Algebra}
\newcommand{\vosa}{vertex operator subalgebra}

\newcommand{\fpvosa}{fixed-point vertex operator subalgebra}

\newcommand{\gdh}{generalised deep hole}

\newcommand{\GDH}{Generalised Deep Hole}

\newcommand{\ch}{\operatorname{ch}}
\newcommand{\Ch}{\operatorname{Ch}}

\newcommand{\id}{\operatorname{id}}
\newcommand{\amgis}{\zeta}
\newcommand{\eps}{\varepsilon}
\newcommand{\lcm}{\operatorname{lcm}}
\newcommand{\SLZ}{\operatorname{SL}_2(\mathbb{Z})}

\newcommand{\ee}{e}
\newcommand{\g}{\mathfrak{g}}
\newcommand{\hh}{\mathfrak{h}}
\newcommand{\h}{{L_\C}}
\newcommand{\ad}{\operatorname{ad}}

\newcommand{\CC}{\operatorname{C}}
\newcommand{\orb}{\operatorname{orb}}

\newcommand{\strathol}{strongly rational, holomorphic}
\newcommand{\strat}{strongly rational}
\newcommand{\II}{I\!I}
\newcommand{\spn}{\operatorname{span}}

\renewcommand{\O}{\operatorname{O}}
\newcommand{\Co}{\operatorname{Co}}

\renewcommand{\Im}{\operatorname{Im}}
\newcommand{\fqs}{discriminant form}

\newcommand{\oddity}{\operatorname{oddity}}
\newcommand{\mmod}{\!\mod}

\newlength{\myl}
\begin{document}

\title[Generalised Deep Holes of the Leech Lattice VOA]{Dimension Formulae and Generalised Deep Holes of the Leech Lattice Vertex Operator Algebra}
\author[Sven Möller and Nils~R. Scheithauer]{Sven Möller\textsuperscript{\lowercase{a},\lowercase{b}}and Nils~R. Scheithauer\textsuperscript{\lowercase{c}}}
\thanks{\textsuperscript{a}{Rutgers University, Piscataway, NJ, United States of America}}
\thanks{\textsuperscript{b}{Research Institute for Mathematical Sciences, Kyoto University, Kyoto, Japan}}
\thanks{\textsuperscript{c}{Technische Universität Darmstadt, Darmstadt, Germany}}
\dedicatory{Dedicated to the memory of John Horton Conway.}

\begin{abstract}
We prove a dimension formula for the weight-$1$ subspace of a \voa{} $V^{\orb(g)}$ obtained by orbifolding a \strathol{} \voa{} $V$ of central charge $24$ with a finite-order automorphism $g$. Based on an upper bound derived from this formula we introduce the notion of a \gdh{} in $\Aut(V)$.

Then we show that the orbifold construction defines a bijection between the \gdh{}s of the Leech lattice \voa{} $V_\Lambda$ with non-trivial fixed-point Lie subalgebra and the \strathol{} \voa{}s of central charge $24$ with non-vanishing weight-$1$ space. This provides a uniform construction of these \voa{}s and naturally generalises the correspondence between the deep holes of the Leech lattice $\Lambda$ and the $23$ Niemeier lattices with non-vanishing root system found by Conway, Parker, Sloane and Borcherds.
\end{abstract}

\vspace*{-20pt}
\maketitle

\vspace*{-20pt}
\setcounter{tocdepth}{1}
\tableofcontents


\vspace*{-20pt}
\section{Introduction}
In 1968 Niemeier classified the positive-definite, even, unimodular lattices of rank $24$ \cite{Nie68,Nie73}. He showed that up to isomorphism there are exactly $24$ such lattices and that the isomorphism class is uniquely determined by the root system. The Leech lattice $\Lambda$ is the unique lattice in this genus without roots. There are several proofs of this result. Niemeier applied Kneser's neighbourhood method. Venkov found a proof based on harmonic theta series \cite{Ven80}. It can also be derived from Conway, Parker and Sloane's classification of deep holes of the Leech lattice \cite{CPS82,Bor85b} and from the Smith-Minkowski-Siegel mass formula \cite{CS82c,CS99}.
Finally, it also follows from the classification of certain automorphic representations of $\O_{24}$ \cite{CL19}.

Conway, Parker and Sloane found a nice construction of the Niemeier lattices starting from the Leech lattice \cite{CPS82,CS82}. They showed that up to equivalence there are exactly $23$ deep holes in the Leech lattice $\Lambda$, i.e.\ points in $\Lambda\otimes_\Z\R$ which have maximal distance to the Leech lattice, and that they are in bijection with the Niemeier lattices different from the Leech lattice. The construction is as follows. Let $d$ be a deep hole corresponding to the Niemeier lattice $N$. Then the $\Z$-module in $\Lambda\otimes_\Z\R$ generated by $\{x\in\Lambda\,|\,(x,d)\in\Z\}$ and $d$ is isomorphic to $N$.

The classification of \strathol{} \voa{}s of central charge $24$ bears similarities to the classification of positive-definite, even, unimodular lattices of rank $24$. The weight-$1$ subspace $V_1$ of a \strathol{} \voa{} $V$ of central charge $24$ is a reductive Lie algebra.
In 1993 Schellekens \cite{Sch93} (see also \cite{EMS20a}) showed that there are at most $71$ possibilities for this Lie algebra using the theory of Jacobi forms. He conjectured that all potential Lie algebras are realised and that the $V_1$-structure fixes the \voa{} up to isomorphism. By the work of many authors over the past three decades the following result is now proved.
\begin{thn*}
Up to isomorphism there are exactly $70$ \strathol{} \voa{}s $V$ of central charge $24$ with $V_1\neq\{0\}$. Such a \voa{} is uniquely determined by its $V_1$-structure.
\end{thn*}
The proof is based on a case-by-case analysis and uses a variety of methods. One of the main results of this paper is a uniform proof of the existence part which generalises Conway and Sloane's construction of the Niemeier lattices from the Leech lattice.

One method to construct \voa{}s is the orbifold construction \cite{EMS20a}. Let $V$ be a \strathol{} \voa{} and $g$ an automorphism of $V$ of finite order~$n$ and type~$0$. Then the fixed-point subalgebra $V^g$ is a \strat{} \voa{}. It has exactly $n^2$ non-isomorphic irreducible modules, which can be realised as the eigenspaces of $g$ acting on the twisted modules $V(g^i)$ of $V$. If the twisted modules $V(g^i)$ have positive conformal weight for $i\neq0\mmod{n}$, then the sum $V^{\orb(g)}:=\oplus_{i\in\Z/n\Z}V(g^i)^g$ is a \strathol{} \voa{}. There is also an inverse orbifold construction, i.e.\ an automorphism $h$ of $V^{\orb(g)}$ such that $(V^{\orb(g)})^{\orb(h)}\cong V$.

We can use the deep holes of the Leech lattice to construct the \voa{}s corresponding to the Niemeier lattices.
\begin{thn*}[\autoref{prop:deepholecons}]
Let $d\in\Lambda\otimes_\Z\R$ be a deep hole in the Leech lattice $\Lambda$ corresponding to the Niemeier lattice $N$. Then $g=\e^{2\pi\i d_0}$ is an automorphism of the \voa{} $V_{\Lambda}$ associated with $\Lambda$ of order equal to the Coxeter number of $N$ and type~$0$. The corresponding orbifold construction $V_{\Lambda}^{\orb(g)}$ is isomorphic to the \voa{} $V_N$ associated with $N$.
\end{thn*}
We shall see that the other elements on Schellekens' list can be constructed in a similar way.

Modular forms for the Weil representation of $\SLZ$ play an important role in many areas of mathematics. The simplest examples are theta series. Let $L$ be a positive-definite, even lattice of even rank $2k$ and $D=L'/L$ its \fqs{}. Then $\theta(\tau)=\sum_{\gamma\in D}\theta_{\gamma}(\tau)\ee^{\gamma}$ with $\theta_{\gamma}(\tau)=\sum_{\alpha\in\gamma+L}q^{(\alpha,\alpha)/2}$ is a modular form of weight~$k$ for the Weil representation of $D$.

Another example comes from orbifold theory. Under the same conditions as above the $n^2$ characters of the irreducible modules of $V^g$ combine to a modular form of weight~$0$ for the Weil representation
of the hyperbolic lattice $\II_{1,1}(n)$. Pairing this form with a certain Eisenstein series of weight~$2$ for the dual Weil representation we obtain:
\begin{thn*}[\autoref{thm:dimform2} and \autoref{cor:dimbound2}]
Let $V$ be a \strathol{} \voa{} of central charge $24$ and $g$ an automorphism of $V$ of finite order $n>1$ and type~$0$. We assume that the twisted modules $V(g^i)$ have positive conformal weight for $i\neq0\mmod{n}$. Then
\begin{equation*}
\dim(V_1^{\orb(g)})=24+\sum_{m\mid n}c_n(m)\dim(V_1^{g^m})-R(g)
\end{equation*}
where the $c_n(m)$ are defined by $\sum_{m\mid n}c_n(m)(t,m)=n/t$ for all $t\mid n$ and the rest term $R(g)$ is non-negative. In particular
\begin{equation*}
\dim(V_1^{\orb(g)})\leq24+\sum_{m\mid n}c_n(m)\dim(V_1^{g^m}).
\end{equation*}
\end{thn*}
The rest term $R(g)$ is given explicitly. It depends on the dimensions of the weight spaces of the irreducible $V^g$-modules of weight less than 1.

The dimension formula was first proved by Montague for $n=2,3$ \cite{Mon94}, then generalised to $n=5,7,13$ \cite{Moe16} and finally to all $n>1$ such that $\Gamma_0(n)$ has genus~$0$ in \cite{EMS20b}. The previous proofs all used explicit expansions of the corresponding Hauptmoduln. We show here that the dimension formula is an obstruction coming from the Eisenstein space. Pairing the character of $V^g$ with other modular forms for the dual Weil representation we can obtain further restrictions.

The above upper bound is our motivation for the following definition. Let $V$ be a \strathol{} \voa{} of central charge $24$ and $g$ an automorphism of $V$ of finite order $n>1$. Suppose $g$ has type~$0$ and $V^g$ satisfies the positivity condition. Then $g$ is called a \emph{\gdh} of $V$ if
\begin{enumerate}
\item the upper bound in the dimension formula is attained, i.e.\ $\dim(V_1^{\orb(g)})=24+\sum_{m\mid n}c_n(m)\dim(V_1^{g^m})$,
\item $\rk(V_1^{\orb(g)})$ is minimal, i.e.\ $\rk(V_1^{\orb(g)})=\rk(V_1^g)$.
\end{enumerate}
By convention, we also call the identity a \gdh{}.

We show that if $d\in\Lambda\otimes_\Z\R$ is a deep hole of the Leech lattice, then $g=\e^{2\pi\i d_0}$ is a \gdh{} of the \voa{} $V_{\Lambda}$.

The existence of \gdh{}s is restricted by Deligne's bound on the Fourier coefficients of cusp forms. Pairing the character of $V^g$ with a certain cusp form we obtain:
\begin{thn*}[\autoref{thm:delignebound}]
In the above situation suppose that the order of $g$ is a prime $p$ such that $\Gamma_0(p)$ has positive genus. Then $R(g)\geq24$. In particular, $g$ is not a \gdh{}.
\end{thn*}

Using inverse orbifolding and an averaged version of Kac's very strange formula \cite{Kac90} we show (recall that algebraic conjugacy means conjugacy of cyclic subgroups):
\begin{thn*}[\autoref{thm:main2}]
The cyclic orbifold construction $g\mapsto V_\Lambda^{\orb(g)}$ defines a bijection between the algebraic conjugacy classes of \gdh{}s $g$ in $\Aut(V_\Lambda)$ with $\rk((V_\Lambda^g)_1)>0$ and the isomorphism classes of \strathol{} \voa{}s $V$ of central charge $24$ with $V_1\neq\{0\}$.
\end{thn*}

Then, we give a uniform construction of the $70$ non-trivial Lie structures on Schellekens' list:
\begin{thn*}[\autoref{thm:main3}]
Let $\g$ be one of the $70$ non-zero Lie algebras on Schelle\-kens' list (Table~1 in \cite{Sch93}). Then there exists a \gdh{} $g\in\Aut(V_{\Lambda})$ such that $(V_\Lambda^{\orb(g)})_1\cong\g$.
\end{thn*}
We explicitly give a \gdh{} for each case. In order to determine the $V_1$-structure of the corresponding orbifold construction we use Schellekens' classification of the possible $V_1$-structures and Kac's theory of finite-order automorphisms of simple Lie algebras \cite{Kac90}.

Our results generalise the holy construction of the Niemeier lattices with non-trivial root system by Conway, Parker, Sloane and Borcherds.

We remark that another uniform construction of the \voa{}s on Schellekens' list with non-trivial $V_1$ was given by Höhn (see \cite{Hoe17} and \cite{Lam20}) using a different approach based on commuting pairs.

The classification of the \strathol{} \voa{}s of central charge $24$ with non-trivial $V_1$ described above and \autoref{thm:main2} yield the classification of the \gdh{}s of $V_\Lambda$ with $\rk((V_\Lambda^g)_1)>0$. We observe:
\begin{thn*}[\autoref{thm:main4}]
Under the natural projection $\Aut(V_\Lambda)\to\O(\Lambda)$ the $70$ algebraic conjugacy classes of \gdh{}s $g$ with $\rk((V_\Lambda^g)_1)>0$ map to the $11$ algebraic conjugacy classes in $\O(\Lambda)$ with cycle shapes $1^{24}$, $1^82^8$, $1^63^6$, $2^{12}$, $1^42^24^4$, $1^45^4$, $1^22^23^26^2$, $1^37^3$, $1^22^14^18^2$, $2^36^3$ and $2^210^2$.
\end{thn*}
More precisely, we exactly recover the decomposition of the genus of the Moonshine module described by Höhn in \cite{Hoe17} (see \autoref{table:11}).

Finally, we give a full classification of the \gdh{}s of the Leech lattice \voa{} $V_\Lambda$, including those with $\rk((V_\Lambda^g)_1)=0$, based on Carnahan's constructions of the Moonshine module \cite{Car18} (see \autoref{thm:class}).


\subsection*{Outline}
In \autoref{sec:la} we review finite-order automorphisms of simple Lie algebras and prove an averaged version of Kac's very strange formula.

In \autoref{sec:voa} we recall some results on lattice \voa{}s, automorphisms of the \voa{} $V_\Lambda$ associated with the Leech lattice and the cyclic orbifold construction.

In \autoref{sec:mf} we review modular forms of weight $2$ for $\Gamma_0(n)$ and vector-valued modular forms for the Weil representation. Then we construct a certain Eisenstein series of weight $2$ for the dual Weil representation of $\II_{1,1}(n)$.

In \autoref{sec:dimform} we derive a formula for the dimension of the weight-$1$ space of the cyclic orbifold construction by pairing the character of $V^g$ with the Eisenstein series of the preceding section and show that the dimension is bounded from above. We introduce the notion of extremal automorphisms, for which the upper bound is attained.

In \autoref{sec:gdh} we define \gdh{}s, certain extremal automorphisms, and prove a bijection between the \strathol{} \voa{}s of central charge $24$ with non-vanishing weight-$1$ space and the \gdh{}s of $V_\Lambda$ with non-vanishing fixed-point Lie subalgebra. Then we describe $70$ \gdh{}s realising all non-zero Lie algebras on Schellekens' list.


\subsection*{Notation}
Throughout this text we use the notation $e(x):=\e^{2\pi\i x}$ and $\Z_n:=\Z/n\Z$ for $n\in\Ns$. The variable $\tau$ is always from the complex upper half-plane $\HH=\{z\in\C\,|\,\Im(z)>0\}$ and $q=e(\tau)=\e^{2\pi\i\tau}$. The Dedekind eta function is defined as $\eta(\tau)=q^{1/24}\prod_{n=1}^{\infty}(1-q^n)$. We set $\Gamma:=\SLZ$, and we denote by $\sigma=\sigma_1$ the usual sum-of-divisors function extended to the positive rational numbers by $\sigma(r)=0$ if $r\notin\Ns$.


\subsection*{Acknowledgements}
The authors thank Moritz Dittmann, Jethro van Ekeren, Gerald Höhn, Yi-Zhi Huang, Ching Hung Lam, Jim Lepowsky, Masahiko Miyamoto, Hiroki Shimakura, Jakob Stix and Don Zagier for helpful discussions, and the anonymous referees for their suggestions.

Sven Möller was supported by a JSPS \emph{Postdoctoral Fellowship for Research in Japan} and by JSPS Grant-in-Aid KAKENHI 20F40018. Nils Scheithauer was supported by the DFG through the CRC \emph{Geometry and Arithmetic of Uniformized Structures}, project number 444845124. Both authors also thank the LOEWE research unit \emph{Uniformized Structures in Arithmetic and Geometry} for its support.


\section{Lie Algebras}\label{sec:la}
In this section we review the theory of finite-order automorphisms of finite-dimensional, simple Lie algebras and the very strange formula \cite{Kac90}. Then we prove an averaged version of the very strange formula.

The Lie algebras in this paper are algebras over $\C$.


\subsection{Finite-Order Automorphisms of Simple Lie Algebras}\label{sec:lieaut}
The finite-order automorphisms of simple Lie algebras were classified in \cite{Kac90}, Chapter~8. Since we shall only need inner automorphisms, we restrict to this case.

Let $\g$ be a finite-dimensional, simple Lie algebra of type $X_l$. Fix a Cartan subalgebra $\hh$ with corresponding root system $\Phi\subseteq\hh^*$ and a set of simple roots $\{\alpha_1,\ldots,\alpha_l\}\subseteq\Phi$. The height of a root $\alpha=\sum_{i=1}^lm_i\alpha_i\in\Phi$ is defined as $\operatorname{ht}(\alpha)=\sum_{i=1}^lm_i$. There is a unique root $\theta\in\Phi$ of maximal height. We normalise the non-degenerate, symmetric, invariant bilinear form $(\cdot,\cdot)$ on $\g$ such that $(\theta,\theta)=2$. The Weyl vector $\rho$ of $\g$ is the half sum of the positive roots $\rho=(1/2)\sum_{\alpha\in\Phi^+}\alpha$ (or the sum of the fundamental weights). The relation $\theta=\sum_{i=1}^la_i\alpha_i$ defines the Kac labels (or marks) $a_1,\ldots,a_l$, complemented with $a_0=1$. The untwisted affine Dynkin diagram $X_l^{(1)}$ associated with $X_l$ has $l+1$ nodes labelled by the indices $0,\ldots,l$ (see Table Aff~1 in \cite{Kac90}).

By Theorem~8.6 in \cite{Kac90} the conjugacy classes of inner automorphisms of $\g$ of order $n\in\Ns$ are in bijection with the sequences $s=(s_0,\ldots,s_l)$ of non-negative, coprime integers satisfying $\sum_{i=0}^la_is_i=n$, modulo automorphisms of the affine Dynkin diagram $X_l^{(1)}$. The conjugacy class associated with a sequence $s$ is represented by the inner automorphism $\sigma_s$ uniquely specified by the relations
\begin{equation*}
\sigma_s(x_i)=e(s_i/n)x_i
\end{equation*}
for $i=1,\ldots,l$ where $x_i$ spans the $1$-dimensional root space $\g_{\alpha_i}$ corresponding to the simple root $\alpha_i$.

The automorphism $\sigma_s$ is characterised by the vector $\lambda_s\in\hh^*$ defined by $(\alpha_i,\lambda_s)=s_i/n$ for all simple roots $\alpha_i$. Identifying the Cartan subalgebra $\hh$ with its dual $\hh^*$ via $(\cdot,\cdot)$ we can write $\sigma_s=\e^{2\pi\i\ad(\lambda_s)}$.


\subsection{Very Strange Formula}\label{sec:verystrange}
The very strange formula, which generalises the strange formula of Freudenthal--de~Vries, relates the dimensions of the eigenspaces of a finite-order automorphism of a simple Lie algebra to the distance of the corresponding vector to the Weyl vector. Here, we only state it for inner automorphisms.

Let $\g$ be a finite-dimensional, simple Lie algebra with dual Coxeter number $h^\vee$. Given an automorphism $\sigma$ of $\g$ of order $n$, we define
\begin{equation*}
B(\sigma):=\frac{\dim(\g)}{24}-\frac{1}{4n^2}\sum_{j=1}^{n-1}j(n-j)\dim(\g_{(j)})
\end{equation*}
where $\g_{(j)}$ is the eigenspace $\g_{(j)}=\{x\in\g\,|\,\sigma x=e(j/n)x\}$. We note that $B(\sigma)$ can be written as
\begin{equation*}
B(\sigma)=\frac{1}{4}\sum_{j=0}^{n-1}B_2(j/n)\dim(\g_{(j)})
\end{equation*}
where $B_2(x)=x^2-x+1/6$ is the second Bernoulli polynomial. In both formulae we may replace the (exact) order $n$ of $\sigma$ by any multiple of it.

\begin{prop}[Very Strange Formula, formula~(12.3.6) in \cite{Kac90}]\label{prop:verystrange}
Let $\g$ be a finite-dimensional, simple Lie algebra and $\sigma$ an inner automorphism of $\g$ of order~$n$ corresponding to the sequence $s=(s_0,\ldots,s_l)$. Then
\begin{equation*}
\frac{h^\vee}{2}\left|\lambda_s-\frac{\rho}{h^\vee}\right|^2=\frac{\dim(\g)}{24}-\frac{1}{4n^2}\sum_{j=1}^{n-1}j(n-j)\dim(\g_{(j)}).
\end{equation*}
\end{prop}
The squared norm on the left-hand side is formed with respect to the bilinear form on $\hh^*$ induced by $(\cdot,\cdot)$.

We define the automorphism $\sigma_{\rho/h^\vee}=\e^{2\pi\i\ad(\rho/h^\vee)}$, whose conjugacy class is characterised by $\lambda_s=\rho/h^\vee$. The following corollary is immediate:
\begin{cor}\label{cor:B}
Let $\g$ be a finite-dimensional, simple Lie algebra. Then $B(\sigma)\geq0$ for all finite-order, inner automorphisms $\sigma$ of $\g$, and equality is attained if and only if
$\sigma$ is conjugate to $\sigma_{\rho/h^\vee}$.
\end{cor}

The automorphism $\sigma_{\rho/h^\vee}$ has order $n=lh^\vee$ where $l\in\{1,2,3\}$ denotes the lacing number of $\g$, is regular, i.e.\ the fixed-point Lie subalgebra is abelian, and rational, i.e.\ $\sigma_{\rho/h^\vee}$ is conjugate to $\sigma_{\rho/h^\vee}^i$ for all $i\in\Z_n^*$. In particular, $\sigma_{\rho/h^\vee}$ is quasirational, i.e.\ the dimensions of the eigenspaces $\g_{(j)}$ only depend on $(j,n)$.

If $\g$ is simply-laced, then $\sigma_{\rho/h^\vee}$ is the up to conjugacy unique minimal-order regular automorphism of $\g$ and $s=(1,\ldots,1)$ (see Exercise~8.11 in \cite{Kac90}).

Finally, we note that if an automorphism $\sigma$ of $\g$ is conjugate to $\sigma_{\rho/h^\vee}$, then it must already be conjugate to $\sigma_{\rho/h^\vee}$ under an inner automorphism of $\g$ because the sequence $s$ characterising $\sigma_{\rho/h^\vee}$ is invariant under the diagram automorphisms of the untwisted affine Dynkin diagram of $\g$.


\subsection{Averaged Very Strange Formula}\label{sec:averageverystrange}
In the following we prove an averaged version of Kac's very strange formula. Remarkably, the coefficients of the eigenspaces that will appear in this formula are identical to those in the dimension formula in \autoref{thm:dimform2}.

Given an automorphism $\sigma$ of $\g$ of order $n$ (or dividing $n$), we define
\begin{equation*}
A(\sigma):=\frac{1}{\phi(n)}\sum_{i\in\Z_n^*}B(\sigma^i).
\end{equation*}
Clearly, $A(\sigma)=B(\sigma)$ if $\sigma$ is quasirational. This is the case that was studied in \cite{ELMS21}. Here, we do not make this assumption.

The definition of $A(\sigma)$ and \autoref{cor:B} imply:
\begin{prop}
Let $\g$ be a finite-dimensional, simple Lie algebra. Then $A(\sigma)\geq0$ for all finite-order, inner automorphisms $\sigma$ of $\g$, and equality is attained if and only if
$\sigma$ is conjugate to $\sigma_{\rho/h^\vee}$.
\end{prop}
\begin{proof}
Let $\sigma$ be an inner automorphism of $\g$ of order $n$. Then $B(\sigma^i)\geq0$ for all $i\in\Z_n^*$ implies $A(\sigma)\geq0$. If $A(\sigma)=0$, then $B(\sigma)=0$ so that $\sigma$ is conjugate to $\sigma_{\rho/h^\vee}$. Conversely, if $\sigma$ is conjugate to $\sigma_{\rho/h^\vee}$, then so is $\sigma^i$ for all $i\in\Z_n^*$ by the rationality of $\sigma_{\rho/h^\vee}$. Hence, $B(\sigma^i)=0$ for all $i\in\Z_n^*$, and therefore $A(\sigma)=0$.
\end{proof}

We now compute the coefficients of the eigenspace dimensions $\dim(\g_{(j)})$ in $A(\sigma)$. Because of the averaging these coefficients only depend on $(j,n)$. Equivalently, we may rewrite $A(\sigma)$ as a linear combination of the fixed-point dimensions $\dim(\g^{\sigma^m})$ for $m\mid n$.

\enlargethispage{\baselineskip}

Let $n\in\Ns$. We define the constants $c_n(m)\in\Q$ for $m\mid n$ by the linear system of equations
\begin{equation*}
\sum_{m\mid n}c_n(m)(t,m)=n/t
\end{equation*}
for all $t\mid n$. An explicit formula will be given in \autoref{sec:dimform2}.

\begin{thm}[Averaged Very Strange Formula]\label{thm:averageverystrange}
Let $\g$ be a finite-dimensional, simple Lie algebra and $\sigma$ an inner automorphism of $\g$ of order~$n$. Then
\begin{equation*}
\frac{1}{\phi(n)}\sum_{i\in\Z_n^*}\frac{h^\vee}{2}\left|\lambda_{s^{[i]}}-\frac{\rho}{h^\vee}\right|^2=\frac{1}{24n}\sum_{m\mid n}c_n(m)\dim(\g^{\sigma^m})
\end{equation*}
where $s^{[i]}$ is the sequence characterising $\sigma^i$.
\end{thm}
Again, we may replace the order $n$ in the above formula by any multiple of it.
\begin{proof}
We assume that $n>1$. By the very strange formula (see \autoref{prop:verystrange}) and the definition of $A(\sigma)$, the left-hand side equals
\begin{align*}
A(\sigma)&=\frac{1}{\phi(n)}\sum_{i\in\Z_n^*}\frac{1}{4}\sum_{j=0}^{n-1}B_2(j/n)\dim(\g_{(ij)})\\
&=\sum_{j=0}^{n-1}\Bigg(\frac{1}{\phi(n/(j,n))}\!\!\sum_{\substack{i=0\\(i,n)=(j,n)}}^{n-1}\!\!\frac{1}{4}B_2(i/n)\Bigg)\dim(\g_{(j)}).
\end{align*}
Let $k>1$. Then
\begin{equation*}
\sum_{\substack{i=1\\(i,k)=1}}^{k}\!\!1=\phi(k),\quad
\sum_{\substack{i=1\\(i,k)=1}}^{k}\!\!i=\frac{1}{2}k\phi(k),\quad
\sum_{\substack{i=1\\(i,k)=1}}^{k}\!\!i^2=\frac{1}{3}\phi(k)(k^2+\frac{1}{2}\lambda(k))
\end{equation*}
with
\begin{equation*}
\lambda(k):=\prod_{\substack{p\mid k\\p\text{ prime}}}(-p).
\end{equation*}
(For the third sum see, e.g., \cite{PS71}, Abschnitt~VIII, Aufgabe~27.) Hence, summing $B_2(i/n)=(i/n)^2-i/n+1/6$ over $i$ yields
\begin{equation*}
\frac{1}{\phi(n/(j,n))}\!\!\sum_{\substack{i=0\\(i,n)=(j,n)}}^{n-1}\!\frac{1}{4}B_2(i/n)=\frac{1}{24}\frac{\lambda(n/(j,n))}{(n/(j,n))^2}.
\end{equation*}
Noting that
\begin{equation*}
\sum_{\substack{j=0\\(j,n)=d}}^{n-1}\!\!\dim(\g_{(j)})=\sum_{m\mid(n/d)}\mu(n/dm)\dim(\g^{\sigma^m})
\end{equation*}
for all $d\mid n$ with the Möbius function $\mu$ we obtain
\begin{align*}
A(\sigma)
&=\sum_{m\mid n}\Big(\sum_{d\mid(n/m)}\frac{\lambda(n/d)\mu(n/dm)}{24(n/d)^2}\Big)\dim(\g^{\sigma^m}).
\end{align*}
Finally we show that
\begin{equation*}
\sum_{m\mid n}\Big(\sum_{d\mid(n/m)}\frac{\lambda(n/d)\mu(n/dm)}{(n/d)^2}\Big)(t,m)=\frac{1}{t}
\end{equation*}
for all $t\mid n$ using the multiplicativity of the left-hand side in $n$, which implies that
\begin{equation*}
\sum_{d\mid(n/m)}\frac{\lambda(n/d)\mu(n/dm)}{(n/d)^2}=\frac{1}{n}c_n(m).
\end{equation*}
Now the assertion follows.
\end{proof}


\section{\VOA{}s}\label{sec:voa}
In this section we recall some results about lattice \voa{}s and their automorphism groups and describe the cyclic orbifold construction developed in \cite{EMS20a,Moe16}.

The \voa{}s in this paper are defined over $\C$.


\subsection{Lattice \VOA{}s}\label{sec:lat}
To keep the exposition self-contained we recall a few well-known facts about lattice \voa{}s \cite{Bor86,FLM88}.

For a positive-definite, even lattice $L$ with bilinear form $(\cdot,\cdot)\colon L\times L\to\Z$ the associated \voa{} of central charge $c=\rk(L)$ is given by
\begin{equation*}
V_L=M(1)\otimes\C[L]_\eps
\end{equation*}
where $M(1)$ is the Heisenberg \voa{} associated with $\h=L\otimes_\Z\C$ and $\C[L]_\eps$ is the twisted group algebra, i.e.\ the algebra with basis $\{\ee^\alpha\,|\,\alpha\in L\}$ and products $\ee^\alpha\ee^\beta=\eps(\alpha,\beta)\ee^{\alpha+\beta}$ for all $\alpha,\beta\in L$ where $\eps\colon L\times L\to\{\pm1\}$ is a $2$-cocycle satisfying $\eps(\alpha,\beta)/\eps(\beta,\alpha)=(-1)^{(\alpha,\beta)}$. Note that $\eps$ defines a central extension $\hat{L}$ of $L$ by $1\to\{\pm1\}\to\{\pm1\}\times_\eps L\to L\to0$ and $\C[L]_\eps=\C[\hat{L}]/J$ where $J$ is the kernel of the linear homomorphism $\C[\hat{L}]\to\C[L]$ defined by $\ee^{(\pm1,\alpha)}\mapsto\pm\ee^{\alpha}$.

\medskip

Let $\nu\in\O(L)$, the orthogonal group of the lattice $L$, and let $\eta\colon L\to\{\pm1\}$ be a function. Then the map $\phi_\eta(\nu)$ acting on $\C[L]_\eps$ as $\phi_\eta(\nu)(\ee^\alpha)=\eta(\alpha)\ee^{\nu\alpha}$ for $\alpha\in L$ defines an automorphism of $\C[L]_\eps$ if and only if
\begin{equation*}
\frac{\eta(\alpha)\eta(\beta)}{\eta(\alpha+\beta)}=\frac{\eps(\alpha,\beta)}{\eps(\nu\alpha,\nu\beta)}
\end{equation*}
for all $\alpha,\beta\in L$. In this case $\phi_\eta(\nu)$ is called a \emph{lift} of $\nu$. Since the lifts can be identified with the automorphisms of $\hat{L}$ which preserve the norm in $L$, we denote the group of lifts by $\O(\hat{L})$. In contrast to $\O(L)$ the group $\O(\hat{L})$ acts naturally on the \voa{} $V_L=M(1)\otimes\C[L]_\eps$.

There is a short exact sequence
\begin{equation*}
1\to\Hom(L,\{\pm1\})\to\O(\hat{L})\to\O(L)\to1
\end{equation*}
with the surjection $\O(\hat{L})\to\O(L)$ given by $\phi_\eta(\nu)\mapsto\nu$. The image of $\Hom(L,\{\pm1\})$ in $\O(\hat{L})$ are exactly the lifts of the identity in $\O(L)$.

If the restriction of $\eta$ to the fixed-point sublattice $L^\nu$ is trivial, we call $\phi_\eta(\nu)$ a \emph{standard lift} of $\nu$. It is always possible to choose $\eta$ in this way (see \cite{Lep85}, Section~5).

Let $\phi_\eta(\nu)$ be a standard lift of $\nu$, and suppose that $\nu$ has order~$m$. If $m$ is odd or if $m$ is even and $(\alpha,\nu^{m/2}\alpha)$ is even for all $\alpha\in L$, then the order of $\phi_\eta(\nu)$ is also $m$. Otherwise the order of $\phi_\eta(\nu)$ is $2m$, in which case we say $\nu$ exhibits \emph{order doubling}.

For any \voa{} $V$ of CFT-type $K:=\langle\e^{v_0}\,|\,v\in V_1\rangle$ is a subgroup of $\Aut(V)$, called the \emph{inner automorphism group} of $V$. By \cite{DN99}, Theorem~2.1, the automorphism group of $V_L$ can be written as
\begin{equation*}
\Aut(V_L)=\O(\hat{L})\cdot K.
\end{equation*}
Here, $K$ is normal in $\Aut(V_L)$, $\Hom(L,\{\pm1\})$ is a subgroup of $K\cap\O(\hat{L})$ and $\Aut(V_L)/K$ is isomorphic to a quotient of $\O(L)$.

All standard lifts of a given lattice automorphism $\nu\in\O(L)$ are conjugate in $\Aut(V_L)$, in fact under elements of $K$ (see \cite{EMS20a}, Proposition~7.1).

Choosing a section $\O(L)\to\O(\hat{L})$, $\nu\mapsto\phi_\eta(\nu)$, we can write every automorphism of $V_L$ as $\phi_\eta(\nu)\sigma$ for some $\nu\in\O(L)$, $\sigma\in K$ because any two lifts of $\nu$ only differ by a homomorphism $L\to\{\pm 1\}$, which can be absorbed into $\sigma$.

\medskip

We also recall some results about the twisted modules of lattice \voa{}s. The irreducible $\phi_\eta(\nu)$-twisted modules of a lattice \voa{} $V_L$ for standard lifts $\phi_\eta(\nu)$ are described in \cite{DL96,BK04}. Together with the results in Section~5 of \cite{Li96} this allows us to describe the irreducible $g$-twisted $V_L$-modules for all finite-order automorphisms $g\in\Aut(V_L)$ (cf.\ \cite{HM20}).

By construction, $(V_L)_1$ is a reductive Lie algebra,
and the subspace $\hh:=\{h(-1)\otimes\ee^0\,|\,h\in\h\}\cong\h$ is a Cartan subalgebra of $(V_L)_1$. Now, let $g=\phi_\eta(\nu)\sigma_h$ where $\phi_\eta(\nu)$ is some lift of $\nu\in\O(L)$ and $\sigma_h=\e^{2\pi\i h_0}$ for some $h\in\h$. Note that $g$ has finite order if and only if $h\in L\otimes_\Z\Q$. Suppose $\nu$ has order $m$. In Lemma~8.3 of \cite{EMS20b} (see also Lemma~3.4 in \cite{LS19}) it is shown that $\phi_\eta(\nu)\sigma_h$ is conjugate to $\phi_\eta(\nu)\sigma_{\pi_\nu(h)}$ where $\pi_\nu=\frac{1}{m}\sum_{i=0}^{m-1}\nu^i$ is the projection of $\h$ onto the elements of $\h$ fixed by $\nu$. Hence we may assume that $h\in\pi_\nu(\h)$. Then $\sigma_h$ and $\phi_\eta(\nu)$ commute, and we can apply the results in \cite{Li96} to $g=\phi_\eta(\nu)\sigma_h$.

We shall be interested in the case that $L$ is unimodular. Then $V_L$ is holomorphic and there is a unique $g$-twisted $V_L$-module $V_L(g)$ for each $g\in\Aut(V_L)$ of finite order. Let $g=\phi_\eta(\nu)\sigma_h$ where $\nu\in\O(L)$ has order~$m$, $\phi_\eta(\nu)$ is a standard lift of~$\nu$, $\sigma_h=\e^{2\pi\i h_0}$ and $h\in\pi_\nu(\Lambda\otimes_\Z\Q)$. Suppose that $\nu$ has cycle shape $\prod_{t\mid m}t^{b_t}$, i.e.\ $\nu$ has characteristic polynomial $\prod_{t\mid m}(x^t-1)^{b_t}$ as automorphism of $\h$. Note that $\sum_{t\mid m}tb_t=\rk(L)=c$. Then the conformal weight of $V_L(g)$ is given by
\begin{equation*}
\rho(V_L(g))=\frac{1}{24}\sum_{t\mid m}b_t\left(t-\frac{1}{t}\right)+\min_{\alpha\in -h+\pi_{\nu}(L)}\frac{(\alpha,\alpha)}{2}\geq0.
\end{equation*}
The number
\begin{equation*}
\rho_\nu:=\frac{1}{24}\sum_{t\mid m}b_t\left(t-\frac{1}{t}\right)
\end{equation*}
is called the \emph{vacuum anomaly} of $V_L(g)$. It is positive for $\nu\neq\id$.


\subsection{Automorphisms of the Leech Lattice \VOA{}}
We now specialise the exposition to the Leech lattice.

The Leech lattice $\Lambda$ is the up to isomorphism unique positive-definite, even, unimodular lattice of rank $24$ without roots, i.e.\ vectors of squared norm $2$. The automorphism group $\O(\Lambda)$ of the Leech lattice is Conway's group $\Co_0$. Because $(V_\Lambda)_1=\hh=\{h(-1)\otimes\ee^0\,|\,h\in\Lambda_\C\}$, the inner automorphism group is given by
\begin{equation*}
K=\{\sigma_h\,|\,h\in\Lambda_\C\}
\end{equation*}
with $\sigma_h=\e^{2\pi\i h_0}$ and is abelian. Moreover, $\Aut(V_\Lambda)/K\cong\O(\Lambda)$, i.e.\ there is a short exact sequence
\begin{equation*}
1\to K\to\Aut(V_\Lambda)\to\O(\Lambda)\to1,
\end{equation*}
because $K\cap\O(\hat{\Lambda})=\Hom(\Lambda,\{\pm1\})$ in the special case of the Leech lattice. Hence, every automorphism of $V_\Lambda$ is of the form
\begin{equation*}
\phi_\eta(\nu)\sigma_h
\end{equation*}
for a lift $\phi_\eta(\nu)$ of some $\nu\in\O(\Lambda)$ and some $h\in\Lambda_\C$. The surjective homomorphism $\Aut(V_\Lambda)\to\O(\Lambda)$ in the short exact sequence is given by $\phi_\eta(\nu)\sigma_h\mapsto\nu$.

Recall that it suffices to take the lift $\phi_\eta(\nu)$ from a fixed section. Moreover, since $\sigma_h=\id$ if and only if $h\in\Lambda'=\Lambda$, we can take $h\in\Lambda_\C/\Lambda$.

\medskip

In the following we describe the conjugacy classes of $\Aut(V_\Lambda)$. As mentioned above, the automorphism $\phi_\eta(\nu)\sigma_h$ is conjugate to $\phi_\eta(\nu)\sigma_{\pi_\nu(h)}$ for any $h\in\Lambda_\C$, and $\phi_\eta(\nu)$ and $\sigma_{\pi_\nu(h)}$ commute.

Fix a section $\varphi\colon\O(L)\to\O(\hat{L})$, and for $\nu\in\O(\Lambda)$ denote by $\eta_\nu\colon\Lambda\to\{\pm1\}$ the corresponding function characterising the action of $\varphi(\nu)$ on $\C[\Lambda]_\eps$.
\begin{lem}\label{lem:conjsurj}
Any automorphism in $\Aut(V_\Lambda)$ is conjugate to an automorphism of the form $\varphi(\nu)\sigma_h$ where
\begin{enumerate}
\item\label{item:rep1} $\nu$ is in a fixed set $N$ of representatives for the conjugacy classes in $\O(\Lambda)$,
\item\label{item:rep2} $h$ is in a fixed set $H_\nu$ of representatives for the orbits of the action of $C_{\O(\Lambda)}(\nu)$ on $\pi_\nu(\Lambda_\C)/\pi_\nu(\Lambda)$.
\end{enumerate}
\end{lem}
\begin{proof}
Let $g\in\Aut(V_\Lambda)$. We have already argued that $g$ is conjugate to $\varphi(\nu)\sigma_h$ for some $\nu\in\O(\Lambda)$ and $h\in\pi_\nu(\Lambda_\C)/\pi_\nu(\Lambda)$. Now suppose that $\mu=\tau^{-1}\nu\tau$ for some $\tau,\mu\in\O(\Lambda)$. Let $\phi_\eta(\tau)$ be any lift of $\tau$. Then $\varphi(\mu)^{-1}\phi_\eta(\tau)^{-1}\varphi(\nu)\sigma_h\phi_\eta(\tau)$ acts as $\mu^{-1}\tau^{-1}\nu\tau=\id$ on $(V_\Lambda)_1\cong\Lambda_\C$. By Lemma~2.5 in \cite{DN99} this implies that $\varphi(\mu)^{-1}\phi_\eta(\tau)^{-1}\varphi(\nu)\sigma_h\phi_\eta(\tau)=\sigma_{h'}$ for some $h'\in\Lambda_\C$, i.e.\ $\varphi(\nu)\sigma_h$ is conjugate to $\varphi(\mu)\sigma_{h'}$. As explained above we may even assume that $h'\in\pi_\mu(\Lambda_\C)/\pi_\mu(\Lambda)$. This proves item~\eqref{item:rep1}.

To see item~\eqref{item:rep2} we consider $\varphi(\nu)\sigma_h$ with $h\in\pi_\nu(\Lambda_\C)$ and let $h'\in\pi_\nu(\Lambda_\C)$ with $h'+\pi_\nu(\Lambda)=\tau h+\pi_\nu(\Lambda)$ for some $\tau\in C_{\O(\Lambda)}(\nu)$ and prove that $\varphi(\nu)\sigma_h$ and $\varphi(\nu)\sigma_{h'}$ are conjugate. Indeed, $(\varphi(\nu)\sigma_{h'})^{-1}\varphi(\tau)\varphi(\nu)\sigma_h\varphi(\tau)^{-1}$ acts as identity on $(V_\Lambda)_1\cong\Lambda_\C$ and on $\ee^\alpha\in\C[\Lambda]_\eps$ as
\begin{align*}
&\sigma_{-h'}\varphi(\nu)^{-1}\varphi(\tau)\varphi(\nu)\sigma_h\varphi(\tau)^{-1}\ee^\alpha\\
&=\frac{\e^{-2\pi\i(h',\nu^{-1}\tau\nu\tau^{-1}\alpha)}\eta_\tau(\nu\tau^{-1}\alpha)\eta_\nu(\tau^{-1}\alpha)\e^{2\pi\i(h,\tau^{-1}\alpha)}}{\eta_\nu(\nu^{-1}\tau\nu\tau^{-1}\alpha)\eta_\tau(\tau^{-1}\alpha)}\ee^{\nu^{-1}\tau\nu\tau^{-1}\alpha}\\
&=\e^{2\pi\i(\tau h-h',\alpha)}\frac{\eta_\tau(\nu\tau^{-1}\alpha)\eta_\nu(\tau^{-1}\alpha)}{\eta_\nu(\alpha)\eta_\tau(\tau^{-1}\alpha)}\ee^\alpha\\
&=:\lambda(\alpha)\ee^\alpha.
\end{align*}
The map $\lambda\colon\Lambda\to\C^*$ is a homomorphism that is trivial on $\Lambda^\nu$. This means that $(\varphi(\nu)\sigma_{h'})^{-1}\varphi(\tau)\varphi(\nu)\sigma_h\varphi(\tau)^{-1}=\sigma_f$ for some $f\in\Lambda_\C$ with $(f,\alpha)\in\Z$ for all $\alpha\in\Lambda^\nu$. We just showed that $\varphi(\nu)\sigma_h$ and $\varphi(\nu)\sigma_{h'+f}$ are conjugate. The latter is conjugate to $\varphi(\nu)\sigma_{h'+\pi_\nu(f)}$, but $\pi_\nu(f)\in (\Lambda^\nu)'$ and $(\Lambda^\nu)'=\pi_\nu(\Lambda)$ because $\Lambda$ is unimodular. This proves the claim.
\end{proof}

The above lemma provides the surjective map $(\nu,h)\mapsto\varphi(\nu)\sigma_h$ from the set $Q:=\{(\nu,h)\,|\,\nu\in N,h\in H_\nu\}$ onto the conjugacy classes of $\Aut(V_\Lambda)$. We shall now prove that this map is in fact a bijection.
\begin{prop}\label{prop:conjclassbij}
The map $(\nu,h)\mapsto\varphi(\nu)\sigma_h$ is a bijection from the set $Q$ to the conjugacy classes of $\Aut(V_\Lambda)$.
\end{prop}
\begin{proof}
We saw in \autoref{lem:conjsurj} that this map is surjective. To see the injectivity let $(\nu,h)$ and $(\mu,f)$ be in the set $Q$ and assume that the corresponding automorphisms $\varphi(\nu)\sigma_h$ and $\varphi(\mu)\sigma_f$ are conjugate in $\Aut(V_\Lambda)$. We need to show that $(\nu,h)=(\mu,f)$.

Applying the surjection $\Aut(V_\Lambda)\to\O(\Lambda)$ immediately yields that $\nu$ and $\mu$ are conjugate and hence $\nu=\mu$.

Now assume $\varphi(\nu)\sigma_h$ and $\varphi(\nu)\sigma_f$ are conjugate under some automorphism $\varphi(\tau)\sigma_s$. Then $\tau\in C_{\O(\Lambda)}(\nu)$ and $(\varphi(\nu)\sigma_f)^{-1}(\varphi(\tau)\sigma_s)^{-1}\varphi(\nu)\sigma_h\varphi(\tau)\sigma_s=\id$. The latter acts on $\ee^\alpha\in\C[\Lambda]_\eps$ by multiplication with
\begin{equation*}
\frac{\eta_\tau(\alpha)\eta_\nu(\tau\alpha)}{\eta_\tau(\nu\alpha)\eta_\nu(\alpha)}\e^{2\pi\i(s-\nu^{-1}s+\tau^{-1}h-f,\alpha)}=1
\end{equation*}
for all $\alpha\in\Lambda$. Since $\nu$ and $\tau$ commute and $\eta_\nu$ is trivial on $\Lambda^\nu$, this implies that $\e^{2\pi\i(\tau^{-1}h-f,\alpha)}=1$ for all $\alpha\in\Lambda^\nu$. Hence $\tau^{-1}h-f\in(\Lambda^\nu)'=\pi_\nu(\Lambda)$. This means that $h+\pi_\nu(\Lambda)$ and $f+\pi_\nu(\Lambda)$ are in the same orbit under the action of $C_{\O(\Lambda)}(\nu)$ and hence identical.
\end{proof}
Recall that the automorphism $\varphi(\nu)\sigma_h$ in $\Aut(V_\Lambda)$ has finite order if and only if $h$ is in $\pi_\nu(\Lambda\otimes_\Z\Q)$.


\subsection{Orbifold Construction}\label{sec:orbifold}
In \cite{EMS20a} (see also \cite{Moe16}) we develop a cyclic orbifold theory for holomorphic \voa{}s. Below we give a short summary.

A \voa{} $V$ is called \emph{\strat{}} if it is rational (as defined, e.g., in \cite{DLM97}), $C_2$-cofinite (or lisse), self-contragredient (or self-dual) and of CFT-type. Then $V$ is also simple. Moreover, a rational \voa{} $V$ is said to be \emph{holomorphic} if $V$ itself is the only irreducible $V$-module. Note that the central charge $c$ of a \strathol{} \voa{} $V$ is necessarily a non-negative multiple of $8$, a simple consequence of Zhu's modular invariance result \cite{Zhu96}. Examples of \strat{} \voa{}s are the lattice \voa{}s discussed above.

\medskip

Let $V$ be a \strathol{} \voa{} and $G=\langle g\rangle$ a finite, cyclic group of automorphisms of $V$ of order~$n$.

By \cite{DLM00}, for each $i\in\Z_n$, there is an irreducible $g^i$-twisted $V$-module $V(g^i)$, which is unique up to isomorphism. The uniqueness of $V(g^i)$ implies that there is a representation $\phi_i\colon G\to\Aut_\C(V(g^i))$ of $G$ on the vector space $V(g^i)$ such that
\begin{equation*}
\phi_i(g)Y_{V(g^i)}(v,x)\phi_i(g)^{-1}=Y_{V(g^i)}(g v,x)
\end{equation*}
for all $v\in V$. This representation is unique up to an $n$-th root of unity. Let $W^{(i,j)}$ denote the eigenspace of $\phi_i(g)$ in $V(g^i)$ corresponding to the eigenvalue $e(j/n)$. On $V(g^0)=V$ we choose $\phi_0(g)=g$.

The \fpvosa{} $V^g=W^{(0,0)}$ is \strat{} by \cite{DM97,Miy15,CM16,McR21} and has exactly $n^2$ irreducible modules, namely the $W^{(i,j)}$, $i,j\in\Z_n$ \cite{MT04,DRX17}. We further show that the conformal weight $\rho(V(g))$ of $V(g)$ is in $(1/n^2)\Z$, and we define the type $t\in\Z_n$ of $g$ by $t=n^2\rho(V(g))\mmod{n}$.

Assume for simplicity that $g$ has type $t=0$, i.e.\ that $\rho(V(g))\in(1/n)\Z$. We may choose the representations $\phi_i$ such that the conformal weights satisfy
\begin{equation*}
\rho(W^{(i,j)})=\frac{ij}{n}\mmod{1}
\end{equation*}
and $V^g$ has fusion rules
\begin{equation*}
W^{(i,j)}\boxtimes W^{(l,k)}\cong W^{(i+l,j+k)}
\end{equation*}
for all $i,j,k,l\in\Z_n$ \cite{EMS20a}, i.e.\ the fusion algebra of $V^g$ is the group ring $\C[D]$ where $D$ is the \fqs{} with group structure $\Z_n\times\Z_n$ and quadratic form defined by the conformal weights modulo $1$. This is the \fqs{} of the hyperbolic lattice $\II_{1,1}(n)$, i.e.\ the lattice with Gram matrix $\left(\begin{smallmatrix}0&n\\n&0\end{smallmatrix}\right)$. In particular, all $V^g$-modules are simple currents.

In essence, the results in \cite{EMS20a} show that for a \strathol{} \voa{} $V$ and cyclic $G\cong\Z_n$ the module category of $V^G$ is the twisted group double $\mathcal{D}_\omega(G)$ where the $3$-cocycle $[\omega]\in H^3(G,\C^*)\cong\Z_n$ is determined by the type $t\in\Z_n$.
This proves a conjecture by Dijkgraaf, Vafa, Verlinde and Verlinde \cite{DVVV89} who stated it for arbitrary finite $G$.

If $24\mid c$, then the characters of the irreducible $V^g$-modules
\begin{equation*}
\ch_{W^{(i,j)}}(\tau)=\tr|_{W^{(i,j)}}q^{L_0-c/24},
\end{equation*}
$i,j\in\Z_n$, form a vector-valued modular form of weight~$0$ for the Weil representation $\rho_D$ of $\Gamma=\SLZ$ on $\C[D]$. Since $D$ has level $n$, the individual characters $\ch_{W^{(i,j)}}(\tau)$ are modular forms for the congruence subgroup $\Gamma(n)$. Moreover, $\ch_{W^{(0,0)}}(\tau)=\ch_{V^g}(\tau)$ is modular for $\Gamma_0(n)$. Similar results hold if $24\nmid c$ or if $t\neq0$. We shall explain these notions in more detail in \autoref{sec:mf}.

In general, we say that a simple \voa{} $V$ satisfies the \emph{positivity condition} if $\rho(W)>0$ for any irreducible $V$-module $W\not\cong V$ and $\rho(V)=0$.

Now, if $V^g$ satisfies the positivity condition, then the direct sum of $V^g$-modules
\begin{equation*}
V^{\orb(g)}:=\bigoplus_{i\in\Z_n}W^{(i,0)}
\end{equation*}
carries the structure of a \strathol{} \voa{} of the same central charge as $V$ and is called the \emph{orbifold construction} of $V$ associated with $g$ \cite{EMS20a}. The \voa{} structure of $V^{\orb(g)}$ as extension of $V^g$ is unique up to isomorphism. It depends only up to algebraic conjugacy on $g$, i.e.\ if $\langle g\rangle$ and $\langle h\rangle$ for some $h\in\Aut(V)$ are conjugate in $\Aut(V)$, then $V^{\orb(g)}\cong V^{\orb(h)}$. Note that the sum $\bigoplus_{j\in\Z_n}W^{(0,j)}$ is just the old \voa{} $V$.

\medskip

There is also an \emph{inverse} (or \emph{reverse}) orbifold construction. Given a
\voa{} $V^{\orb(g)}$ obtained by the orbifold construction, we define an automorphism $\amgis$ of $V^{\orb(g)}$ of order~$n$ via $\amgis v:=e(i/n)v$ for $v\in W^{(i,0)}$, and the unique irreducible $\amgis^j$-twisted $V^{\orb(g)}$-module is given by $V^{\orb(g)}(\amgis^j)=\bigoplus_{i\in\Z_n}W^{(i,j)}$, $j\in\Z_n$ (see \cite{Moe16}, Theorem~4.9.6). Then
\begin{equation*}
(V^{\orb(g)})^{\orb(\amgis)}=\bigoplus_{j\in\Z_n}W^{(0,j)}=V,
\end{equation*}
i.e.\ orbifolding with $\amgis$ is inverse to orbifolding with $g$.

If two elements $g,h\in\Aut(V)$ of type $0$ and satisfying the positivity condition are (algebraically) conjugate, the inverse-orbifold automorphisms corresponding to $g$ and $h$ are (algebraically) conjugate under some isomorphism $\varphi\colon V^{\orb(g)}\to V^{\orb(h)}$.


\subsection{Orbifold Characters}\label{sec:charlator}
We recall a formula for the character of the orbifold construction from \cite{GK19b}, which we shall use in the proof of Theorem \ref{thm:main3}.

We continue in the setting of the previous section, i.e.\ we let $V$ be a \strathol{} \voa{} and $g$ an automorphism of $V$ of order~$n$ such that $V^g$ satisfies the positivity condition. Suppose that $g$ has type~$0$ and $24\mid c$. Then the character
\begin{equation*}
\ch_{V^{\orb(g)}}(\tau)=\tr|_{V^{\orb(g)}}q^{L_0-c/24} 
\end{equation*}
is a modular form for $\Gamma=\SLZ$ of weight $0$, holomorphic on $H$ with a possible pole at the cusp $\infty$. For $k\mid n$ define
\begin{equation*}
D_k(\tau)=\sum_{\substack{j\in\Z_n\\(j,n)=k}}T(0,j)(\tau)
\end{equation*}
with $T(0,j)(\tau)=\tr|_Vg^jq^{L_0-c/24}$. Then $D_k$ is a modular form for $\Gamma_0(n/k)$ and
\begin{equation*}
\ch_{V^{\orb(g)}}(\tau)=\frac{1}{n}\sum_{k\mid n}\sum_{M\in\Gamma_0(n/k)\backslash\Gamma}\!\!D_k|_M(\tau)
\end{equation*}
where $D_k|_M(\tau)=D_k(M\tau)$ (this notation is explained in \autoref{sec:mf}).

\medskip

Now, we specialise to the case where $V=V_L$ is the lattice \voa{} associated with a positive-definite, even, unimodular lattice $L$ such that $24\mid\rk(L)$. Let $g=\phi_\eta(\nu)\sigma_h$ where $\phi_\eta(\nu)$ is a standard lift of $\nu\in\O(L)$ and $h\in L_\C$. The eta product associated with $\nu$ is defined as
\begin{equation*}
\eta_{\nu}(\tau)=\prod_{t\mid m}\eta(t\tau)^{b_t} 
\end{equation*}
if $\nu$ has order $m$ and cycle shape $\prod_{t\mid m}t^{b_t}$. Define 
\begin{equation*}
w_k\colon L^{\nu^k}\to\C^*,\quad\alpha\mapsto\e^{2\pi\i k(\pi_\nu(h),\pi_\nu(\alpha))}\begin{cases}(-1)^{\langle\alpha,\nu^{k/2}\alpha\rangle}&\text{if $m,k\in2\Z$,}\\
1&\text{else.}\end{cases}
\end{equation*}
Then $w_k$ is a homomorphism and the kernel $M_{k,d}$ of $w_k^d$ is a finite-index sublattice of $L^{\nu^k}$. A slight generalisation of the argument given in \cite{GK19b} shows that
\begin{equation*}
D_k(\tau)=\frac{1}{\eta_{\nu^k}(\tau)}\sum_{d\mid n/k}\frac{n}{kd}\mu(d)\theta_{M_{k,d}}(\tau)
\end{equation*}
where $\theta_{M_{k,d}}$ is the theta series of $M_{k,d}$. (Note that there is a typo in Theorem~3.9 of \cite{GK19b} and its proof.)


\section{Modular Forms}\label{sec:mf}
In this section we describe some results on modular forms of weight~$2$. In particular, we construct an Eisenstein series of weight $2$ for the dual Weil representation of $\II_{1,1}(n)$.


\subsection{Modular Forms of Weight 2 for \texorpdfstring{$\Gamma_0(n)$}{Gamma0(n)}}\label{sec:gamma}
We construct a special Eisenstein series of weight $2$ for $\Gamma_0(n)$ and recall the Deligne bound. For background on modular forms we refer to \cite{DS05,Miy06}.

The group $\Gamma=\SLZ$ acts on the upper half-plane $\HH$ by Möbius transformations. Given a positive integer $n$ the subgroup $\Gamma_0(n)=\left\{\left(\begin{smallmatrix}a&b\\c&d\end{smallmatrix}\right)\in\Gamma\,\middle|\,c=0\mmod{n}\right\}$ has index $\psi(n)$ in $\Gamma$ where $\psi(n)=n\prod_{p|n}(1+1/p)$ denotes the Dedekind $\psi$-function.

The quotient $X_0(n)=\Gamma_0(n)\backslash\HH^*$ where $\HH^*=\HH\cup P$ is the upper half-plane complemented by the cusps $P=\Q\cup\{\infty\}$ is a compact Riemann surface. The class of $a/c\in\Q$ with $(a,c)=1$ in $\Gamma_0(n)\backslash P$ is characterised by the invariants $(c,n)$, a divisor of $n$, and $ac/(c,n)\in\Z_{(c,n/(c,n))}^*$.

It follows that $\Gamma_0(n)$ has $\eps=\sum_{c\mid n}\phi((c,n/c))$ classes of cusps. The cusp $a/c\in\Q$ with $(a,c)=1$ has width $t_{a/c}:=n/(c^2,n)$. Note that $\sum_{s\in\Gamma_0(n) \backslash P}t_s=\psi(n)$. A complete system of representatives for the cusps $\Gamma_0(n)\backslash P$ is given by the rational numbers $a/c$ where $c$ ranges over the positive divisors of $n$ and $a$ over a complete set of representatives of $\Z_{(c,n/c)}^*$ such that $(a,c)=1$.

We denote by $\mathcal{M}_k(\Gamma_0(n))$ the space of holomorphic modular forms for $\Gamma_0(n)$ of weight $k\in\Ns$ and trivial character. These are holomorphic functions on $\HH$ which are invariant under modular transformations, i.e.\ $f|_M=f$ for all $M\in\Gamma_0(n)$ where $f|_M(\tau):=(c\tau+d)^{-k}f(M\tau)$ for $M=\left(\begin{smallmatrix}a&b\\c&d\end{smallmatrix}\right)\in\operatorname{SL}_2(\R)$, and which are holomorphic at the cusps.

Let $f\in\mathcal{M}_k(\Gamma_0(n))$ and $s\in\Gamma_0(n)\backslash P$ a cusp represented by $a/c$ with $(a,c)=1$. Choose $M_s=\left(\begin{smallmatrix}a&b\\c&d\end{smallmatrix}\right)\in\Gamma$. Then $M_s\infty=s$. The Fourier expansion of $f|_{M_s}$ is called an expansion of $f$ at $s$. It is an expansion in integral powers of $q^{1/t_s}=e(\tau/t_s)$. In general, it depends on the choice of $M_s$. However, its coefficients at integral powers of $q$ do not.

\medskip

We now specialise to weight~$2$. The Eisenstein series
\begin{equation*}
E_2(\tau)=1-24\sum_{m=1}^\infty\sigma(m)q^m
\end{equation*}
transforms under $M=\left(\begin{smallmatrix}a&b\\c&d\end{smallmatrix}\right)\in\Gamma$ as
\begin{equation*}
E_2|_M(\tau)=E_2(\tau)+\frac{12}{2\pi\i}\frac{c}{c\tau+d}.
\end{equation*}
For $k\in\Ns$ we define $F_k:=\frac{1}{\sqrt{k}}\left(\begin{smallmatrix}k&0\\0&1\end{smallmatrix}\right)$ and
\begin{equation*}
E_2^{(k)}(\tau):=E_2(k\tau)=\frac{1}{k} E_2|_{F_k}(\tau).
\end{equation*}
We describe the transformation behaviour of $E_2^{(k)}$ under $\Gamma$:
\begin{prop}\label{prop:rescaledtrans}
Let $M=\left(\begin{smallmatrix}a&b\\c&d\end{smallmatrix}\right)\in\Gamma$ and $k\in\Ns$. Define
\begin{equation*}
u=(k,c)>0,\quad w=\frac{k}{(k,c)}>0
\end{equation*}
and choose $v\in\Z$ such that
\begin{equation*}
v=\left(\frac{c}{(k,c)}\right)^{-1}\!d\mmod{\frac{k}{(k,c)}}
\end{equation*}
where $(c/(k,c))^{-1}$ denotes the inverse of $c/(k,c)$ modulo $k/(k,c)$. Then
\begin{equation*}
E_2^{(k)}|_M(\tau)=\frac{1}{w^2}E_2\left(\frac{u\tau+v}{w}\right)+\frac{12}{2\pi\i}\frac{c}{k(c\tau+d)}.
\end{equation*}
\end{prop}
\begin{proof}
Let $N:=\frac{1}{\sqrt{uw}}\left(\begin{smallmatrix}u&v\\0&w\end{smallmatrix}\right)$. Then
\begin{equation*}
F_kMN^{-1}=\begin{pmatrix}aw&bu-av\\c/u&(du-cv)/k\end{pmatrix}\in\Gamma
\end{equation*}
and
\begin{align*}
E_2^{(k)}|_M(\tau)
&=\frac{1}{k}E_2|_{F_kM}(\tau)=\frac{1}{k}E_2|_{F_kMN^{-1}N}(\tau)\\
&=\frac{1}{k}\frac{u}{w}\left( E_2(N\tau)+\frac{12}{2\pi\i}\frac{c}{u(c\tau+d)/w}\right)\\
&=\frac{1}{w^2}E_2(N\tau)+\frac{12}{2\pi\i}\frac{c}{k(c\tau+d)}.\qedhere
\end{align*}
\end{proof}

The Fourier expansion of the function $E_2((u\tau+v)/w)$ in the proposition is
\begin{equation*}
E_2\left(\frac{u\tau+v}{w}\right)=1-24\sum_{m=1}^\infty\sigma(m)e(mv/w)q^{mu/w},
\end{equation*}
which is an expansion in powers of $q^{k/(c^2,k)}$ because $u/w=\frac{u/(u,w)}{w/(u,w)}=\frac{u/(u,w)}{k/(c^2,k)}$.

The space of modular forms of weight~$2$ for $\Gamma_0(n)$ decomposes as
\begin{equation*}
\mathcal{M}_2(\Gamma_0(n))=\mathcal{S}_2(\Gamma_0(n))\oplus\mathcal{E}_2(\Gamma_0(n))
\end{equation*}
where $\mathcal{S}_2(\Gamma_0(n))$ denotes the space of cusp forms and $\mathcal{E}_2(\Gamma_0(n))$ the Eisenstein space, the orthogonal complement of $\mathcal{S}_2(\Gamma_0(n))$ with respect to the Petersson inner product. The dimensions of the spaces are given by
\begin{align*}
\dim(\mathcal{S}_2(\Gamma_0(n)))&=g(X_0(n)),\\
\dim(\mathcal{E}_2(\Gamma_0(n)))&=\eps-1
\end{align*}
where $g(X_0(n))$ is the genus of the modular curve $X_0(n)=\Gamma_0(n)\backslash\HH^*$ and $\eps$ the number of cusps of $\Gamma_0(n)$. An explicit basis of $\mathcal{E}_2(\Gamma_0(n))$ is given, for example, in \cite{DS05}, Theorem~4.6.2.

Let $n>1$ and $k\mid n$, $k>1$. Then
\begin{equation*}
f_k(\tau):=E_2(\tau)-kE_2^{(k)}(\tau)=E_2(\tau)-kE_2(k\tau)
\end{equation*}
is in $\mathcal{E}_2(\Gamma_0(n))$. If $n$ is square-free, the functions $f_k$ already form a basis of $\mathcal{E}_2(\Gamma_0(n))$.
The expansion of $E_2^{(k)}(\tau)$ at a cusp $s=a/c$ is given by
\begin{equation*}
E_2^{(k)}|_{M_s}(\tau)
=\frac{(k,c)^2}{k^2}\left(1-24\sum_{m=1}^\infty\sigma(m)e(mv/w)q^{m(k,c)^2/k}\right)+\ldots
\end{equation*}
where we omitted the non-modular term, which cancels out in $f_k$. This is an expansion in $q^{1/t_s}$ where $t_s=n/(c^2,n)$ is the width of the cusp $s$ since $k/(c^2,k)\mid t_s$. It follows that the constant term of $f_k$ at the cusp $s=a/c$ is
\begin{equation*}
[f_k|_{M_s}](0)=1-\frac{(k,c)^2}{k}.
\end{equation*}

We construct an Eisenstein series with certain constant terms at the cusps:
\begin{prop}\label{prop:specialeis}
Suppose $n>1$. Then there is a unique $f\in\mathcal{E}_2(\Gamma_0(n))$ satisfying
\begin{equation*}
[f|_{M_s}](0)=1
\end{equation*}
for all cusps $s\neq\infty$. This function can be written as
\begin{equation*}
f(\tau)=\frac{1}{\phi(n)}\sum_{\substack{k\mid n\\k>1}}k\mu(n/k)f_k(\tau)=\frac{1}{\phi(n)}\sum_{\substack{k\mid n\\k>1}}k\mu(n/k)\left(E_2(\tau)-kE_2(k\tau)\right)
\end{equation*}
and the constant coefficient of the Fourier expansion of $f$ at $\infty$ is given by
\begin{equation*}
[f](0)=1-\psi(n).
\end{equation*}
\end{prop}
\begin{proof}
Let $S$ be the set of cusps of $\Gamma_0(n)$ without $\infty$. Then $\mathcal{E}_2(\Gamma_0(n))$ has a basis consisting of Eisenstein series $g_s$, $s\in S$, which are $1$ at $s$ and vanish at the other cusps in $S$. This implies the first statement. Now let
\begin{equation*}
f(\tau)=\frac{1}{\phi(n)}\sum_{\substack{k\mid n\\k>1}}k\mu(n/k)f_k(\tau)
\end{equation*}
and $s=a/c$ with $(c,n)<n$. Then
\begin{align*}
[f|_{M_s}](0)
&=\frac{1}{\phi(n)}\sum_{\substack{k\mid n\\k>1}}k\mu(n/k)[f_k|{M_s}](0)=\frac{1}{\phi(n)}\sum_{\substack{k\mid n\\k>1}}k\mu(n/k)\left(1-\frac{(k,c)^2}{k}\right)\\
&=\frac{1}{\phi(n)}\sum_{k\mid n}k\mu(n/k)-\frac{1}{\phi(n)}\sum_{k\mid n}\mu(n/k)(k,c)^2\\
&=1-0=1.
\end{align*}
This proves the second statement. The third follows from the residue theorem.
\end{proof}

The Eisenstein series $f$ in the previous proposition can also be written as
\begin{equation*}
f(\tau)=E_2(\tau)-\frac{1}{\phi(n)}\sum_{k\mid n}k^2\mu(n/k)E_2(k\tau)
\end{equation*}
and its Fourier expansion at a cusp $s=a/c$ is
\begin{align*}
f|_{M_s}(\tau)&=1-24\sum_{m=1}^\infty\sigma(m)q^m\\
&\quad-\frac{1}{\phi(n)}\sum_{k\mid n}\mu(n/k)(k,c)^2\left(1-24\sum_{m=1}^\infty\sigma(m)e(mv/w)q^{m(k,c)^2/k}\right)
\end{align*}
with $M_s=\left(\begin{smallmatrix}a&b\\c&d\end{smallmatrix}\right)\in\Gamma$ and $v,w$ as in \autoref{prop:rescaledtrans}.

\medskip

We now consider certain cusp forms. Let $p$ be a prime and $W_p:=\frac{1}{\sqrt{p}}\left(\begin{smallmatrix}0&-1\\p&0\end{smallmatrix}\right)$. Then the Fricke involution
\begin{equation*}
W_p\colon\mathcal{S}_2(\Gamma_0(p))\to\mathcal{S}_2(\Gamma_0(p)),\quad f\mapsto f|_{W_p}
\end{equation*}
is diagonalisable with eigenvalues $\pm1$. The corresponding eigenspace decomposition
\begin{equation*}
\mathcal{S}_2(\Gamma_0(p))=\mathcal{S}_2(\Gamma_0(p))^+\oplus\mathcal{S}_2(\Gamma_0(p))^-
\end{equation*}
is orthogonal with respect to the Petersson inner product. Let $\Gamma_0(p)^+:=\Gamma_0(p)\cup W_p\Gamma_0(p)$ and $X_0(p)^+:=\Gamma_0(p)^+\backslash\HH^*$.

\begin{prop}\label{prop:cuspminus}
Let $p$ be a prime such that $X_0(p)$ has genus $g(X_0(p))>0$. Then $\mathcal{S}_2^-(\Gamma_0(p))$ is non-trivial.
\end{prop}
\begin{proof}
By a classical result of Fricke (see \cite{Fri11}, p. 366)
\begin{equation*}
g(X_0(p)^+)=\frac{1}{2}(g(X_0(p))+1)-\eps_ph(p)
\end{equation*}
for all primes $p\geq5$ where $h(p)$ denotes the class number of $\Q(\sqrt{-p})$ and
\begin{equation*}
\eps_p=
\begin{cases}
\frac{1}{4}&\text{if }p=1,5\mmod{8},\\
\frac{1}{2}&\text{if }p=7\mmod{8},\\
1          &\text{if }p=3\mmod{8}.
\end{cases}
\end{equation*}
Hence
\begin{align*}
\dim(\mathcal{S}_2(\Gamma_0(p))^-)
&=\dim(\mathcal{S}_2(\Gamma_0(p)))-\dim(\mathcal{S}_2(\Gamma_0(p))^+)\\
&=g(X_0(p))-g(X_0(p)^+)\\
&=\frac{1}{2}(g(X_0(p))-1)+\eps_ph(p)>0
\end{align*}
if $g(X_0(p))>0$.
\end{proof}

The space $\mathcal{S}_2(\Gamma_0(p))^-$ has a basis consisting of simultaneous eigenforms for all Hecke operators $T_m$, $m\in\Ns$ (see \cite{AL70} and Theorem~9.27 in \cite{Kna92}).

We state \emph{Deligne's bound} on the growth of the coefficients of cusp forms:
\begin{thm}\label{thm:deligne}
Let $p$ be a prime such that $g(X_0(p))>0$. Let $f\in\mathcal{S}_2(\Gamma_0(p))^-$ be a normalised eigenform, i.e.\ $f(\tau)=\sum_{m=1}^{\infty}a(m)q^m$ with $a(1)=1$ and $T_mf=a(m)f$ for all $m\in\Ns$. Then
\begin{enumerate}
\item $a(p^k)=1$ for all $k\in\Ns$,
\item $a(2),\ldots,a(p-1)\in\R$,
\item $|a(m)|\leq\sigma_0(m)\sqrt{m}$ for all $m\in\Ns$.
\end{enumerate}
\end{thm}
\begin{proof}
By Theorem~9.27 in \cite{Kna92},
$a(p)=1$ and $a(p^k)=a(p)^k=1$ for all $k\in\Ns$.
Since the Hecke operators $T_m$ for $(m,p)=1$ are self-adjoint with respect to the Petersson inner product, the coefficients $a(2),\ldots,a(p-1)$ are real (see Theorem 9.18 in \cite{Kna92}).
Finally the $L$-function of $f$ is of the form
\begin{equation*}
L(f,s)=\frac{1}{1-a(p)p^{-s}}\prod_{\substack{t\text{ prime}\\t\neq p}}\frac{1}{1-a(t)t^{-s}+t^{1-2s}}.
\end{equation*}
Deligne has shown (see \cite{Del74}, Théorème~8.2) that the reciprocal roots of the polynomial $1-a(t)X+tX^2$ have absolute value $\sqrt{t}$, i.e.\ that
\begin{equation*}
1-a(t)X-tX^2=(1-\alpha_t X)(1-\bar{\alpha}_t X)
\end{equation*}
with $|\alpha_t|=\sqrt{t}$. Writing $(1-a(t)t^{-s}+t^{1-2s})^{-1}=(1-\alpha_tt^{-s})^{-1}(1-\bar{\alpha}_tt^{-s})^{-1}$ and expanding each factor in a geometric series we find that the coefficient of $t^{-ms}$ is given by $\alpha_t^m+\alpha_t^{m-1}\bar{\alpha}_t+\ldots+\bar{\alpha}_t^m$. This means that if the positive integer $m$ with $(m,p)=1$ has prime factorisation $m=t_1^{e_1}\ldots t_r^{e_r}$, then
\begin{equation*}
a(m)=\prod_{j=1}^r(\alpha_{t_j}^{e_j}+\alpha_{t_j}^{e_j-1}\bar{\alpha}_{t_j}+\ldots+\bar{\alpha}_{t_j}^{e_j}),
\end{equation*}
implying
\begin{equation*}|a(m)|\leq\prod_{j=1}^r(e_j+1)t_j^{e_j/2}=\sigma_0(m)\sqrt{m}.
\end{equation*}
The general case follows from the multiplicativity of the $a(m)$ and the first statement.
\end{proof}


\subsection{Modular Forms for the Weil Representation}\label{sec:weil}
In this section we recall some properties of modular forms for the Weil representation of $\Gamma$ from \cite{Sch09}.

Let $D$ be a \fqs{} with quadratic form $q\colon D\to\Q/\Z$ and associated bilinear form $(\cdot,\cdot)\colon D\times D\to\Q/\Z$ (see, e.g., \cite{Nik80,CS99}). The level of $D$ is the smallest positive integer $n$ such that $nq(\gamma)=0\mmod{1}$ for all $\gamma\in D$.

Let $c$ be an integer. Then $c$ acts by multiplication on $D$ and there is an exact sequence $0\to D_c\to D\to D^c\to0$ where $D_c$ is the kernel and $D^c$ the image of this map. The group $D^c$ is the orthogonal complement of $D_c$.

Let $D^{c*}$ be the set of elements $\alpha\in D$ satisfying $cq(\gamma)+(\alpha,\gamma)=0\mmod{1}$ for all $\gamma\in D_c$. Then $D^{c*}$ is a coset of $D^c$. Once we have chosen a Jordan decomposition of $D$, there is a canonical coset representative $x_c$ of $D^{c*}$ satisfying $2x_c=0$. Write $\alpha\in D^{c*}$ as $\alpha=x_c+c\gamma$. Then $q_c(\alpha):=cq(\gamma)+(x_c,\gamma)\mmod{1}$ is independent of the choice of $\gamma$. This gives a well-defined map $q_c\colon D^{c*}\to\Q/\Z$.

If $D$ has no odd $2$-adic Jordan components, then $x_c=0$ so that $D^c=D^{c*}$ and $q_c(\gamma)=cq(\mu)\mmod{1}$ for $\gamma=c\mu\in D^c$.

Suppose $D$ has even signature. We define a scalar product $(\cdot,\cdot)$ on the group ring $\C[D]$ which is linear in the first and antilinear in the second variable by $(\ee^\gamma,\ee^\beta):=\delta_{\gamma,\beta}$. Then there is a unitary action $\rho_D$ of the group $\Gamma$ on $\C[D]$ satisfying
\begin{align*}
\rho_D(T)\ee^\gamma&=e(q(\gamma))\ee^{\gamma},\\
\rho_D(S)\ee^\gamma&=\frac{e(\sign(D)/8)}{\sqrt{|D|}}\sum_{\beta\in D}e(-(\gamma,\beta))\ee^\beta
\end{align*}
where $S=\left(\begin{smallmatrix}0&-1\\1&0\end{smallmatrix}\right)$ and $T=\left(\begin{smallmatrix}1&1\\0&1\end{smallmatrix}\right)$ are the standard generators of $\Gamma$. The representation $\rho_D$ is called the \emph{Weil representation} of $\Gamma$ on $\C[D]$. It commutes with $\O(D)$. We also consider the \emph{dual Weil representation} $\bar\rho_D$, the complex conjugate of $\rho_D$.

Suppose that the level of $D$ divides $n$. We define a quadratic Dirichlet character $\chi_D\colon\Gamma_0(n)\to\{\pm1\}$ by
\begin{equation*}
\chi_D(M):=\left(\frac{a}{|D|}\right)e\big((a-1)\oddity(D)/8\big)
\end{equation*}
for $M=\left(\begin{smallmatrix}a&b\\c&d\end{smallmatrix}\right)\in\Gamma_0(n)$. Then $M\in\Gamma_0(n)$ acts in Weil representation $\rho_D$ as
\begin{equation*}
\rho_D(M)\ee^\gamma=\chi_D(M)e(bdq(\gamma))\ee^{d\gamma}.
\end{equation*}
A general formula for $M\in\Gamma$ is given in \cite{Sch09}, Theorem~4.7.

\medskip

Let
\begin{equation*}
F(\tau)=\sum_{\gamma\in D}F_{\gamma}(\tau)\ee^\gamma
\end{equation*}
be a holomorphic function on the upper half-plane $\HH$ with values in $\C[D]$ and $k\in\Z$. Then $F$ is a modular form for $\rho_D$ of weight~$k$ if
\begin{equation*}
F(M\tau)=(c\tau+d)^k\rho_D(M)F(\tau)
\end{equation*}
for all $M=\left(\begin{smallmatrix}a&b\\c&d\end{smallmatrix}\right)$ in $\Gamma$ and $F$ is meromorphic at the cusp $\infty$. We call $F$ \emph{holomorphic} if it is also holomorphic at $\infty$.

Classical examples of modular forms transforming under the Weil representation are theta series. Let $L$ be a positive-definite, even lattice of even rank $2k$ with bilinear form $(\cdot,\cdot)\colon L\times L\to\Z$ and let $D=L'/L$ be the associated \fqs{}. Define
\begin{equation*}
\theta=\sum_{\gamma\in D}\theta_{\gamma}\ee^\gamma\quad\text{with}\quad\theta_\gamma(\tau):=\sum_{\alpha\in\gamma+L}q^{(\alpha,\alpha)/2}
\end{equation*}
for all $\gamma\in D$. Then $\theta$ is a modular form for the Weil representation $\rho_D$ of weight~$k$ which is holomorphic at $\infty$.

Another example arises from orbifold theory (see \autoref{sec:orbifold}). Let $V$ be a \strathol{} \voa{} of central charge $c$ and $g$ an automorphism of $V$ of finite order~$n$ and type~$0$ such that $V^g$ satisfies the positivity condition. Suppose $24\mid c$. Then the irreducible $V^g$-modules are given by $W^\gamma$, $\gamma\in D$, with conformal weights $\rho(W^\gamma)=q(\gamma)\mmod{1}$ where the \fqs{} $D=\Z_n\times\Z_n$ has quadratic form $q((i,j))=ij/n\mmod{1}$. Note that $D$ has order~$n^2$ and level~$n$. The vector-valued character $\Ch_{V^g}\colon\HH\to\C[D]$ defined by
\begin{equation*}
\Ch_{V^g}(\tau)=\sum_{\gamma\in D}\ch_{W^\gamma}(\tau)\ee^\gamma\quad\text{with}\quad\ch_{W^\gamma}(\tau)=\tr|_{W^\gamma}q^{L_0-c/24}
\end{equation*}
is a modular form of weight~$0$ for the Weil representation $\rho_D$ of $\Gamma$, holomorphic on $\HH$ with a possible pole at the cusp $\infty$.

\medskip

Let $D$ be a \fqs{} of even signature and $n$ a positive integer such that the level of $D$ divides $n$. Let $F=\sum_{\gamma\in D}F_{\gamma}\ee^\gamma$ be a modular form for $\rho_D$. Then the formula for $\rho_D$ shows that $F_0$ is a modular form for $\Gamma_0(n)$ of character $\chi_D$. Conversely:
\begin{prop}\label{prop:lift1}
Let $D$ be a \fqs{} of even signature and level dividing~$n$. Let $f$ be a scalar-valued modular form for $\Gamma_0(n)$ of weight~$k$ and character $\chi_D$. Then
\begin{equation*}
F=\sum_{M\in\Gamma_0(n)\backslash\Gamma}f|_M\rho_D(M^{-1})\ee^0
\end{equation*}
is a modular form for $\rho_D$ of weight~$k$ which is invariant under $\O(D)$.
\end{prop}
This is Theorem~6.2 in \cite{Sch06} in the special case of trivial support. An explicit description of the components of the lift $F$ is given in \cite{Sch09}, Theorem~5.7 (and for $n$ square-free in \cite{Sch06}, Theorem~6.5). An analogous statement holds if we replace $\rho_D$ by the dual Weil representation $\bar\rho_D$.

The dimension formulae proved in this text are instances of the following pairing argument:
\begin{prop}\label{prop:pairing}
Let $D$ be a \fqs{} of even signature and $G$ a modular form for $\rho_D$ of weight $2-k$, $k\in\Z$. Let $F$ be a modular form for $\bar\rho_D$ of weight~$k$. Then the constant coefficient of $(G,\bar{F})=\sum_{\gamma\in D}G_\gamma F_{\gamma}$ vanishes.
\end{prop}
\begin{proof}
The unitarity of $\rho_D$ with respect to $(\cdot,\cdot)$ implies that $(G,\bar{F})$ is a scalar-valued modular form for $\Gamma$ of weight~$2$ which is holomorphic on $\HH$ and meromorphic at $\infty$. Hence $(G,\bar{F})d\tau$ defines a $1$-form on the Riemann sphere with a pole at $\infty$. By the residue theorem its residue has to vanish. Since $d\tau=\frac{1}{2\pi\i}\frac{dq}{q}$, the constant term in the Fourier expansion of $(G,\bar{F})$ must be $0$.
\end{proof}


\subsection{Vector-Valued Eisenstein Series} \label{sec:vector}
In this section we construct an Eisenstein series $F$ for the dual Weil representation of $\II_{1,1}(n)$ by lifting the Eisenstein series $f\in\mathcal{E}_2(\Gamma_0(n))$ from \autoref{prop:specialeis} and calculate some of its Fourier coefficients. Pairing $F$ with the character of $V^g$ will give the first dimension formula.

Let $n$ be a positive integer. Consider the \fqs{} $D=(\Z_n\times\Z_n,q)$ with quadratic form $q((i,j))=ij/n\mmod{1}$. This is the \fqs{} of the hyperbolic lattice $\II_{1,1}(n)$ and of the fusion algebra of the orbifold \voa{} $V^g$ described above. We note that the character $\chi_D$ is trivial because $D$ has no odd $2$-adic components and its order is $n^2$.

Let $f$ be a modular form for $\Gamma_0(n)$ of weight~$2$ with trivial character. Then \autoref{prop:lift1} allows us to lift $f$ to a vector-valued modular form $F$ for the dual Weil representation $\bar\rho_D$. We describe the components of $F$. The expansion of $f$ at a cusp $s\in\Gamma_0(n)\backslash P$ is an expansion in $q^{1/t_s}$ where $t_s$ denotes the width of $s$. We decompose
\begin{equation*}
f|_{M_s}(\tau)=g_{t_s,0}(\tau)+\ldots+g_{t_s,t_s-1}(\tau)
\end{equation*}
with $g_{t_s,j}|_T=e(j/t_s)g_{t_s,j}$ for $j\in\Z_{t_s}$. In general, the terms in this decomposition depend on the choice of $M_s$ but $g_{t_s,0}$ does not. The special case of Theorem~5.7 in \cite{Sch09} for trivial support applied to $D=(\Z_n\times\Z_n,q)$ gives:
\begin{prop}\label{prop:lift2}
Let $f\in\mathcal{M}_2(\Gamma_0(n))$ and $F$ be the lift (from \autoref{prop:lift1}) of $f$ to a vector-valued modular form for the dual Weil representation $\bar\rho_D$ where $D=(\Z_n\times\Z_n,q)$. Then
\begin{equation*}
F_\gamma(\tau)=\sum_{s\in\Gamma_0(n)\backslash P}\delta_{\gamma\in D^c}
\frac{1}{((c,n),n/(c,n))}e(dq_c(\gamma))g_{t_s,j_\gamma}(\tau)
\end{equation*}
for all $\gamma\in D$ where $M_s=\left(\begin{smallmatrix}a&b\\c&d\end{smallmatrix}\right)\in\Gamma$ is such that $M_s\infty=s$ and $j_\gamma\in\{0,\ldots,t_s-1\}$ such that $j_\gamma/t_s=-q(\gamma)\mmod{1}$.
\end{prop}
Recall that since $D$ has no odd $2$-adic components, $q_c(\gamma)=cq(\mu)\mmod{1}$ for any $\mu\in D$ such that $\gamma=c\mu$.

\medskip

For the remainder of this section, let $F=\sum_{\gamma\in D}F_{\gamma}\ee^\gamma$ be the lift of the Eisenstein series $f\in\mathcal{E}_2(\Gamma_0(n))$ from \autoref{prop:specialeis} to a modular form for $\bar\rho_D$. We calculate the lowest-order Fourier coefficients of the components $F_\gamma$ and $[F_0](1)$.

For $k\in\Ns$ and $m\in\Z_k$ we define the \emph{Ramanujan sum}
\begin{equation*}
C_k(m):=\sum_{a\in\Z_k^*}e(am/k)=\sum_{c\mid(k,m)}c\mu(k/c).
\end{equation*}
\begin{prop}\label{prop:Fconst}
Let $\gamma\in D$ with $q(\gamma)=0\mmod{1}$. Then the constant term in the Fourier expansion of $F_\gamma$ is
\begin{equation*}
[F_\gamma](0)=-\delta_{\gamma,0}\psi(n)+\sum_{c\mid n}\delta_{\gamma\in D^c}\frac{1}{(c,n/c)}C_{(c,n/c)}((c,n/c)q_c(\gamma)).
\end{equation*}
\end{prop}
\begin{proof}
Let $\gamma\in D$ with $q(\gamma)=0\mmod{1}$. Then \autoref{prop:lift2} implies that
\begin{align*}
[F_\gamma](0)&=\sum_{s\in\Gamma_0(n)\backslash P}\delta_{\gamma\in D^c}\frac{1}{((c,n),n/(c,n))}e(dq_c(\gamma))[g_{t_s,0}](0)\\
&=\sum_{s\in\Gamma_0(n)\backslash P}\delta_{\gamma\in D^c}\frac{1}{((c,n),n/(c,n))}e(dq_c(\gamma))[f|_{M_s}](0)\\
&=\sum_{s\in\Gamma_0(n)\backslash P}\delta_{\gamma\in D^c}\frac{1}{((c,n),n/(c,n))}e(dq_c(\gamma))-\delta_{\gamma,0}\psi(n)
\end{align*}
where for each cusp $s\in\Gamma_0(n)\backslash P$ we have chosen a matrix $M_s=\left(\begin{smallmatrix}a&b\\c&d\end{smallmatrix}\right)\in\Gamma$ such that $M_s\infty=s$. The sum can be partially carried out to remove the complex factor $e(dq_c(\gamma))$. Indeed, choosing the system of representatives for the cusps described in \autoref{sec:gamma} we obtain
\begin{align*}
[F_\gamma](0)&=\sum_{c\mid n}\delta_{\gamma\in D^c}\frac{1}{(c,n/c)}\sum_{a\in\Z_{(c,n/c)}^*}\!\!e(dq_c(\gamma))-\delta_{\gamma,0}\psi(n)\\
&=\sum_{c\mid n}\delta_{\gamma\in D^c}\frac{1}{(c,n/c)}\sum_{a\in\Z_{(c,n/c)}^*}\!\!e(aq_c(\gamma))-\delta_{\gamma,0}\psi(n)
\end{align*}
since $q_c(\gamma)=0\mmod{1/c}$ and $ad=1\mmod{c}$.
Note that $(c,n/c)q_c(\gamma)=0\mmod{1}$ if $\gamma\in D^c$ and $q(\gamma)=0\mmod{1}$. Hence
\begin{align*}
[F_\gamma](0)
&=\sum_{c\mid n}\delta_{\gamma\in D^c}\frac{1}{(c,n/c)}\sum_{a\in\Z_{(c,n/c)}^*}\!\!e\left(\frac{a(c,n/c)q_c(\gamma)}{(c,n/c)}\right)-\delta_{\gamma,0}\psi(n)\\
&=\sum_{c\mid n}\delta_{\gamma\in D^c}\frac{1}{(c,n/c)}C_{(c,n/c)}((c,n/c)q_c(\gamma))-\delta_{\gamma,0}\psi(n).\qedhere
\end{align*}
\end{proof}

For $c\mid n$ we define the isotropic subgroup $H_c$ of $D$ by
\begin{equation*}
H_c:=\{(ci,nj/c)\,|\,i,j\in\Z_n\}\cong\Z_{n/c}\times\Z_c.
\end{equation*}
Then we can rewrite the Fourier coefficients $[F_\gamma](0)$ for isotropic $\gamma \in D$ as follows:

\begin{prop}\label{prop:Fconst2}
Let $\gamma\in D$ with $q(\gamma)=0\mmod{1}$. Then
\begin{equation*}
[F_\gamma](0)=-\delta_{\gamma,0}\psi(n)+\sum_{c\mid n}\delta_{\gamma\in H_c}\frac{\phi((c,n/c))}{(c,n/c)}.
\end{equation*}
\end{prop}

Recall that we extended the sum-of-divisors function $\sigma$ to the positive rational numbers by $\sigma(r)=0$ if $r\notin\Ns$.

\begin{prop}\label{prop:constterm}
The coefficient at $q$ in the Fourier expansion of $F_0$ is given by
\begin{equation*}
[F_0](1)=24-24\sum_{c\mid n}\frac{\phi((c,n/c))}{(c,n/c)}.
\end{equation*}
\end{prop}
\begin{proof}
\autoref{prop:lift2} implies that
\begin{equation*}
F_0(\tau)=\sum_
{s\in\Gamma_0(n)\backslash P}\frac{1}{((c,n),n/(c,n))}g_{t_s,0}(\tau).
\end{equation*}
From the Fourier expansion of $f$ at $s$ calculated above we obtain
\begin{align*}
[g_{t_s,0}](1)&=[f|_{M_s}](1)\\
&=-24+\frac{24}{\phi(n)}\sum_{k\mid n}\mu(n/k)(k,c)^2\sigma\left(\frac{k}{(k,c)^2}\right)e\left(\frac{v}{(k,c)}\right)
\end{align*}
with $v\in\Z$ such that $v=(c/(k,c))^{-1}d\mmod{k/(k,c)}$ and the inverse taken modulo $k/(k,c)$ (see \autoref{prop:rescaledtrans}).
It suffices to show that
\begin{equation*}
\sum_
{s\in\Gamma_0(n)\backslash P}\frac{1}{((c,n),n/(c,n))}\sum_{k\mid n}\mu(n/k)(k,c)^2\sigma\left(\frac{k}{(k,c)^2}\right)e\left(\frac{v}{(k,c)}\right)=\phi(n).
\end{equation*}

As in the proof of \autoref{prop:Fconst}, we partially carry out the sum over the cusps to remove the complex factors. Since $(c/(k,c))^{-1}/(k,c)=0\mmod{1/c}$ and $ad=1\mmod{c}$, we obtain
\begin{align*}
&\sum_{c\mid n}\frac{1}{(c,n/c)}\sum_{k\mid n}\mu(n/k)(k,c)^2\sigma\left(\frac{k}{(k,c)^2}\right)\!\!\sum_{a\in\Z_{(c,n/c)}^*}\!\!\!\!e\left(\frac{d(c/(k,c))^{-1}}{(k,c)}\right)\\
&=\sum_{c\mid n}\frac{1}{(c,n/c)}\sum_{k\mid n}\mu(n/k)(k,c)^2\sigma\left(\frac{k}{(k,c)^2}\right)\!\!\sum_{a\in\Z_{(c,n/c)}^*}\!\!\!\!e\left(\frac{a(c/(k,c))^{-1}}{(k,c)}\right).
\end{align*}
Now $c\mid n$, $k\mid n$, $n/k$ square-free and $(k,c)^2\mid k$ imply $(k,c)=(c,n/c)$.
Then the last sum becomes
\begin{align*}
\sum_{a\in\Z_{(c,n/c)}^*}\!\!\!\!e\left(\frac{a(c/(k,c))^{-1}}{(c,n/c)}\right)=\sum_{a\in\Z_{(c,n/c)}^*}\!\!\!\!e\left(\frac{a}{(c,n/c)}\right)=C_{(c,n/c)}(1)=\mu((c,n/c))
\end{align*}
because $(c,n/c)$ and $(c/(k,c))^{-1}$ are coprime. Overall, the expression of interest simplifies to
\begin{align*}
\sum_{c\mid n}\frac{1}{(c,n/c)}\sum_{k\mid n}\mu(n/k)(k,c)^2\sigma\left(\frac{k}{(k,c)^2}\right)\mu((c,n/c)),
\end{align*}
which is multiplicative in $n$, as is $\phi(n)$. We verify that both expressions coincide for prime powers, which concludes the proof.
\end{proof}

For $\gamma\in D$ with $q(\gamma)\neq0\mmod{1}$ we introduce $r_\gamma\in(0,1)$ such that $r_\gamma=-q(\gamma)\mmod{1}$.
\begin{prop}\label{prop:Rcoeff1}
Let $\gamma\in D$ with $q(\gamma)\neq0\mmod{1}$. Then
\begin{align*}
[F_\gamma](r_\gamma)&=\frac{24}{\phi(n)}\sum_{c\mid n}\delta_{\gamma\in D^c}\frac{1}{(c,n/c)}\sum_{k\mid n}\mu(n/k)(k,c)^2\times\\
&\quad\times\sigma\left(r_\gamma\frac{k}{(k,c)^2}\right)C_{(c,n/c)}\left((c,n/c)\left(q_c(\gamma)+r_\gamma\frac{(c/(k,c))^{-1}}{(k,c)}\right)\right)
\end{align*}
where the inverse is taken modulo $k/(k,c)$.
\end{prop}
\begin{proof}
By \autoref{prop:lift2}
\begin{equation*}
[F_\gamma](r_\gamma)=\sum_{s\in\Gamma_0(n)\backslash P}\delta_{\gamma\in D^c}\frac{1}{((c,n),n/(c,n))}e(dq_c(\gamma))[f|_{M_s}](r_\gamma)
\end{equation*}
with $M_s=\left(\begin{smallmatrix}a&b\\c&d\end{smallmatrix}\right)\in\Gamma$ such that $M_s\infty=s$. The formula for the Fourier expansion of $f$ at $s$ gives
\begin{equation*}
[f|_{M_s}](r_\gamma)=\frac{24}{\phi(n)}\sum_{k\mid n}\mu(n/k)(k,c)^2\sigma\left(r_\gamma\frac{k}{(k,c)^2}\right)e\left(r_\gamma\frac{v}{(k,c)}\right)
\end{equation*}
where $v=d(c/(k,c))^{-1}$ with inverse taken modulo $k/(k,c)$. Hence,
\begin{align*}
[F_\gamma](r_\gamma)&=\frac{24}{\phi(n)}\!\!\sum_{s\in\Gamma_0(n)\backslash P}\!\!\!\!\delta_{\gamma\in D^c}\frac{1}{((c,n),n/(c,n))}\sum_{k\mid n}\mu(n/k)(k,c)^2\times\\
&\quad\times\sigma\left(r_\gamma\frac{k}{(k,c)^2}\right)e\left(d\left(q_c(\gamma)+r_\gamma\frac{(c/(k,c))^{-1}}{(k,c)}\right)\right).
\end{align*}
Again, we take the standard representatives of the cusps and note that $q_c(\gamma)+r_\gamma(c/(k,c))^{-1}/(k,c)=0\mmod{1/c}$ and $ad=1\mmod{c}$. Then
\begin{align*}
[F_\gamma](r_\gamma)&=\frac{24}{\phi(n)}\sum_{c\mid n}\delta_{\gamma\in D^c}\frac{1}{(c,n/c)}\sum_{k\mid n}\mu(n/k)(k,c)^2\times\\
&\quad\times\sigma\left(r_\gamma\frac{k}{(k,c)^2}\right)\!\!\!\sum_{a\in\Z_{(c,n/c)}^*}\!\!\!e(ay_{n,\gamma}(c,k)/(c,n/c))
\end{align*}
with
\begin{equation*}
y_{n,\gamma}(c,k):=(c,n/c)\left(q_c(\gamma)+r_\gamma\frac{(c/(k,c))^{-1}}{(k,c)}\right)
\end{equation*}
for $c\mid n$, $k\mid n$, $\gamma\in D^c$ and $r_\gamma k/(k,c)^2\in\Z$. One can show that $y_{n,\gamma}(c,k)=0\mmod{1}$.
Hence,
\begin{equation*}
\sum_{a\in\Z_{(c,n/c)}^*}\!\!\!\!e(ay_{n,\gamma}(c,k)/(c,n/c))=C_{(c,n/c)}(y_{n,\gamma}(c,k)),
\end{equation*}
which implies the assertion.
\end{proof}

We shall see in \autoref{prop:pos1} that the coefficients $[F_\gamma](r_\gamma)$ are positive.


\section{Dimension Formulae}\label{sec:dimform}
In this section we prove a formula for the dimension $\dim(V_1^{\orb(g)})$ of the weight-$1$ space of the orbifold construction in central charge $24$, depending on the dimensions of the fixed-point Lie subalgebras $\dim(V_1^{g^m})$, $m\mid n$, where $n$ is the order of $g$, and terms of lower $L_0$-weight.

This generalises previous results for $n=2,3$ in \cite{Mon94} (see also \cite{LS19}), for $n=5,7,13$ in \cite{Moe16} and for all $n>1$ such that the modular curve $X_0(n)$ has genus~$0$ in \cite{EMS20b}. However, our approach is fundamentally different.


\subsection{First Dimension Formula}
The first dimension formula will be proved using the pairing argument described in \autoref{prop:pairing}.

Let $V$ be a \strathol{} \voa{} $V$ of central charge $24$ and $g\in\Aut(V)$ of finite order~$n$ and type~$0$. Suppose that the fixed-point subalgebra $V^g$ satisfies the positivity condition. Then $V^g$ has fusion algebra $\C[D]$ with \fqs{} $D=(\Z_n\times\Z_n,q)$ where $q((i,j))=ij/n\mmod{1}$. In particular, the irreducible $V^g$-modules are given by $W^\gamma$, $\gamma\in D$, with conformal weights $\rho(W^\gamma)=q(\gamma)\mmod{1}$ (see \autoref{sec:orbifold}). The character
\begin{equation*}
\Ch_{V^g}(\tau)=\sum_{\gamma\in D}\ch_{W^\gamma}(\tau)\ee^\gamma
\end{equation*}
is a vector-valued modular form of weight~$0$ for the Weil representation $\rho_D$ of $\Gamma$, holomorphic on $\HH$ and with a pole at the cusp $\infty$ (see \autoref{sec:weil}).
The lowest-order terms in the Fourier expansion of $\Ch_{V^g}$ are
\begin{equation*}
\ch_{W^\gamma}(\tau)=
\begin{cases}q^{-1}+\dim(V_1^g)+O(q)&\text{if }\gamma=0,\\
\dim(W_1^\gamma)+O(q)&\text{if }\gamma\neq0\text{ and }q(\gamma)=0\mmod{1},\\
\dim(W_{1-r_\gamma}^\gamma)q^{-r_\gamma}+O(q^{1-r_\gamma})&\text{if }\gamma\neq0\text{ and }q(\gamma)\neq0\mmod{1}
\end{cases}
\end{equation*}
because $W^0=W^{(0,0)}=V^g$ is of CFT type and $V^g$ satisfies the positivity condition. Recall that $r_\gamma\in(0,1)$ is chosen such that $r_\gamma=-q(\gamma)\mmod{1}$.

Pairing $\Ch_{V^g}$ with a holomorphic vector-valued modular form $F$ of weight~$2$ for the dual Weil representation $\bar\rho_D$ (see \autoref{prop:pairing}) we obtain:
\begin{prop}\label{prop:pairing2}
Let $V$ be a \strathol{} \voa{} of central charge $24$ and $g\in\Aut(V)$ of finite order~$n$ and type~$0$ such that $V^g$ satisfies the positivity condition. Suppose $F$ is a holomorphic modular form of weight~$2$ for the dual Weil representation $\bar\rho_D$. Then
\begin{equation*}
[F_0](1)+\sum_{\substack{\gamma\in D\\q(\gamma)=0\bmod{1}}}[F_\gamma](0)\dim(W_1^\gamma)+\sum_{\substack{\gamma\in D\\q(\gamma)\neq0\bmod{1}}}[F_\gamma](r_\gamma)\dim(W_{1-r_\gamma}^\gamma)=0.
\end{equation*}
\end{prop}
We emphasise that any modular form $F\neq0$ imposes a restriction on the character $\Ch_{V^g}$. Our first dimension formula is obtained by choosing for $F$ the lift of the Eisenstein series from \autoref{prop:specialeis},
\begin{equation*}
f(\tau)=E_2(\tau)-\frac{1}{\phi(n)}\sum_{k\mid n}k^2\mu(n/k)E_2(k\tau).
\end{equation*}

\begin{thm}[First Dimension Formula]\label{thm:dimform1}
Let $V$ be a \strathol{} \voa{} of central charge $24$ and $g\in\Aut(V)$ of finite order $n>1$ and type~$0$ such that $V^g$ satisfies the positivity condition. Then
\begin{equation*}
\psi(n)\dim(V_1^g)-\sum_{c\mid n}\frac{\phi((c,n/c))}{(c,n/c)}\dim(V_1^{\orb(g^c)})=24-24\sum_{c\mid n}\frac{\phi((c,n/c))}{(c,n/c)}+\tilde{R}(g)
\end{equation*}
with rest term
\begin{equation*}
\tilde{R}(g)=\sum_{\substack{\gamma\in D\\q(\gamma)\neq0\bmod{1}}}\tilde{d}_n(\gamma)\dim(W_{1-r_\gamma}^\gamma)
\end{equation*}
where $\tilde{d}_n(\gamma)$ equals the expression $[F_\gamma](r_\gamma)$ from \autoref{prop:Rcoeff1}.
\end{thm}
\begin{proof}
Let $F$ be the lift of $f$. \autoref{prop:Fconst2} and $V^{\orb(g^c)}=\bigoplus_{\gamma\in H_c}W^\gamma$ imply
\begin{equation*}
\sum_{\substack{\gamma\in D\\q(\gamma)=0\bmod{1}}}[F_\gamma](0)\dim(W_1^{\gamma})=\sum_{c\mid n}\frac{\phi((c,n/c))}{(c,n/c)}\dim(V_1^{\orb(g^c)})-\psi(n)\dim(V_1^g).
\end{equation*}
Finally, \autoref{prop:pairing2} together with \autoref{prop:constterm} gives
\begin{equation*}
\psi(n)\dim(V_1^g)-\sum_{c\mid n}\frac{\phi((c,n/c))}{(c,n/c)}\dim(V_1^{\orb(g^c)})=24-24\sum_{c\mid n}\frac{\phi((c,n/c))}{(c,n/c)}+\tilde{R}(g)
\end{equation*}
with $\tilde{R}(g)$ as stated.
\end{proof}
The formula extends to $n=1$ if we define $\tilde{R}(g)=0$ in this case.

The rest term $\tilde{R}(g)$ contains exactly the terms involving the dimensions of the weight spaces of the irreducible $V^g$-modules of weight less than $1$. We shall show that $\tilde{R}(g)\geq0$ (see \autoref{prop:pos1}).

Specialising to $n=2,3,4,5,6,7,8,9,10,12,13,16,18,25$, i.e.\ those $n>1$ for which the modular curve $X_0(n)$ has genus~$0$, the result becomes Theorem~3.1 in \cite{EMS20b}, which was proved using expansions of Hauptmoduln.

In the case that the order of $g$ is a prime $p$ such that $X_0(p)$ has positive genus we shall see that $\tilde{R}(g)\geq24$. This follows from pairing $\Ch_{V^g}$ with the lift of a Hecke eigenform in $\mathcal{S}_2(\Gamma_0(n))^-$.


\subsection{Second Dimension Formula}\label{sec:dimform2}
Linearly combining the dimension formula in \autoref{thm:dimform1} applied to the powers $g^{n/m}$ of $g$ with $m\mid n$ we obtain the second dimension formula, a closed formula for $\dim(V^{\orb(g)}_1)$. This will be the main result of this section.

We introduce the arithmetic function $\lambda$ by
\begin{equation*}
\lambda(n):=\prod_{\substack{p\mid n\\p\text{ prime}}}(-p)
\end{equation*}
for $n\in\Ns$ and define
\begin{align*}
\xi_n(m)&:=\frac{\lambda(n/m)}{n/m}\frac{\phi((m,n/m))}{(m,n/m)},\\
c_n(m)&:=\xi_n(n/m)\psi(n/m)=\frac{\lambda(m)}{m}\frac{\phi((m,n/m))}{(m,n/m)}\psi(n/m)
\end{align*}
for $n\in\Ns$ and $m\mid n$ (see also \autoref{sec:averageverystrange}). Note that for a given $n\in\Ns$, the numbers $\xi_n(m)$ and $c_n(m)$ for $m\mid n$ can be equivalently specified by the systems of linear equations
\begin{equation*}
\sum_{t\mid m\mid n}\xi_n(m)\frac{\phi((t,m/t))}{(t,m/t)}=\delta_{t,n}
\end{equation*}
(so that in particular $\sum_{m\mid n}\xi_n(m)=\delta_{1,n}$) and
\begin{equation*}
\sum_{m\mid n}c_n(m)(t,m)=n/t
\end{equation*}
for all $t\mid n$.

Taking the linear combination of the dimension formulae in \autoref{thm:dimform1} for $g^{n/m}$ with coefficients $\xi_n(m)$ for $m\mid n$ yields:
\begin{thm}[Second Dimension Formula]\label{thm:dimform2}
Let $V$ be a \strathol{} \voa{} of central charge $24$ and $g\in\Aut(V)$ of order $n>1$ and type~$0$ such that $V^g$ satisfies the positivity condition. Then
\begin{equation*}
\dim(V_1^{\orb(g)})=24+\sum_{m\mid n}c_n(m)\dim(V_1^{g^m})-R(g)
\end{equation*}
with rest term
\begin{equation*}
R(g)=\sum_{m\mid n}\xi_n(n/m)\tilde{R}(g^m)=\!\!\!\!\sum_{\substack{\gamma\in D\\q(\gamma)\neq0\bmod{1}}}\!\!\!\!d_n(\gamma)\dim(W_{1-r_\gamma}^\gamma)
\end{equation*}
where for $\gamma=(i,j)\in D$ such that $q(\gamma)\neq0\mmod{1}$
\begin{equation*}
d_n(\gamma)=\sum_{m\mid (i,n)}\xi_n(n/m)\tilde{d}_{n/m}(i/m,j).
\end{equation*}
\end{thm}
\begin{proof}
The dimension formula in \autoref{thm:dimform1} applied to $g^{n/m}$ of order~$m$ for $m\mid n$ reads
\begin{align*}
&\psi(m)\dim(V_1^{g^{n/m}})-\sum_{c\mid m}\frac{\phi((c,m/c))}{(c,m/c)}\dim(V_1^{\orb(g^{nc/m})})\\
&=24-24\sum_{c\mid m}\frac{\phi((c,m/c))}{(c,m/c)}+\tilde{R}(g^{n/m}).
\end{align*}
Summing this equation over $m\mid n$ with coefficients $\xi_n(m)$ yields
\begin{align*}
&\sum_{m\mid n}c_n(n/m)\dim(V_1^{g^{n/m}})-\sum_{m\mid n}\xi_n(m)\sum_{c\mid m}\frac{\phi((c,m/c))}{(c,m/c)}\dim(V_1^{\orb(g^{nc/m})})\\
&=24\delta_{1,n}-24\sum_{m\mid n}\xi_n(m)\sum_{c\mid m}\frac{\phi((c,m/c))}{(c,m/c)}+R(g).
\end{align*}
Replacing $c$ by $m/c$ and swapping the order of the two summations we obtain
\begin{align*}
&\sum_{m\mid n}\xi_n(m)\sum_{c\mid m}\frac{\phi((c,m/c))}{(c,m/c)}\dim(V_1^{\orb(g^{nc/m})})\\
&=\sum_{m\mid n}\xi_n(m)\sum_{c\mid m}\frac{\phi((c,m/c))}{(c,m/c)}\dim(V_1^{\orb(g^{n/c})})\\
&=\sum_{c\mid n}\dim(V_1^{\orb(g^{n/c})})\sum_{c\mid m\mid n}\xi_n(m)\frac{\phi((c,m/c))}{(c,m/c)}\\
&=\sum_{c\mid n}\dim(V_1^{\orb(g^{n/c})})\delta_{c,n}=\dim(V_1^{\orb(g)}).
\end{align*}
Similarly,
\begin{equation*}
\sum_{m\mid n}\xi_n(m)\sum_{c\mid m}\frac{\phi((c,m/c))}{(c,m/c)}=1.
\end{equation*}
This proves
\begin{equation*}
\dim(V_1^{\orb(g)})=24-24\delta_{1,n}+\sum_{m\mid n}c_n(m)\dim(V_1^{g^m})-R(g).
\end{equation*}
Recall that
\begin{equation*}
\tilde{R}(g)=\sum_{\substack{i,j\in\Z_n\\ij\neq0\bmod{n}}}\tilde{d}_n(i,j)\dim_g(i,j)
\end{equation*}
with $\dim_g(i,j):=\dim(W_{1-r_{(i,j)}}^{(i,j)})$. To compute $\tilde{R}(g^m)$ for $m\mid n$ we note that
\begin{equation*}
\dim_{g^m}(i,j)=\sum_{l\in\Z_m}\dim_g(im,j+ln/m).
\end{equation*}
Hence
\begin{align*}
R(g)&=\sum_{m\mid n}\xi_n(n/m)\tilde{R}(g^m)=\sum_{m\mid n}\xi_n(n/m)\!\!\sum_{\substack{i,j\in\Z_{n/m}\\ij\neq0\bmod{n/m}}}\!\!\tilde{d}_{n/m}(i,j)\dim_{g^m}(i,j)\\
&=\sum_{m\mid n}\xi_n(n/m)\!\!\sum_{\substack{i,j\in\Z_{n/m}\\ij\neq0\bmod{n/m}}}\!\!\tilde{d}_{n/m}(i,j)\sum_{l\in\Z_m}\dim_g(im,j+ln/m),
\end{align*}
which may be rewritten as
\begin{align*}
R(g)&=\sum_{\substack{i,j\in\Z_n\\ij\neq0\bmod{n}}}\!\!\Big(\sum_{m\mid (i,n)}\xi_n(n/m)\tilde{d}_{n/m}(i/m,j)\Big)\dim_g(i,j)\\
&=\sum_{\substack{i,j\in\Z_n\\ij\neq0\bmod{n}}}d_n(i,j)\dim_g(i,j)
\end{align*}
with
\begin{equation*}
d_n(i,j):=\sum_{m\mid (i,n)}\xi_n(n/m)\tilde{d}_{n/m}(i/m,j)
\end{equation*}
for all $i,j\in\Z_n$ with $ij\neq0\mmod{n}$. This proves the theorem.
\end{proof}
Again, the rest term $R(g)$ contains exactly the terms involving the dimensions of the weight spaces of the irreducible $V^g$-modules of weight less than $1$. We shall prove in \autoref{prop:pos2} that $R(g)\geq0$. The theorem generalises Corollary~4.3 in \cite{EMS20b}.

We remark that if the order of $g$ is a prime~$p$, then the first and second dimension formulae both reduce to
\begin{equation*}
\dim(V_1)+\dim(V_1^{\orb(g)})=24+(p+1)\dim(V_1^g)-R(g)
\end{equation*}
with rest term
\begin{equation*}
R(g)=\tilde{R}(g)=\frac{24}{p-1}\sum_{k=1}^{p-1}\sigma(p-k)\!\!\sum_{\substack{\gamma\in D\\q(\gamma)=k/p\bmod{1}}}\!\!\dim(W_{k/p}^\gamma).
\end{equation*}

A similar formula was derived in \cite{BM21} in the special case of $p=11,17,19$, the primes with $g(X_0(p))=1$ using the theory of mock modular forms.


\subsection{Dimension Bounds}
In this section we show that the rest terms in \autoref{thm:dimform1} and \autoref{thm:dimform2} are non-negative. This implies that there is an upper bound for the dimension of the weight-$1$ subspace of the orbifold construction. Based on this observation we define the notion of an extremal orbifold.

\begin{prop}\label{prop:pos1}
In the rest term $\tilde{R}(g)$ in \autoref{thm:dimform1} the coefficients $\tilde{d}_n(\gamma)$ for all $\gamma\in D$ with $q(\gamma)\neq0\mmod{1}$ satisfy
\begin{equation*}
\tilde{d}_n(\gamma)=[F_\gamma](r_\gamma)\in\frac{24}{\phi(n)}
\Ns.
\end{equation*}
This implies
\begin{equation*}
\tilde{R}(g)\geq0.
\end{equation*}
\end{prop}
\begin{proof}
Since the argument is tedious, we only sketch it. We write
\begin{equation*}
\tilde{d}_n(\gamma)=[F_\gamma](r_\gamma)=\frac{24}{\phi(n)}\sigma(r_{\gamma,[n]})\sum_{c\mid n}\tilde{f}_{n,\gamma}(c)
\end{equation*}
with
\begin{equation*}
\tilde{f}_{n,\gamma}(c):=\delta_{\gamma\in D^c}\frac{1}{(c,n/c)}\sum_{k\mid n}\mu(n/k)(k,c)^2\sigma\left(r_{\gamma,n}\frac{k}{(k,c)^2}\right)C_{(c,n/c)}(y_{n,\gamma}(c,k))
\end{equation*}
where we factored the rational number $r_{\gamma}=r_{\gamma,n}r_{\gamma,[n]}$ such that $r_{\gamma,n}$ contains all the prime divisors of $n$, i.e.\ $r_{\gamma,n}=\prod_{p\mid n}p^{\nu_p(r_{\gamma})}$. Note that $r_{\gamma,[n]}\in\Ns$ because $nr_{\gamma}\in\Ns$. Then the functions $\tilde{f}_{n,\gamma}$ are multiplicative after normalisation. Writing $\sum_{c\mid n}\tilde{f}_{n,\gamma}(c)$ as a product over the primes dividing $n$ we see that this sum is a positive integer.
\end{proof}

For the rest term in \autoref{thm:dimform2} we show a similar statement. The results differ slightly, insofar as the coefficients in the rest term $R(g)$ are only non-negative rather than positive.

We denote by $i_n$ for $i\in\Z_n$ the representative of $i$ in $\{0,1,\ldots,n-1\}$.
\begin{prop}\label{prop:pos2}
In the rest term $R(g)$ in \autoref{thm:dimform2} the coefficients $d_n(\gamma)$ for all $\gamma=(i,j)\in D$ with $q(\gamma)\neq0\mmod{1}$ satisfy
\begin{equation*}
d_n(\gamma)\in\frac{24}{\phi(n)}
\N.
\end{equation*}
This implies
\begin{equation*}
R(g)\geq0.
\end{equation*}
Moreover, $d_n(i,j)=0$ if and only if $(i,j,n)\nmid\lceil i_nj_n/n\rceil$.
\end{prop}
\begin{proof}
The proof is analogous to that of \autoref{prop:pos1}. Here we write
\begin{equation*}
d_n(\gamma)=\frac{24}{\phi(n)}\sigma(r_{\gamma,[n]})\sum_{c\mid n}f_{n,\gamma}(c)
\end{equation*}
with
\begin{align*}
f_{n,\gamma}(c)
&:=\delta_{\gamma\in D^c}\sum_{m\mid (i,n)/c}\frac{\lambda(m)}{m}\frac{\phi((m,n/m))}{(m,n/m)}\frac{\phi(n)}{\phi(n/m)}\frac{1}{(c,(n/m)/c)}\times\\
&\quad\times\sum_{k\mid n/m}\mu((n/m)/k)(k,c)^2\sigma\left(r_{\gamma,n}\frac{k}{(k,c)^2}\right)\times\\
&\quad\times C_{(c,(n/m)/c)}\left((c,(n/m)/c)\left(q_c(\gamma)+r_\gamma\frac{(c/(k,c))^{-1}}{(k,c)}\right)\right)
\end{align*}
where the inverse is taken modulo $k/(k,c)$. Then $f_{n,\gamma}$ is multiplicative in $c$ after normalisation. Writing $\sum_{c\mid n}f_{n,\gamma}(c)$ as a product over the prime divisors of $n$ we can deduce the assertions.
\end{proof}
Note that, while the coefficients $\tilde{d}_n(i,j)$ in \autoref{thm:dimform1} are symmetric in $i,j\in\Z_n$, the $d_n(i,j)$ in general are not. However, the condition that $d_n(i,j)$ vanishes is symmetric.

We now state two important corollaries of the above propositions.
\begin{cor}[First Dimension Bound]\label{cor:dimbound1}
Let $V$ be a \strathol{} \voa{} of central charge $24$ and $g\in\Aut(V)$ of finite order $n>1$ and type~$0$ such that $V^g$ satisfies the positivity condition. Then
\begin{equation*}
\psi(n)\dim(V_1^g) -\sum_{c\mid n}\frac{\phi((c,n/c))}{(c,n/c)}\dim(V_1^{\orb(g^c)})
\geq24-24\sum_{c\mid n}\frac{\phi((c,n/c))}{(c,n/c)}.
\end{equation*}
\end{cor}
The next corollary provides an upper bound for the dimension of the weight-$1$ Lie algebra $V_1^{\orb(g)}$ purely in terms of the restriction of the automorphism $g$ to the Lie algebra $V_1$.
\begin{cor}[Second Dimension Bound]\label{cor:dimbound2}
Let $V$ be a \strathol{} \voa{} of central charge $24$ and $g\in\Aut(V)$ of finite order $n>1$ and type~$0$ such that $V^g$ satisfies the positivity condition. Then
\begin{equation*}
\dim(V_1^{\orb(g)})\leq24+\sum_{m\mid n}c_n(m)\dim(V_1^{g^m}).
\end{equation*}
\end{cor}

We introduce the notion of extremal automorphisms that attain the upper bound for $\dim(V_1^{\orb(g)})$.
\begin{defi}[Extremality]\label{def:extremal}
Let $V$ be a \strathol{} \voa{} of central charge $24$ and $g\in\Aut(V)$ of finite order $n>1$. Suppose that $g$ has type~$0$ and that $V^g$ satisfies the positivity condition. Then $g$ is called \emph{extremal} if
\begin{equation*}
\dim(V_1^{\orb(g)})=24+\sum_{m\mid n}c_n(m)\dim(V_1^{g^m}),
\end{equation*}
i.e.\ if $R(g)=0$. By convention, we also call the identity extremal.
\end{defi}

It is in general not easy to identify extremal orbifolds besides explicitly computing $\dim(V^{\orb(g)}_1)$, either via the rest term $R(g)$ in the dimension formula or directly (as explained in, e.g., Section~5.6 of \cite{Moe16}). Both approaches are in general computationally demanding. However, we can formulate some easily verifiable necessary and sufficient conditions for extremality:
\begin{prop}\label{prop:necsufff}
Let $V$ be a \strathol{} \voa{} of central charge $24$ and $g\in\Aut(V)$ of order $n>1$ and type~$0$ such that $V^g$ satisfies the positivity condition. Then:
\begin{enumerate}
\item\label{item:crit1} If the conformal weight $\rho(V(g^i))\geq1$ for all $i\in\Z_n\setminus\{0\}$, then $g$ is extremal.
\item\label{item:crit2} If $g$ is extremal, then $\rho(V(g^i))\geq1$ for all $i\in\Z_n$ with $(i,n)=1$.
\item\label{item:crit3} If $n$ is prime, then $g$ is extremal if and only if $\rho(V(g^i))\geq1$ for all $i\in\Z_n\setminus\{0\}$.
\end{enumerate}
\end{prop}
\begin{proof}
Item~\eqref{item:crit1} is immediate since the condition implies that $\dim(W_{1-r_\gamma}^\gamma)=0$ for all $\gamma\in D$ with $q(\gamma)\neq0\mmod{1}$.

Suppose that $i\in\Z_n$ such that $(i,n)=1$. Then \autoref{thm:dimform2} shows that for all $j\in\Z_n\setminus\{0\}$, $d_n(i,j)=\tilde{d}_n(i,j)$, which is positive by \autoref{prop:pos1}. Then, $R(g)$ can only vanish if $\dim(W_{1-r_{(i,j)}}^{(i,j)})=0$ for all $j\in\Z_n\setminus\{0\}$, which implies that $\rho(V(g^i))\geq1$. This proves item~\eqref{item:crit2}.

Item~\eqref{item:crit3} follows from items \eqref{item:crit1} and \eqref{item:crit2}.
\end{proof}

Extremal automorphism cannot exist for all orders $n$. This follows from the existence of certain Hecke eigenforms and the Deligne bound.
\begin{thm}\label{thm:delignebound}
Let $V$ be a \strathol{} \voa{} of central charge $24$ and $g\in\Aut(V)$ of prime order~$p$ and type~$0$ such that $V^g$ satisfies the positivity condition. If $g(X_0(p))>0$, then $R(g)\geq24$ and hence $g$ cannot be extremal.
\end{thm}
\begin{proof}
Let $p$ be a prime such that $g(X_0(p))>0$. Then by \autoref{prop:cuspminus} the space $\mathcal{S}_2(\Gamma_0(p))^-$ of cusp forms of weight~$2$ on which the Fricke involution $W_p$ acts as $-1$ is non-trivial. Recall that $\mathcal{S}_2(\Gamma_0(p))^-$ has a basis consisting of Hecke eigenforms. Let $f\in\mathcal{S}_2(\Gamma_0(p))^-$ with $f(\tau)=\sum_{m=1}^\infty a(m)q^m$ be such an eigenform, normalised so that $a(1)=1$. Then \autoref{thm:deligne} shows that $a(2),\ldots,a(p-1)\in\R$, $a(p)=1$ and $|a(m)|\leq\sigma_0(m)\sqrt{m}$ for all $m\in\Ns$.

Since $S=W_p\frac{1}{\sqrt{p}}\left(\begin{smallmatrix}1&0\\0&p\end{smallmatrix}\right)$, we obtain
\begin{equation*}
f|_S(\tau)=\frac{1}{p}f|_{W_p}(\tau/p)=-\frac{1}{p}f(\tau/p).
\end{equation*}
Then we decompose $f(\tau/p)$, which has a Fourier expansion in $q^{1/p}$, as
\begin{equation*}
f(\tau/p)=g_0(\tau)+\ldots+g_{p-1}(\tau)
\end{equation*}
with $g_j|_T=e(j/p)g_j$ for $j\in\Z_p$.

Let $F(\tau)=\sum_{\gamma\in D}F_\gamma(\tau)\ee^\gamma$ be the lift (see \autoref{prop:lift1}) of $f$ to a modular form for the dual Weil representation $\bar\rho_D$. Then \autoref{prop:lift2} yields
\begin{equation*}
F_\gamma(\tau)=\delta_{\gamma,0}f(\tau)-\frac{1}{p}g_{j_\gamma}(\tau)
\end{equation*}
where $j_\gamma\in\{0,\ldots,p-1\}$ such that $j_\gamma/p=-q(\gamma)\mmod{1}$. We insert $F$ into \autoref{prop:pairing2}. Clearly, since $f$ is a cusp form, so is $F$, i.e.\ $[F_\gamma](0)$ vanishes for all $\gamma\in D$. We compute
\begin{equation*}
[F_0](1)=[f](1)-\frac{1}{p}[g_0](1)=[f](1)-\frac{1}{p}[f](p)=a(1)-a(p)/p=1-1/p
\end{equation*}
and
\begin{equation*}
[F_\gamma](r_\gamma)=-\frac{1}{p}[g_{j_\gamma}](r_\gamma)=-\frac{1}{p}[f](j_\gamma)=-\frac{1}{p}a(j_\gamma)
\end{equation*}
for $\gamma\in D$ with $q(\gamma)\neq0\mmod{1}$. Recall that $r_\gamma\in(0,1)$ such that $r_\gamma=-q(\gamma)\mmod{1}$. Hence, by \autoref{prop:pairing2},
\begin{equation*}
p-1=\!\!\sum_{\substack{\gamma\in D\\q(\gamma)=0\bmod{1}}}\!\!a(j_\gamma)\dim(W_{1-r_\gamma}^\gamma)
=\sum_{k=1}^{p-1}a(p-k)d_k
\end{equation*}
with
\begin{equation*}
d_k:=\!\!\sum_{\substack{\gamma\in D\\q(\gamma)=k/n\bmod{1}}}\!\!\dim(W_{k/n}^\gamma).
\end{equation*}
On the other hand, the rest term $R(g)$ reduces to
\begin{equation*}
R(g)=\tilde{R}(g)=\frac{24}{p-1}\sum_{k=1}^{p-1}\sigma(p-k)d_k
\end{equation*}
since $p$ is prime.

Finally, Deligne's bound $|a(m)|\leq\sigma_0(m)\sqrt{m}\leq\sigma(m)$ for $m\in\Ns$ implies
\begin{align*}
\sum_{k=1}^{p-1}\sigma(p-k)d_k&=\sum_{k=1}^{p-1}a(p-k)d_k+\sum_{k=1}^{p-2}(\sigma(p-k)-a(p-k))d_k\\
&=p-1+\sum_{k=1}^{p-2}(\sigma(p-k)-a(p-k))d_k\\
&\geq p-1,
\end{align*}
proving the assertion.
\end{proof}
We remark that the theorem does not extend to the non-prime case. For example, there are extremal orbifolds of order $14$ (see Table \ref{table:70}).

\medskip

Recall the inverse orbifold construction summarised at the end of \autoref{sec:orbifold}. We show that it preserves extremality:
\begin{prop}\label{prop:inverse}
Let $V$ be a \strathol{} \voa{} of central charge $24$ and $g\in\Aut(V)$ of finite order $n>1$ and type~$0$ such that $V^g$ satisfies the positivity condition. Then $g$ is extremal if and only if the inverse-orbifold automorphism $\amgis\in\Aut(V^{\orb(g)})$ is extremal.
\end{prop}
\begin{proof}
The rest term $R(g)$ in \autoref{thm:dimform2} is given by
\begin{equation*}
R(g)=\!\!\sum_{\substack{i,j\in\Z_n\\ij\neq0\bmod{n}}}\!\!d_n(i,j)\dim(W_{1-r_{(i,j)}}^{(i,j)}).
\end{equation*}
Let $\amgis\in\Aut(V^{\orb(g)})$ be the inverse-orbifold automorphism. By definition of $\amgis$, $W^{(i,j)}$ is the eigenspace of $\amgis$ acting on $V^{\orb(g)}(\amgis^j)$ with eigenvalue $e(i/n)$. Hence, the rest term in the corresponding dimension formula is given by
\begin{equation*}
R(\amgis)=\!\!\sum_{\substack{i,j\in\Z_n\\ij\neq0\bmod{n}}}\!\!d_n(i,j)\dim(W_{1-r_{(i,j)}}^{(j,i)})=\!\!\sum_{\substack{i,j\in\Z_n\\ij\neq0\bmod{n}}}\!\!d_n(j,i)\dim(W_{1-r_{(i,j)}}^{(i,j)}),
\end{equation*}
noting that $r_{(i,j)}=r_{(j,i)}$.

Since $g$ is extremal if and only if $\dim(W_{1-r_{(i,j)}}^{(i,j)})=0$ or $d_n(i,j)=0$ for all $i,j\in\Z_n$ with $ij\neq0\mmod{n}$, and similarly for $\amgis$, the assertion follows from the fact that $d_n(i,j)=0$ if and only if $d_n(j,i)=0$ for all $i,j\in\Z_n$ with $ij\neq0\mmod{n}$ (see \autoref{prop:pos2}).
\end{proof}

Finally, as consequence of the dimension bounds, we give another proof of a result on Kneser neighbours of Niemeier lattices by Chenevier and Lannes \cite{CL19}. For $n\in\Ns$, two positive-definite, even, unimodular lattices $L_1$ and $L_2$ are called \emph{(Kneser) $n$-neighbours} if their intersection $K:=L_1\cap L_2$ satisfies $L_1/K\cong\Z_n\cong L_2/K$. Equivalently, $K$ is a positive-definite, even lattice satisfying $K'/K\cong D$, the hyperbolic \fqs{} with group structure $\Z_n\times\Z_n$ and quadratic form $q((i,j))=ij/n\mmod{1}$, and $L_1$ and $L_2$ are unions of the cosets of $K$ corresponding to two trivially intersecting, maximal isotropic subgroups of $D$.

Let $L$ be a positive-definite, even lattice. For $d\in L\otimes_\Z\Q$ we define the finite-index sublattice $L^d=\{\alpha\in L\,|\,(\alpha,d)\in\Z\}$ and denote by $N_r(L)$ the number of lattice vectors $\alpha\in L$ of squared norm $(\alpha,\alpha)=r$. If $L$ is also unimodular, then the $n$-neighbours of $L$ are exactly the lattices $M$ of the form $M=\spn_\Z\{L^d,d\}$ with $d\in L\otimes_\Z\Q$ such that $(d,d)\in2\Z$ and $d+L$ has order~$n$ in $(L\otimes_\Z\Q)/L$. Moreover, $N_2(L)$ is the number of roots of $L$.

\begin{cor}
Let $L_1$ and $L_2$ be Niemeier lattices. If $n>1$ and $L_1$ and $L_2$ are $n$-neighbours, then $L_2=\spn_\Z\{L_1^d,d\}$ for some $d\in L_1\otimes_\Z\Q$ as above and
\begin{equation*}
N_2(L_2)\leq24n+\sum_{m\mid n}c_n(m)N_2(L_1^{md})
\end{equation*}
where the $c_n(m)$ are defined by $\sum_{m\mid n}c_n(m)(t,m)=n/t$ for all $t\mid n$.
\end{cor}
\begin{proof}
Let $V:=V_{L_1}$ be the \strathol{} \voa{} of central charge~$24$ associated with the Niemeier lattice $L_1$ and $g:=\sigma_d=\e^{2\pi\i d_0}$ the inner automorphism of $V_{L_1}$ associated with $d$. Then $g$ has order~$d$ and type~$0$ and it is not difficult to see that $V^{g^m}=V_{L_1^{md}}$ and $V^{\orb(g)}\cong V_{L_2}$ (more details of this construction are given in \autoref{sec:deepholes}). Finally, the assertion follows from \autoref{cor:dimbound2} with the observation that $\dim((V_L)_1)=\rk(L)+N_2(L)$ for any positive-definite, even lattice $L$.
\end{proof}
Since the Leech lattice $\Lambda$ has no roots, $N_2(L)=0$ for all sublattices $L$ of $\Lambda$. Then the above corollary simplifies to:
\begin{cor}
Let $L$ be a Niemeier lattice. If $n>1$ and $L$ is an $n$-neighbour of the Leech lattice $\Lambda$, then $N_2(L)\leq24n$.
\end{cor}
This is the statement of Proposition~3.4.1.1 in \cite{CL19}. They even show that the converse of the statement holds (cf.\ Theorem~B in \cite{CL19}). 

Using \autoref{thm:delignebound} we can replace the term $24p$ by $24(p-1)$ in both results if $n=p$ is a prime with $g(X_0(p))>0$.

We remark that it is not necessary to use \voa{}s to prove these results with the methods in this text. It suffices to consider the theta series of the lattices involved divided by the modular discriminant $\Delta$, which are modular forms of weight~$0$ satisfying the assumptions of this section (and are precisely the characters of the corresponding lattice \voa{}s).


\section{\GDH{}s}\label{sec:gdh}
In this section we introduce the notion of \emph{\gdh{}}s, certain extremal automorphisms of holomorphic \voa{}s of central charge $24$.

Then we show that there is a bijection
\begin{equation*}
g\mapsto V_\Lambda^{\orb(g)}
\end{equation*}
between the \gdh{}s $g$ of the Leech lattice \voa{} $V_\Lambda$ with $\rk((V_\Lambda^g)_1)>0$
and the \strathol{} \voa{}s $V$ of central charge $24$ with $V_1\neq\{0\}$.
This naturally generalises Conway, Parker and Sloane's and Borcherds' correspondence between the Niemeier lattices and the deep holes of the Leech lattice $\Lambda$.

We also prove a uniform construction of the $70$ non-vanishing Lie algebras on Schellekens' list by explicitly giving a \gdh{} of $V_\Lambda$ for each case. Finally, we classify all \gdh{}s of $V_\Lambda$ up to algebraic conjugacy.


\subsection{Deep Holes of the Leech Lattice}\label{sec:deepholes}
As a motivating example for the definition of a \gdh{} we study the orbifold constructions associated with the deep holes of the Leech lattice $\Lambda$.

Recall that up to isomorphism there are exactly $24$ positive-definite, even, unimodular lattices of rank $24$ and that the Leech lattice $\Lambda$ is the unique one amongst them without roots.

The weight-$1$ space of the Leech lattice \voa{} $V_\Lambda$ is the abelian Lie algebra $(V_\Lambda)_1\cong\Lambda_\C$. Hence, the dimension formula simplifies.
\begin{cor}\label{cor:dimfleech}
Let $g$ be an automorphism of $V_\Lambda$ of finite order $n>1$ and type~$0$ such that $V_\Lambda^g$ satisfies the positivity condition. Suppose that $g$ projects to the automorphism $\nu\in\O(\Lambda)$ of cycle shape $\prod_{t\mid m}t^{b_t}$. Then
\begin{equation*}\label{cor:dimleech}
\dim((V_\Lambda^{\orb(g)})_1)=24+n\sum_{t\mid m}\frac{b_t}{t}-R(g)=24+24n(1-\rho_\nu)-R(g)
\end{equation*}
where $\rho_\nu$ is the vacuum anomaly introduced in \autoref{sec:lat}.
\end{cor}
\begin{proof}
The automorphism $\nu^k$ has cycle shape $\prod_{t\mid m}(t/(k,t))^{b_t(k,t)}$. Hence
\begin{equation*}
\dim((V_\Lambda^{g^k})_1)=\rk(\Lambda^{\nu^k})=\sum_{t\mid m} b_t(k,t)
\end{equation*}
so that the relevant term in \autoref{thm:dimform2} becomes
\begin{equation*}
\sum_{k\mid n}c_n(k)\dim((V_\Lambda^{g^k})_1)=\sum_{k\mid n}c_n(k)\sum_{t\mid m}b_t(k,t)=n\sum_{t\mid m}\frac{b_t}{t}=24n(1-\rho_{\nu}).
\end{equation*}
This proves the assertion.
\end{proof}
The corollary shows that, given an automorphism $\nu\in\O(\Lambda)$, the dimensions of extremal orbifold constructions associated with automorphisms of $V_\Lambda$ projecting to $\nu$ lie in $24+(m\sum_{t\mid m}b_t/t)\Ns$. Note that for automorphisms of the Leech lattice $(1/24)\sum_{t\mid m}b_t/t=1-\rho_\nu$ has a denominator dividing $m(12,m)$.

\medskip

We now recall the bijection between the $23$ classes of deep holes of the Leech lattice $\Lambda$ and the $23$ Niemeier lattices other than the Leech lattice described in \cite{CPS82,CS82}.

A \emph{hole} in a positive-definite lattice $L$ is a point $d\in L\otimes_\Z\R$ in which the distance to the lattice $L$ assumes a local maximum. If the distance is a global maximum, $d$ is called a \emph{deep hole}. In \cite{CPS82} the authors show that the covering radius of the Leech lattice $\Lambda$ is $\sqrt{2}$ and that there are exactly $23$ orbits of deep holes in
$\Lambda\otimes_\Z\R$ under the affine automorphism group $\O(\Lambda)\ltimes\Lambda$. The authors observe that these are in a natural bijection with the $23$ isomorphism classes of Niemeier lattices other than the Leech lattice itself.

In \cite{CS82}, for each deep hole, a simultaneous construction of the Leech lattice and of the corresponding Niemeier lattice is given. Let $d\in\Lambda\otimes_\Z\R$ be a deep hole of the Leech lattice. Let $n$ be the order of $d+\Lambda\in(\Lambda\otimes_\Z\R)/\Lambda$. We consider the sublattice $\Lambda^d:=\{\alpha\in\Lambda\,|\,(\alpha,d)\in\Z\}$ of $\Lambda$ of index $n$. Then $\spn_\Z\{\Lambda^d,d\}$ contains $\Lambda^d$ as a sublattice of index $n$ and is even, unimodular and in fact isomorphic to the Niemeier lattice $N$ corresponding to $d$. Note that $n=h$ where $h$ is the Coxeter number of (any irreducible component of) the root system of $N$.

The original proof of these results is based on a case-by-case analysis. Borcherds found a conceptual proof in \cite{Bor85b}
by embedding all lattices into $\II_{25,1}$, the unique even, unimodular lattice of Lorentzian signature $(25,1)$.

We now study the orbifold constructions of the Leech lattice \voa{} $V_\Lambda$ corresponding to deep holes of the Leech lattice.
\begin{prop}\label{prop:deepholecons}
Let $d\in\Lambda\otimes_\Z\R$ be a deep hole of the Leech lattice $\Lambda$ corresponding to the Niemeier lattice $N$. Then $g=\sigma_d=\e^{2\pi\i d_0}$ is an automorphism of $V_{\Lambda}$ of order equal to the Coxeter number of $N$ and type~$0$. The corresponding orbifold construction $V_{\Lambda}^{\orb(g)}$ is isomorphic to $V_N$.
\end{prop}
\begin{proof}
Let $d\in\Lambda\otimes_\Z\Q$ and consider the inner automorphism $g=\sigma_d=\e^{2\pi\i d_0}\in\Aut(V_\Lambda)$. Then $g$ is an automorphism of $V_\Lambda$ of order~$n$, the smallest positive integer such that $nd\in\Lambda'=\Lambda$. The conformal weight (see \autoref{sec:lat}) of the unique irreducible $g$-twisted $V_\Lambda$-module $V_\Lambda(g)$ is 
$\rho(V_\Lambda(g))=\min_{\alpha\in-d+\Lambda}(\alpha,\alpha)/2$, which is at most $1$ because the covering radius of $\Lambda$ is $\sqrt{2}$.

Assume that $d$ is a deep hole of $\Lambda$, which is the case if and only if $\rho(V_\Lambda(g))=1$. Hence, $g$ has type~$0$, and we may consider the orbifold construction $V_\Lambda^{\orb(g)}$. The \fpvosa{} $V_\Lambda^g$ is the lattice \voa{} $V_{\Lambda^d}$. Hence, rather than applying the general theory of cyclic orbifolds described in \autoref{sec:orbifold} it suffices to make use of the representation theory of lattice \voa{}s. The $n^2$ irreducible $V_{\Lambda^d}$-modules are given by $V_{\alpha+\Lambda^d}$ for $\alpha+\Lambda^d\in (\Lambda^d)'/\Lambda^d\cong\Z_n^2$. Note that $(\Lambda^d)'=\spn_\Z\{\Lambda,d\}$.

Now, $\Lambda/\Lambda^d$ and $\spn_\Z\{\Lambda^d,d\}/\Lambda^d\cong N/\Lambda^d$ form a pair of trivially intersecting, maximal isotropic subgroups of $(\Lambda^d)'/\Lambda^d$. The sum of irreducible $V_{\Lambda^d}$-modules corresponding to the former gives back $V_\Lambda$ while the one corresponding to the latter yields
\begin{equation*}
V_\Lambda^{\orb(g)}=\bigoplus_{j\in\Z_n}V_{jd+\Lambda^d}\cong V_N.\qedhere
\end{equation*}
\end{proof}
The proposition shows that we can replicate the deep-hole construction of the $23$ Niemeier lattices other than the Leech lattice on the level of \strathol{} \voa{}s of central charge $24$, simply by studying inner automorphisms.

We collect some properties of these $23$ orbifold constructions, which will serve as defining properties of \gdh{}s.

\begin{prop}\label{prop:deepholes}
Let $d$ be a deep hole of $\Lambda$ and $g=\sigma_d=\e^{2\pi\i d_0}$ the corresponding inner automorphism of $V_\Lambda$. Then $g$ has type~$0$ and $V_\Lambda^g$ satisfies the positivity condition. Furthermore
\begin{enumerate}
\item\label{item:deephole2} $g$ is extremal,
\item\label{item:deephole3} $\rk((V_\Lambda^g)_1)=\rk((V_\Lambda^{\orb(g)})_1)=24$.
\end{enumerate}
\end{prop}
\begin{proof}
It was already stated in \autoref{prop:deepholecons} that $g$ has type~$0$.
By \autoref{cor:dimleech}, the upper bound from the dimension formula for orbifold constructions of $V_\Lambda$ with inner automorphisms of order~$n$ is given by
\begin{equation*}
\dim((V_\Lambda^{\orb(g)})_1)\leq24+24n.
\end{equation*}
On the other hand, let $N$ be the Niemeier lattice with $V_{\Lambda}^{\orb(g)}\cong V_N$. Then
\begin{equation*}
\dim((V_\Lambda^{\orb(g)})_1)=\dim((V_N)_1)=24+24h
\end{equation*}
because the number of roots of $N$ is $24h$, $h$ being the Coxeter number of $N$. Since $n=h$, the upper bound is attained and $g$ is extremal.

The automorphism $g$ acts trivially on the $24$-dimensional abelian Lie algebra $(V_\Lambda)_1$ so that $\rk((V_\Lambda^g)_1)=24$. Because $(V_\Lambda^{\orb(g)})_1\cong(V_N)_1$ is a semisimple Lie algebra of rank $24$, the assertion follows.
\end{proof}
Note that $\rho(V_\Lambda(g))=1$ is necessary albeit in general not sufficient for $g$ to be extremal (see \autoref{prop:necsufff}). For many but not all of the $23$ deep holes, $\rho(V_\Lambda(g^i))=1$ for all $i\in\Z_n\setminus\{0\}$, which also implies the extremality of $g$.


\subsection{Definition of \GDH{}s}\label{sec:gdhdef}
We use the properties of the deep-hole orbifold constructions listed in \autoref{prop:deepholes} to define \gdh{}s. In order to obtain non-lattice cases on Schellekens' list as orbifold constructions associated with $V_\Lambda$, we allow $(V_\Lambda^{\orb(g)})_1$ to have rank less than $24$.

\begin{defi}[\GDH{}]\label{def:gendeephole}
Let $V$ be a \strathol{} \voa{} of central charge $24$ and $g\in\Aut(V)$ of finite order $n>1$. Suppose $g$ has type~$0$ and $V^g$ satisfies the positivity condition. Then $g$ is called a \emph{\gdh} of $V$ if
\begin{enumerate}
\item $g$ is extremal, i.e.\ $\dim(V_1^{\orb(g)})=24+\sum_{m\mid n}c_n(m)\dim(V_1^{g^m})$,
\item $\rk(V_1^{\orb(g)})=\rk(V_1^g)$.
\end{enumerate}
\end{defi}

Note that the Lie algebras $V_1^g$ and $V_1^{\orb(g)}$ are reductive \cite{DM04b}. By Lemma~8.1 in \cite{Kac90} the centraliser in $V_1^{\orb(g)}$ of any choice of Cartan subalgebra of $V_1^g$ is a Cartan subalgebra of $V_1^{\orb(g)}$. The second condition is hence equivalent to requiring that any Cartan subalgebra of $V_1^g$ also be a Cartan subalgebra of $V_1^{\orb(g)}$. Since a finite-order automorphism of a reductive Lie algebra is inner if and only if it fixes a Cartan subalgebra pointwise (see the proof of Proposition~8.6 in \cite{Kac90}), the second condition can also be replaced by the condition that the inverse-orbifold automorphism restricts to an inner automorphism on $V_1^{\orb(g)}$.


If $V=V_\Lambda$, the \voa{} associated with the Leech lattice $\Lambda$, then the rank condition is equivalent to demanding that $(V_\Lambda^g)_1$, which as subalgebra of $(V_\Lambda)_1$ is abelian, be a Cartan subalgebra of $(V_\Lambda^{\orb(g)})_1$.

Again, we also call the identity a \gdh{}.


\subsection{Uniform Construction of Schellekens' List}\label{sec:constr}
In the following, we show that there is a bijection between the \strathol{} \voa{}s of central charge $24$ with non-trivial weight-$1$ subspace and the \gdh{}s of the Leech lattice \voa{} $V_\Lambda$ with non-trivial fixed-point Lie subalgebra. This is a natural generalisation of the bijection between the deep holes of the Leech lattice $\Lambda$ and the $23$ Niemeier lattices with non-trivial root system proved in \cite{CPS82,CS82,Bor85b}.

Then we prove a uniform construction of the $70$ non-zero Lie algebra structures on Schellekens' list by explicitly giving a \gdh{} of $V_\Lambda$ for each case.

This is an improvement over the potpourri of constructions in \cite{FLM88,DGM90,Don93,DGM96,Lam11,LS12a,LS15a,Miy13b,EMS20a,Moe16,SS16,LS16a,LS16,LL20}. We remark that other uniform constructions are given in \cite{Hoe17} (and \cite{Lam20}), as well as in \cite{HM20}.

\smallskip

First, we combine the dimension formula in \autoref{thm:dimform2} with the averaged very strange formula in \autoref{thm:averageverystrange}. This generalises Theorem~3.3 in \cite{ELMS21} to non-quasirational automorphisms.
\begin{prop}\label{prop:dimstrange}
Let $V$ be a \strathol{} \voa{} of central charge $24$ whose weight-$1$ Lie algebra $V_1=\bigoplus_{i=1}^r\g_i$ is semisimple and $g$ an inner automorphism of $V$ of finite order $n>1$ such that $g$ is of type~$0$ and that $V^g$ satisfies the positivity condition. Then
\begin{equation*}
\dim(V^{\orb(g)}_1)=24+\frac{12n}{\phi(n)}\sum_{i=1}^r\sum_{j\in\Z_n^*}h^\vee_i\left|\lambda_{s_i^{[j]}}-\frac{\rho_i}{h_i^\vee}\right|^2-R(g).
\end{equation*}
Here, $R(g)\geq0$ is as in \autoref{thm:dimform2}, $\rho_i$ is the Weyl vector and $h_i^\vee$ the dual Coxeter number of $\g_i$ and $s_i^{[j]}$ the sequence characterising $g^j|_{\g_i}$.
\end{prop}
\begin{proof}
As inner automorphism, $g$ restricts to each simple ideal of $V_1$ so that the relevant term in the dimension formula in \autoref{thm:dimform2} becomes
\begin{align*}
\sum_{m\mid n}c_n(m)\dim(V_1^{g^m})
&=\sum_{i=1}^r\sum_{m\mid n}c_n(m)\dim(\g_i^{g^m})\\
&=\frac{12n}{\phi(n)}\sum_{i=1}^r\sum_{j\in\Z_n^*}h^\vee_i\left|\lambda_{s_i^{[j]}}-\frac{\rho_i}{h_i^\vee}\right|^2
\end{align*}
by \autoref{thm:averageverystrange}, where $n$ is in general a multiple of the order of $g|_{V_1}$ and $g|_{\g_i}$.
\end{proof}
Note that the assumption that $V_1$ is semisimple and $g$ inner is merely made for convenience and could be easily removed.

\begin{thm}[Holy Correspondence]\label{thm:main2}
The cyclic orbifold construction $g\mapsto V_\Lambda^{\orb(g)}$ defines a bijection between the algebraic conjugacy classes of \gdh{}s $g\in\Aut(V_\Lambda)$ with $\rk((V_\Lambda^g)_1)>0$ and the isomorphism classes of \strathol{} \voa{}s $V$ of central charge $24$ with $V_1\neq\{0\}$.
\end{thm}
\begin{proof}
Let $V$ be a \strathol{} \voa{} $V$ of central charge $24$ with non-trivial weight-$1$ Lie algebra $V_1$. Then $V_1$ is either abelian and of rank $24$ or semisimple \cite{DM04}. In the first case $V$ must be isomorphic to the Leech lattice \voa{} $V_\Lambda$ \cite{DM04b}. Moreover, if $g$ is a \gdh{} such that $V_\Lambda^{\orb(g)}\cong V_\Lambda$, then $24=\rk((V_\Lambda^{\orb(g)})_1)=\rk((V_\Lambda^g)_1)=\rk(\Lambda^{\nu})$ where $\nu$ is the projection of $g$ to $\O(\Lambda)$. Hence, $\nu$ is the identity and \autoref{cor:dimfleech} implies that so is $g$. Therefore, in the following we assume that $\g:=V_1$ is semisimple and that $g$ is non-trivial.

We sketch the proof of the surjectivity, which is given in \cite{ELMS21}. We decompose $\g=\g_1\oplus\ldots\oplus\g_r$ into simple components $\g_i$. Then the \vosa{} of $V$ generated by $V_1$ is isomorphic to the affine \voa{} $L_{\hat\g_1}(k_1,0)\otimes\ldots\otimes L_{\hat\g_r}(k_r,0)$ with levels $k_i\in\Ns$ \cite{DM06b}. Schellekens' lowest-order trace identity \cite{Sch93} (see also \cite{DM04,EMS20a}) states that $h_i^\vee/k_i=(\dim(V_1)-24)/24$ for all $i=1,\ldots,r$ where $h_i^\vee$ is the dual Coxeter number of $\g_i$. Let $\rho_i$ denote the Weyl vector of $\g_i$. Define the inner automorphism $\sigma_u=\e^{2\pi\i u_0}\in\Aut(V)$ with $u:=\sum_{i=1}^r\rho_i/h_i^\vee$. It restricts to the automorphism $\sigma_i:=\sigma_{\rho_i/h_i^\vee}$ introduced in \autoref{sec:verystrange} on each $\g_i$. The order of $\sigma_u$ is $\lcm(\{l_ih_i^\vee\}_{i=1}^r)$ with lacing numbers $l_i\in\{1,2,3\}$. The \fpvosa{} $V^{\sigma_u}$ satisfies the positivity condition and the lowest-order trace identity implies that $\sigma_u$ has type~$0$. Hence, the orbifold construction $V^{\orb(\sigma_u)}$ exists. Inserting $\sigma_u$ in \autoref{prop:dimstrange} yields $\dim(V_1^{\orb(\sigma_u)})=24-R(\sigma_u) \leq24$ since $\lambda_{s_i^{[j]}}=\frac{\rho_i}{h_i^\vee}$ for all $i=1,\ldots,r$ and $j\in\Z_n^*$, each $\sigma_u^j|_{\g_i}$ being conjugate to $\sigma_i$ by the rationality of $\sigma_i$. Again by the lowest-order trace identity, $V_1^{\orb(\sigma_u)}$ cannot be semisimple (or else $\dim(V_1^{\orb(\sigma_u)})>24$) so that it is abelian of rank $24$. Hence, $V^{\orb(\sigma_u)}\cong V_\Lambda$ by \cite{DM04b}. Since $R(\sigma_u)=0$, $\sigma_u$ is extremal.

Finally, we consider the inverse-orbifold automorphism $g\in\Aut(V_\Lambda)$ with respect to $\sigma_u\in\Aut(V)$. The properties of $\sigma_u$ imply that $g$ is a \gdh{} and $V\cong V_\Lambda^{\orb(g)}$. Indeed, $g$ is extremal by \autoref{prop:inverse} and the rank condition is satisfied because $\sigma_u$ is inner. This proves the surjectivity.

We remark that we could have used the very strange formula rather than its averaged version since $\sigma_u|_{V_1}$ is quasirational (see Theorem~3.3 in \cite{ELMS21}).

Next we prove the injectivity. Let $g\in\Aut(V_{\Lambda})$ be a \gdh{} of order $n>1$. Define $V:=V_\Lambda^{\orb(g)}$. Then the inverse-orbifold automorphism $\sigma$ corresponding to $g$ is of type~$0$, inner and extremal by \autoref{prop:inverse}. As above we decompose $\g=\g_1\oplus\ldots\oplus\g_r$ into simple ideals. Then, \autoref{prop:dimstrange} implies that
\begin{equation*}
24=\dim((V_\Lambda)_1)=\dim(V^{\orb(\sigma)}_1)=24+\frac{12n}{\phi(n)}\sum_{i=1}^r\sum_{j\in\Z_n^*}h^\vee_i\left|\lambda_{s_i^{[j]}}-\frac{\rho_i}{h_i^\vee}\right|^2
\end{equation*}
where $s_i^{[j]}$ is the sequence characterising $\sigma^j|_{\g_i}$. This equation can only be satisfied if $\sigma|_{V_1}$ is conjugate to $\e^{2\pi\i\ad(u)}$ with $u=\sum_{i=1}^r\rho_i/h_i^\vee$. In fact, $\sigma|_{V_1}$ is conjugate to $\e^{2\pi\i\ad(u)}$ under an inner automorphism of $V_1$ (see remark at the end of \autoref{sec:verystrange}). Since inner automorphisms of $V_1$ can be extended to inner automorphisms of $V$, $\sigma$ is conjugate in $\Aut(V)$ to $\sigma_u=\e^{2\pi\i u_0}$. Hence, $g$ is conjugate to the inverse-orbifold automorphism of $\sigma_u$. The same is true for any other \gdh{} $h$ of $V_\Lambda$ with $V_\Lambda^{\orb(h)}\cong V$. This proves the injectivity of the map.
\end{proof}


As a consequence of the correspondence, it is possible to classify the \strathol{} \voa{}s of central charge $24$ with non-trivial weight-$1$ space by classifying the \gdh{}s $g\in\Aut(V_{\Lambda})$ satisfying $\rk((V_\Lambda^g)_1)>0$. This shall be addressed in a forthcoming publication \cite{MS21} using geometric properties of \gdh{}s. (In the next section, we proceed in the opposite direction to obtain a classification of the \gdh{}s.)

The surjectivity of the map also follows (less conceptually) from \autoref{thm:main3} below together with Schellekens' classification of the possible $V_1$-structures of the \strathol{} \voa{}s of central charge $24$ and the uniqueness of such a \voa{} for non-trivial $V_1$.

\medskip

We now give a uniform construction of the $70$ semisimple Lie algebra structures on Schellekens' list. The proof uses Schellekens' classification of the possible $V_1$-structures and Kac's theory of finite-order automorphisms of simple Lie algebras \cite{Kac90}.
\begin{thm}[Uniform Construction]\label{thm:main3}
Let $\g$ be one of the $70$ non-zero Lie algebras on Schellekens' list (Table~1 in \cite{Sch93}). Then the corresponding automorphism $g\in\Aut(V_\Lambda)$ listed in the appendix (\autoref{table:11expl} and \autoref{table:70expl}) is a \gdh{} and satisfies $(V_\Lambda^{\orb(g)})_1\cong\g$.
\end{thm}
Some properties that only depend on the algebraic conjugacy class of the \gdh{} $g$ are listed in \autoref{table:11} and \autoref{table:70}.

\renewcommand{\AA}{1^{24}}
\newcommand{\BB}{1^82^8}
\renewcommand{\CC}{1^63^6}
\newcommand{\DD}{2^{12}}
\newcommand{\EE}{1^42^24^4}
\newcommand{\FF}{1^45^4}
\newcommand{\GG}{1^22^23^26^2}
\newcommand{\HHH}{1^37^3}
\newcommand{\III}{1^22^14^18^2}
\newcommand{\JJ}{2^36^3}
\newcommand{\KK}{2^210^2}

\begin{table}
\caption{Summary of the \gdh{}s of the Leech lattice \voa{} $V_\Lambda$ and the corresponding orbifold constructions.}
\renewcommand{\arraystretch}{1.2}
\begin{tabular}{lllrrl}
\cite{Hoe17} & Cycle shape & $n$ & \# & Rk.\ & Dim.\\\hline\hline
A&$\AA$ &$1,2,\ldots,25,30,46$&24&24&$24+24n-24\delta_{n,1}$\\\hline
B&$\BB$ &$2,4,\ldots,18,22,30$&17&16&$24+12n$    \\\hline
C&$\CC$ &$3,6,9,12,18$        &6 &12&$24+8n$     \\\hline
D&$\DD$ &$2,6,10,\ldots,22,46$&9 &12&$24+6n$     \\\hline
E&$\EE$ &$4,8,12,16$          &5 &10&$24+6n$     \\\hline
F&$\FF$ &$5,10$               &2 & 8&$24+(24/5)n$\\\hline
G&$\GG$ &$6,12$               &2 & 8&$24+4n$     \\\hline
H&$\HHH$ &$7$                  &1 & 6&$24+(24/7)n$\\\hline
I&$\III$&$8$                  &1 & 6&$24+3n$     \\\hline
J&$\JJ$ &$6,18$               &2 & 6&$24+2n$     \\\hline
K&$\KK$ &$10$                 &1 & 4&$24+(6/5)n$ \\\hline\hline
L&$1^{-24}2^{24},1^{-12}3^{12},\ldots$&$2,3,\ldots,70,78$&38&0&0\\
\end{tabular}
\renewcommand{\arraystretch}{1}
\label{table:11}
\end{table}

\begin{table}
\caption{The $70$ \gdh{}s of $V_\Lambda$ whose corresponding orbifold constructions realise all non-zero Lie algebras on Schellekens' list (continued on next page).}
\renewcommand{\arraystretch}{1.09}
\begin{tabular}{rllrrll}
No. & No. & $(V_\Lambda^{\orb(g)})_1$ & Dim. & $n$ & $\rho(V_\Lambda(g^c))$ & $(V_\Lambda(g))_1$\\\hline\hline
\multicolumn{6}{c}{Rk. 24, cycle shape $1^{24}$}\\\hline\\[-1.2em]
70 &  A1 & $D_{24,1}$                  & 1128 & 46 & $1, 22/23, 1, 0$                  & $\tilde{D}_{24}$\\
69 &  A2 & $D_{16,1}E_{8,1}$           &  744 & 30 & $1, 14/15, 1, 1, 1, 1, 1, 0$      & $\tilde{D}_{16}\tilde{E}_8$\\
68 &  A3 & $E_{8,1}^3$                 &  744 & 30 & $1, 14/15, 9/10, 5/6, 1, 1, 1, 0$ & $\tilde{E}_8^3$\\
67 &  A4 & $A_{24,1}$                  &  624 & 25 & $1, 1, 0$                         & $\tilde{A}_{24}$\\
66 &  A5 & $D_{12,1}^2$                &  552 & 22 & $1, 10/11, 1, 0$                  & $\tilde{D}_{12}^2$\\
65 &  A6 & $A_{17,1}E_{7,1}$           &  456 & 18 & $1, 1, 1, 1, 1, 0$                & $\tilde{A}_{17}\tilde{E}_7$\\
64 &  A7 & $D_{10,1}E_{7,1}^2$         &  456 & 18 & $1, 8/9, 1, 1, 1, 0$              & $\tilde{D}_{10}\tilde{E}_7^2$\\
63 &  A8 & $A_{15,1}D_{9,1}$           &  408 & 16 & $1, 1, 1, 1, 0$                   & $\tilde{A}_{15}\tilde{D}_9$\\
61 &  A9 & $D_{8,1}^3$                 &  360 & 14 & $1, 6/7, 1, 0$                    & $\tilde{D}_8^3$\\
60 & A10 & $A_{12,1}^2$                &  336 & 13 & $1, 0$                            & $\tilde{A}_{12}^2$\\
59 & A11 & $A_{11,1}D_{7,1}E_{6,1}$    &  312 & 12 & $1, 1, 1, 1, 1, 0$                & $\tilde{A}_{11}\tilde{D}_7\tilde{E}_6$\\
58 & A12 & $E_{6,1}^4$                 &  312 & 12 & $1, 1, 3/4, 1, 1, 0$              & $\tilde{E}_6^4$\\
55 & A13 & $A_{9,1}^2D_{6,1}$          &  264 & 10 & $1, 1, 1, 0$                      & $\tilde{A}_9^2\tilde{D}_6$\\
54 & A14 & $D_{6,1}^4$                 &  264 & 10 & $1, 4/5, 1, 0$                    & $\tilde{D}_6^4$\\
51 & A15 & $A_{8,1}^3$                 &  240 &  9 & $1, 1, 0$                         & $\tilde{A}_8^3$\\
49 & A16 & $A_{7,1}^2D_{5,1}^2$        &  216 &  8 & $1, 1, 1, 0$                      & $\tilde{A}_7^2\tilde{D}_5^2$\\
46 & A17 & $A_{6,1}^4$                 &  192 &  7 & $1, 0$                            & $\tilde{A}_6^4$\\
43 & A18 & $A_{5,1}^4D_{4,1}$          &  168 &  6 & $1, 1, 1, 0$                      & $\tilde{A}_5^4\tilde{D}_4$\\
42 & A19 & $D_{4,1}^6$                 &  168 &  6 & $1, 2/3, 1, 0$                    & $\tilde{D}_4^6$\\
37 & A20 & $A_{4,1}^6$                 &  144 &  5 & $1, 0$                            & $\tilde{A}_4^6$\\
30 & A21 & $A_{3,1}^8$                 &  120 &  4 & $1, 1, 0$                         & $\tilde{A}_3^8$\\
24 & A22 & $A_{2,1}^{12}$              &   96 &  3 & $1, 0$                            & $\tilde{A}_2^{12}$\\
15 & A23 & $A_{1,1}^{24}$              &   72 &  2 & $1, 0$                            & $\tilde{A}_1^{24}$\\
 1 & A24 & $\C^{24}$                   &   24 &  1 & $0$                               & $\emptyset$\\\hline
\multicolumn{6}{c}{Rk. 16, cycle shape $1^82^8$}\\\hline\\[-1.2em]
62 &  B1 & $B_{8,1}E_{8,2}$            &  384 & 30 & $1, 14/15, 1, 1, 1, 1, 1, 0$      & $A_1\tilde{E}_8$\\
56 &  B2 & $C_{10,1}B_{6,1}$           &  288 & 22 & $1, 10/11, 1, 0$                  & $A_1A_9$\\
52 &  B3 & $C_{8,1}F_{4,1}^2$          &  240 & 18 & $1, 8/9, 1, 1, 1, 0$              & $A_2^2A_7$\\
53 &  B4 & $B_{5,1}E_{7,2}F_{4,1}$     &  240 & 18 & $1, 8/9, 1, 1, 1, 0$              & $A_1A_2\tilde{E}_7$\\
50 &  B5 & $A_{7,1}D_{9,2}$            &  216 & 16 & $1, 1, 1, 1, 0$                   & $\tilde{D}_9$\\
47 &  B6 & $B_{4,1}^2D_{8,2}$          &  192 & 14 & $1, 6/7, 1, 0$                    & $A_1^2\tilde{D}_8$\\
48 &  B7 & $B_{4,1}C_{6,1}^2$          &  192 & 14 & $1, 6/7, 1, 0$                    & $A_1A_5^2$\\
44 &  B8 & $A_{5,1}C_{5,1}E_{6,2}$     &  168 & 12 & $1, 1, 1, 1, 1, 0$                & $A_4\tilde{E}_6$\\
40 &  B9 & $A_{4,1}A_{9,2}B_{3,1}$     &  144 & 10 & $1, 1, 1, 0$                      & $A_1\tilde{A}_9$\\
39 & B10 & $B_{3,1}^2C_{4,1}D_{6,2}$   &  144 & 10 & $1, 4/5, 1, 0$                    & $A_1^2A_3\tilde{D}_6$\\
38 & B11 & $C_{4,1}^4$                 &  144 & 10 & $1, 4/5, 1, 0$                    & $A_3^4$\\
33 & B12 & $A_{3,1}A_{7,2}C_{3,1}^2$   &  120 &  8 & $1, 1, 1, 0$                      & $A_2^2\tilde{A}_7$\\
31 & B13 & $A_{3,1}^2D_{5,2}^2$        &  120 &  8 & $1, 1, 1, 0$                      & $\tilde{D}_5^2$\\
26 & B14 & $A_{2,1}^2A_{5,2}^2B_{2,1}$ &   96 &  6 & $1, 1, 1, 0$                      & $A_1\tilde{A}_5^2$\\
25 & B15 & $B_{2,1}^4D_{4,2}^2$        &   96 &  6 & $1, 2/3, 1, 0$                    & $A_1^4\tilde{D}_4^2$\\
16 & B16 & $A_{1,1}^4A_{3,2}^4$        &   72 &  4 & $1, 1, 0$                         & $\tilde{A}_3^4$\\
 5 & B17 & $A_{1,2}^{16}$              &   48 &  2 & $1, 0$                            & $\tilde{A}_1^{16}$\\
\end{tabular}
\renewcommand{\arraystretch}{1}
\label{table:70}
\end{table}
\addtocounter{table}{-1}

\begin{table}
\caption{(continued)}
\renewcommand{\arraystretch}{1.2}
\begin{tabular}{rllrrll}
No. & No. & $(V_\Lambda^{\orb(g)})_1$ & Dim. & $n$ & $\rho(V_\Lambda(g^c))$ & $(V_\Lambda(g))_1$\\\hline\hline
\multicolumn{6}{c}{Rk. 12, cycle shape $1^63^6$}\\\hline
45 &  C1 & $A_{5,1}E_{7,1}$            & 168 & 18 & $1, 1, 1, 1, 1, 0$   & $\tilde{E}_7$\\
34 &  C2 & $A_{3,1}D_{7,3}G_{2,1}$     & 120 & 12 & $1, 1, 1, 1, 1, 0$   & $A_1\tilde{D}_7$\\
32 &  C3 & $E_{6,3}G_{2,1}^3$          & 120 & 12 & $1, 1, 3/4, 1, 1, 0$ & $A_1^3\tilde{E}_6$\\
27 &  C4 & $A_{2,1}^2A_{8,3}$          &  96 &  9 & $1, 1, 0$            & $\tilde{A}_8$\\
17 &  C5 & $A_{1,1}^3A_{5,3}D_{4,3}$   &  72 &  6 & $1, 1, 1, 0$         & $\tilde{A}_5\tilde{D}_4$\\
 6 &  C6 & $A_{2,3}^6$                 &  48 &  3 & $1, 0$               & $\tilde{A}_2^6$\\\hline
\multicolumn{6}{c}{Rk. 12, cycle shape $2^{12}$ (order doubling)}\\\hline
57 & D1a & $B_{12,2}$                  & 300 & 46 & $1, 22/23, 1, 0$     & $A_1$\\
41 & D1b & $B_{6,2}^2$                 & 156 & 22 & $1, 10/11, 1, 0$     & $A_1^2$\\
29 & D1c & $B_{4,2}^3$                 & 108 & 14 & $1, 6/7, 1, 0$       & $A_1^3$\\
23 & D1d & $B_{3,2}^4$                 &  84 & 10 & $1, 4/5, 1, 0$       & $A_1^4$\\
12 & D1e & $B_{2,2}^6$                 &  60 &  6 & $1, 2/3, 1, 0$       & $A_1^6$\\
 2 & D1f & $A_{1,4}^{12}$              &  36 &  2 & $1, 0$               & $\tilde{A}_1^{12}$\\
36 & D2a & $A_{8,2}F_{4,2}$            & 132 & 18 & $1, 1, 1, 1, 1, 0$   & $A_2$\\
22 & D2b & $A_{4,2}^2C_{4,2}$          &  84 & 10 & $1, 1, 1, 0$         & $A_3$\\
13 & D2c & $A_{2,2}^4D_{4,4}$          &  60 &  6 & $1, 1, 1, 0$         & $\tilde{D}_4$\\\hline
\multicolumn{6}{c}{Rk. 10, cycle shape $1^42^24^4$}\\\hline
35 & E1 & $A_{3,1}C_{7,2}$             & 120 & 16 & $1, 1, 1, 1, 0$      & $A_6$\\
28 & E2 & $A_{2,1}B_{2,1}E_{6,4}$      &  96 & 12 & $1, 1, 1, 1, 1, 0$   & $\tilde{E}_6$\\
18 & E3 & $A_{1,1}^3A_{7,4}$           &  72 &  8 & $1, 1, 1, 0$         & $\tilde{A}_7$\\
19 & E4 & $A_{1,1}^2C_{3,2}D_{5,4}$    &  72 &  8 & $1, 1, 1, 0$         & $A_2\tilde{D}_5$\\
 7 & E5 & $A_{1,2}A_{3,4}^3$           &  48 &  4 & $1, 1, 0$            & $\tilde{A}_3^3$\\\hline
\multicolumn{6}{c}{Rk. 8, cycle shape $1^45^4$}\\\hline
20 & F1 & $A_{1,1}^2D_{6,5}$           &  72 & 10 & $1, 1, 1, 0$         & $\tilde{D}_6$\\
 9 & F2 & $A_{4,5}^2$                  &  48 &  5 & $1, 0$               & $\tilde{A}_4^2$\\\hline
\multicolumn{6}{c}{Rk. 8, cycle shape $1^22^23^26^2$}\\\hline
21 & G1 & $A_{1,1}C_{5,3}G_{2,2}$      &  72 & 12 & $1, 1, 1, 1, 1, 0$   & $A_1A_4$\\
 8 & G2 & $A_{1,2}A_{5,6}B_{2,3}$      &  48 &  6 & $1, 1, 1, 0$         & $A_1\tilde{A}_5$\\\hline
\multicolumn{6}{c}{Rk. 6, cycle shape $1^37^3$}\\\hline
11 & H1 & $A_{6,7}$                    &  48 &  7 & $1, 0$               & $\tilde{A}_6$\\\hline
\multicolumn{6}{c}{Rk. 6, cycle shape $1^22^14^18^2$}\\\hline
10 & I1 & $A_{1,2}D_{5,8}$             &  48 &  8 & $1, 1, 1, 0$         & $\tilde{D}_5$\\\hline
\multicolumn{6}{c}{Rk. 6, cycle shape $2^36^3$ (order doubling)}\\\hline
14 & J1a & $A_{2,2}F_{4,6}$            &  60 & 18 & $1, 1, 1, 1, 1, 0$   & $A_2$\\
 3 & J1b & $A_{2,6}D_{4,12}$           &  36 &  6 & $1, 1, 1, 0$         & $\tilde{D}_4$\\\hline
\multicolumn{6}{c}{Rk. 4, cycle shape $2^210^2$ (order doubling)}\\\hline
4 & K1 & $C_{4,10}$                    &  36 & 10 & $1, 1, 1, 0$         & $A_3$
\end{tabular}
\renewcommand{\arraystretch}{1}
\end{table}

\begin{proof}
In \autoref{table:11expl} and \autoref{table:70expl} of the appendix (see also \autoref{table:11} and \autoref{table:70}) we list representatives of $70$ algebraic conjugacy classes of automorphisms $g\in\Aut(V_\Lambda)$ with $\rk((V_\Lambda^g)_1)>0$. In each case, we shall identify the weight-$1$ Lie algebra $(V_\Lambda^{\orb(g)})_1$ of the corresponding orbifold construction as the one listed in \autoref{table:70} and prove that $g$ is a \gdh{}. Since each of the $70$ non-zero Lie algebras on Schellekens' list appears exactly once, this will prove the assertion.

First, we describe these $70$ automorphisms. A complete system of representatives of the conjugacy classes in $\Aut(V_\Lambda)$ is given in \autoref{prop:conjclassbij}. It is a special property of the Leech lattice $\Lambda$ that the algebraic conjugacy classes in $\O(\Lambda)\cong\Co_0$ are uniquely determined by their cycle shape. The $70$ automorphisms project to $11$ different cycle shapes in $\O(\Lambda)$ (listed in \autoref{table:11}) with non-trivial fixed-point sublattices. These are exactly the same conjugacy classes that appear in \cite{Hoe17}, Table~4. Once a conjugacy class in $\O(\Lambda)$ is specified, we choose suitable elements $h\in\pi_\nu(\Lambda_\C)/\pi_\nu(\Lambda)$ (listed in \autoref{table:70expl}) and consider the automorphism $g=\phi_{\eta}(\nu)\sigma_h$ where $\phi_{\eta}(\nu)$ is a standard lift of $\nu\in\O(\Lambda)$. Up to conjugacy $g$ does not depend on the choice of the standard lift of $\nu$.

We explain the entries of \autoref{table:70}. Each block corresponds to a cycle shape of $\nu$ and each row to a choice of $h$, determining the automorphism $g$ of order $n$ (up to conjugacy). We list the (yet to be proved) Lie algebra structure of $(V_\Lambda^{\orb(g)})_1$ together with its dimension and rank, of which the latter only depends on $\nu$. The first column corresponds to the labelling in Table~1 in \cite{Sch93} and the second one to the labelling in Table~4 of \cite{Hoe17}. Then we list the conformal weights $\rho(V_\Lambda(g^c))$ of the unique irreducible $g^c$-twisted $V_\Lambda$-modules for $c\mid n$ in ascending order. (Note that for the $70$ chosen automorphisms $\rho(V_\Lambda(g^i))$ only depends on $(i,n)$.) The last column will be explained below.

\smallskip

And now to the proof:
\begin{enumerate}[wide]
\item  First, we note that in all cases, $\rho(V_\Lambda(g))=1$ so that $g$ has type~$0$, which implies that the corresponding cyclic orbifold construction $V_\Lambda^{\orb(g)}$ exists and is a \strathol{} \voa{} of central charge $24$. Moreover $\rk((V_\Lambda^{\orb(g)})_1)\geq\rk((V_\Lambda^g)_1)=\rk(\Lambda^\nu)>0$ since $\nu$ has non-trivial fixed-point sublattice $\Lambda^\nu$. Hence, the Lie algebra $(V_\Lambda^{\orb(g)})_1$ is one of the $70$ non-zero (semisimple or abelian) Lie algebras on Schellekens' list.

\item The $24$ automorphisms with $\nu=\id$ are exactly the inner automorphisms $g=\sigma_h$ corresponding to the $23$ classes of deep holes $h$ of the Leech lattice $\Lambda$ and the identity. We proved in \autoref{prop:deepholes} that they are \gdh{}s and in \autoref{prop:deepholecons} that the corresponding orbifold constructions realise the $24$ Niemeier lattice \voa{}s, with weight-$1$ Lie algebras of rank $24$.

\item It remains to study the $46$ classes of automorphisms $g=\phi_{\eta}(\nu)\sigma_h$ with non-trivial $\nu$, i.e.\ with $\rk((V_\Lambda^g)_1)=\rk(\Lambda^\nu)<24$. We begin by proving that $\rk((V_\Lambda^{\orb(g)})_1)=\rk((V_\Lambda^g)_1)$ or equivalently that the abelian Lie algebra $(V_\Lambda^g)_1$ is a Cartan subalgebra of $(V_\Lambda^{\orb(g)})_1$, which is one of the two properties of a \gdh{} and additionally determines the rank of the Lie algebra $(V_\Lambda^{\orb(g)})_1$. (In particular, we may then assume that $(V_\Lambda^{\orb(g)})_1$ is semisimple as the only abelian case on Schellekens' list has rank $24$.)

Indeed, by the inverse orbifold construction, $(V_\Lambda^{\orb(g)})_1$ is a $\Z_n$-graded Lie algebra with $0$-component $(V_\Lambda^g)_1$. If the centraliser $C_{W^{(i,0)}_1}((V_\Lambda^g)_1)\subseteq C_{V_\Lambda(g^i)_1}((V_\Lambda^g)_1)=\{0\}$ for all $i\in\Z_n\setminus\{0\}$, then Lemma~8.1 in \cite{Kac90} implies that $(V_\Lambda^g)_1$ is a Cartan subalgebra of $(V_\Lambda^{\orb(g)})_1$. Now, by the definition of the twisted modules for lattice \voa{}s (see \autoref{sec:lat}), the centraliser condition for $i=1$ follows from $h\notin\pi_\nu(\Lambda)$ and similarly for the other powers of $g$ (essentially, the condition becomes $ih\notin\pi_\nu(\Lambda)$ but an extra term appears if $\phi_{\eta}(\nu)^i$ is not a standard lift of $\nu^i$). It is easy to check that all automorphisms $g=\phi_{\eta}(\nu)\sigma_h$ satisfy this (in fact, they were chosen this way), proving that $\rk((V_\Lambda^{\orb(g)})_1)=\rk((V_\Lambda^g)_1)$.

\item In $31$ of the remaining $46$ cases, we observe that $\rho(V_\Lambda(g^i))\geq1$ for all $i\in\Z_n\setminus\{0\}$, implying that $g$ is extremal by \autoref{prop:necsufff}. This shows that $g$ is a \gdh. Moreover, the dimension formula in \autoref{thm:dimform2} then yields the dimension of $(V_\Lambda^{\orb(g)})_1$ as the one listed in \autoref{table:70}.

For $16$ of these \gdh{}s $g$, the rank and dimension of the Lie algebra $(V_\Lambda^{\orb(g)})_1$ uniquely determine it on Schellekens' list. These are the cases $2$, $3$, $4$, $5$, $6$, $7$, $14$, $16$, $17$, $27$, $28$, $35$, $36$, $44$, $45$ and $50$.

\item For the other $15$ \gdh{}s $g$, namely those we claim to correspond to cases $8$, $9$, $10$, $11$, $13$, $18$, $19$, $20$, $21$, $22$, $34$, $26$, $31$, $33$ and $40$ on Schellekens' list, the Lie algebra $(V_\Lambda^{\orb(g)})_1$ is not uniquely determined by its rank and dimension. Here, we need an additional argument.

The inverse orbifold automorphism $\zeta$ on $V_\Lambda^{\orb(g)}$ restricts to an inner automorphism of order dividing $n$ of the semisimple Lie algebra $(V_\Lambda^{\orb(g)})_1$. The theory of finite-order automorphisms of simple Lie algebras was developed by Kac (see \autoref{sec:lieaut} and \cite{Kac90}). The said automorphism is regular, i.e.\ the fixed-point Lie subalgebra $(V_\Lambda^g)_1$ is abelian, and in fact a Cartan subalgebra of $(V_\Lambda^{\orb(g)})_1$.

Then, for a fixed $i\in\Z_n$ with $(i,n)=1$ we consider the adjoint action of $(V_\Lambda^g)_1$ on $(V_\Lambda(g^i))_1\subseteq(V_\Lambda^{\orb(g)})_1$. This defines a subset of the root system of $(V_\Lambda^{\orb(g)})_1$. Crucially, $(V_\Lambda(g^i))_1=W^{(i,0)}_1$ is the eigenspace of the automorphism $\zeta^{i^{-1}}$ (where $i^{-1}$ denotes the inverse of $i$ modulo $n$) restricted to $(V_\Lambda^{\orb(g)})_1$ with eigenvalue $e(1/n)$. By Proposition~8.1~c) in \cite{Kac90} the roots of $(V_\Lambda(g^i))_1$ are the \emph{simple} roots of a subdiagram of the (untwisted) affinisation of the Dynkin diagram associated with the semisimple Lie algebra $(V_\Lambda^{\orb(g)})_1$. (More precisely, it is the subdiagram consisting of those nodes in the affine Dynkin diagram for which the corresponding entry of the sequence $s=(s_0,\ldots,s_l)$ characterising $\zeta^{i^{-1}}$ equals $1$.)

For a given automorphism $g=\phi_{\eta}(\nu)\sigma_h$ this Dynkin diagram associated with the action of $(V_\Lambda^g)_1$ on $(V_\Lambda(g^i))_1$ can be easily determined using the definition of the twisted modules for lattice \voa{}s. We list the result in the last column of \autoref{table:70} and note that it is the same for all $i\in\Z_n$ with $(i,n)=1$.

Then, for each of the $15$ \gdh{}s we check if this Dynkin diagram of $(V_\Lambda(g))_1$ is compatible with the remaining potential Lie algebras for $(V_\Lambda^{\orb(g)})_1$. In each case, this narrows down the list of possible Lie algebras on Schellekens' list to just one. For example, the \gdh{} we list for case $40$ on Schellekens' list has an orbifold Lie algebra $(V_\Lambda^{\orb(g)})_1$ of rank $16$ and dimension $144$. This leaves the cases $38$, $39$ and $40$ on Schellekens' list. However, of the affine diagrams $\tilde{C}_4^4$, $\tilde{B}_3^2\tilde{C}_4\tilde{D}_6$ and $\tilde{A}_4\tilde{A}_9\tilde{B}_3$ only the latter contains $A_1\tilde{A}_9$ as a subdiagram, leaving only case $40$, i.e.\ $A_4A_9B_3$, as weight-$1$ Lie algebra of $V_\Lambda^{\orb(g)}$.

\item It remains to study the $15=46-31$ automorphisms $g$ of which we do not know the dimension of $(V_\Lambda^{\orb(g)})_1$ a priori by using the dimension formula. (By \autoref{cor:dimbound2} we do however know that the dimensions listed in \autoref{table:70} must be upper bounds.) For ten of these automorphisms, namely those we claim to correspond to cases $25$, $32$, $38$, $39$, $47$, $48$, $52$, $53$, $56$ and $62$ on Schellekens' list, the argument we just described involving the Dynkin diagram of $(V_\Lambda(g))_1$ suffices to reduce the list of possible Lie algebra structures of $(V_\Lambda^{\orb(g)})_1$ to just one case. In particular, this determines the dimension of $(V_\Lambda^{\orb(g)})_1$ and gives an a posteriori proof that these automorphisms are \gdh{}s.

For example, the automorphism we list for case $62$ has an orbifold Lie algebra $(V_\Lambda^{\orb(g)})_1$ of rank $16$ and dimension bounded from above by $384$, which is compatible with all $17$ rank-$16$ Lie algebras on Schellekens' list. However, the Dynkin diagram associated with the action of $(V_\Lambda^g)_1$ on $(V_\Lambda(g))_1$ is $A_1\tilde{E}_8$. Noting that the diagram $\tilde{E}_8$ is not a subdiagram of any other irreducible affine Dynkin diagram, an inspection of the Lie algebras of rank $16$ in \autoref{table:70} reveals that only case $62$ on Schellekens' list is compatible. This Lie algebra $B_8E_8$ has dimension $384$ so that $g$ is extremal and a \gdh{}.

\item We are left with the five automorphisms $g$ of order $6$, $10$, $14$, $22$ and $46$ that we postulate to yield cases $12$, $23$, $29$, $41$ and $57$, i.e.\ the Lie algebras $B_2^6$, $B_3^4$, $B_4^3$, $B_6^2$ and $B_{12}$, respectively, in the orbifold construction. These Lie algebras of rank $12$ are too similar to be distinguished by the Dynkin diagram of $(V_\Lambda(g))_1$. We can however show that none of the rank-$12$ Lie algebras on Schellekens' list, except for cases $2$, $12$, $23$, $29$, $41$ and $57$, are compatible with the Dynkin diagram of $(V_\Lambda(g))_1$ in each case. This means that the Lie algebra $(V_\Lambda^{\orb(g)})_1$ is uniquely determined once we know its dimension.

This dimension can be determined using the formulae in \autoref{sec:charlator}. It is a standard exercise to compute the $q$-expansions of the theta series of the lattices $M_{k,d}$ appearing in the functions $D_k$ (for instance, this takes less than a second using the function \texttt{ThetaSeriesModularForm} in \texttt{Magma} \cite{Magma}) to sufficiently high order so as to identify them as modular forms. Let $p=n/2$, half the order of $g$. Then we find
\begin{align*}
D_1(\tau)&=(p-1)T_{2p+p}(\tau),\\
D_2(\tau)&=(p-1)(T_{p+}(\tau)+24),\\
D_p(\tau)&=T_{2-}(\tau),\\
D_{2p}(\tau)&=T_1(\tau)+24=j(\tau)-720
\end{align*}
in terms of certain McKay-Thompson series (see \cite{CN79}). We can also write these functions as eta products. For example,
\begin{equation*}
D_1(\tau)=(p-1)\left(\frac{\eta(\tau)\eta(p\tau)}{\eta(2\tau)\eta(2p\tau)}\right)^{24/(p+1)}\!\!\!\!+\frac{24(p-1)}{p+1}.
\end{equation*}
It is now easy to compute the expansions of the $D_k$ in the different cusps (for instance by using the well-known modular transformation properties of the eta function). We obtain
\begin{equation*}
\ch_{V_\Lambda^{\orb(g)}}(\tau)=j(\tau)-720+12p=q^{-1}+24+12p+O(q)
\end{equation*}
for the character of $V_\Lambda^{\orb(g)}$. Hence, $\dim((V_\Lambda^{\orb(g)})_1)=24+12p=24+6n$, which together with the previous remarks uniquely determines the Lie algebra $(V_\Lambda^{\orb(g)})_1$ for the five automorphisms $g$ in question as cases $12$, $23$, $29$, $41$ and $57$, respectively, on Schellekens' list and in particular shows that the automorphisms are extremal and hence \gdh{}s. This completes the proof.\qedhere
\end{enumerate}
\end{proof}


\subsection{Classification of \GDH{}s}\label{sec:class}
Finally, we give a complete classification of the \gdh{}s of the Leech lattice \voa{} $V_\Lambda$.

\autoref{thm:main2} and the classification of \strathol{} \voa{}s of central charge $24$ (see \cite{DM04,LS19,LS15a,LL20,EMS20b,LS20,KLL18}, a uniform proof is given in \cite{HM20}) yield a classification of the \gdh{}s with $\rk((V_\Lambda^g)_1)>0$. It remains to study those with $\rk((V_\Lambda^g)_1)=0$.

Recall that the Moonshine module $V^\natural$ is a \strathol{} \voa{} $V$ of central charge $24$ with $V_1=\{0\}$ and conjecturally the only such \voa{} up to isomorphism. There are $38$ orbifold constructions of the Moonshine module $V^\natural$ from the Leech lattice \voa{} $V_\Lambda$ \cite{Car18}, including the original construction in \cite{FLM88} of order~$2$.

In \cite{Car18}, based on ideas in \cite{Tui92,Tui95}, the author proves that for all fixed-point free automorphisms $\nu\in\O(\Lambda)$ with vacuum anomaly $\rho_\nu>1$ (there are exactly $40$ such conjugacy classes that give type~$0$, $38$ up to algebraic conjugacy) the corresponding standard lift $g=\phi_{\eta}(\nu)\in\Aut(V_\Lambda)$ (if a lattice automorphism is fixed-point free, all its lifts are standard lifts and conjugate) defines an orbifold construction that yields the Moonshine module $V^\natural$. (By generalising the orbifold construction to non-zero type, a total of $53$ conjugacy classes, $51$ up to algebraic conjugacy, are considered; hence the title of the paper.) No order doubling occurs. We list the $38$ algebraic conjugacy classes in \autoref{table:38}.

\begin{prop}\label{prop:moonshinegdh}
Let $g\in\Aut(V_\Lambda)$ be one of the $38$ automorphisms in \cite{Car18} such that $V_\Lambda^{\orb(g)}\cong V^\natural$. Then $g$ is a \gdh{}.
\end{prop}
\begin{proof}
All fixed-point free automorphisms $\nu\in\O(\Lambda)$ with $\rho_\nu>1$ satisfy (see \autoref{sec:lat})
\begin{equation*}
\rho(V_\Lambda(g))=\rho_\nu=1+1/n
\end{equation*}
where $n$ is the order of $\nu$ and of $g=\phi_{\eta}(\nu)$. Then the dimension bound in the special case of the Leech lattice \voa{} $V_\Lambda$ (see \autoref{cor:dimleech}) yields
\begin{equation*}
\dim((V_\Lambda^{\orb(g)})_1)\leq24+24n(1-\rho_\nu)=0,
\end{equation*}
which implies that $\dim((V_\Lambda^{\orb(g)})_1)=0$ and that $g$ is extremal. Since $(V_\Lambda^g)_1$ is a subspace of $(V_{\Lambda}^{\orb(g)})_1$, $\rk((V_\Lambda^g)_1)=\rk((V_\Lambda^{\orb(g)})_1)=0$. This shows that each $g=\phi_{\eta}(\nu)$ is a \gdh{} with $\dim((V_\Lambda^{\orb(g)})_1)=0$ suggesting that $V_\Lambda^{\orb(g)}\cong V^\natural$, which would immediately follow if it were known that $V^\natural$ is the unique \strathol{} \voa{} $V$ of central charge $24$ with $V_1=\{0\}$. The main work of \cite{Car18} is to prove that indeed $V_\Lambda^{\orb(g)}\cong V^\natural$.
\end{proof}
We shall see that these $38$ automorphisms of $V_\Lambda$ are in fact all the \gdh{}s of $V_\Lambda$ with $\rk(V_\Lambda^g)=0$ up to algebraic conjugacy.

Note that one can show in a case-by-case study that the orbifold constructions associated with the exactly $42$ conjugacy classes, $39$ up to algebraic conjugacy, of fixed-point free automorphisms $\nu\in\O(\Lambda)$ with $\rho_\nu\leq1$ all yield the Leech lattice \voa{} $V_\Lambda$ again (see also Remark~7.5 in \cite{Car18}).

Now we can prove:
\begin{thm}[Classification of \GDH{}s]\label{thm:class}
There are exactly $108$ algebraic conjugacy classes of \gdh{}s $g$ of the Leech lattice \voa{} $V_\Lambda$ (see \autoref{table:11}), namely
\begin{enumerate}
\item $70$ classes with $\rk((V_\Lambda^g)_1)>0$, whose corresponding orbifold constructions yield the $70$ \strathol{} \voa{}s $V$ of central charge $24$ with $V_1\neq\{0\}$ (listed in \autoref{table:11expl} and \autoref{table:70expl} and summarised in \autoref{table:70}), and
\item $38$ classes with $\rk((V_\Lambda^g)_1)=0$, whose corresponding orbifold constructions all yield the Moonshine module $V^\natural$ (listed in \autoref{table:38}).
\end{enumerate}
\end{thm}
\begin{proof}
By \autoref{thm:main3} there are at least $70$ algebraic conjugacy classes of \gdh{}s $g$ with $\rk((V_\Lambda^g)_1)>0$. Schellekens' classification and the uniqueness of a \strathol{} \voa{} of central charge $24$ with a given non-zero weight-$1$ space together with \autoref{thm:main2} imply that there are exactly $70$ classes.

We study the \gdh{}s $g\in\Aut(V_\Lambda)$ with $\rk((V_\Lambda^g)_1)=0$. As explained above, the $38$ automorphisms whose corresponding orbifold constructions yield the Moonshine module $V^\natural$ are \gdh{}s. Clearly, they must satisfy $\rk((V_\Lambda^g)_1)=0$. In fact, these are all the \gdh{}s of $V_\Lambda$ with $\rk((V_\Lambda^g)_1)=0$.

Indeed, suppose that $g$ is an automorphism in $\Aut(V_\Lambda)$ with $\rk((V_\Lambda^g)_1)=0$, i.e.\ $g$ projects to a fixed-point free automorphism $\nu\in\O(\Lambda)$. By \autoref{prop:conjclassbij} the (algebraic) conjugacy classes of such automorphisms are in bijection with the (algebraic) conjugacy classes of fixed-point free automorphisms $\nu\in\O(\Lambda)$, of which there are exactly $95$ ($90$). Now suppose that $g$ is a \gdh{}. Then $\rk((V_\Lambda^{\orb(g)})_1)=0$ and $\rho(V_\Lambda(g))\geq1$. We even have $\rho(V_\Lambda(g))>1$ because $(V_\Lambda(g))_1\subseteq (V_\Lambda^{\orb(g)})_1$. As the automorphism is fixed-point free on $\Lambda$, the conformal weight is simply $\rho(V_\Lambda(g))=\rho_\nu$, which only depends on the cycle shape of $\nu$. There are exactly $53$ ($51$) automorphisms in $\O(\Lambda)$ which are fixed-point free and satisfy $\rho_\nu>1$ up to (algebraic) conjugacy and of those $40$ ($38$) have type~$0$. These are exactly the $38$ automorphisms listed in \autoref{table:38}.
\end{proof}

Finally, we record the following observation (see also \autoref{table:11}):
\begin{thm}\label{thm:main4}
Under the natural projection $\Aut(V_\Lambda)\to\O(\Lambda)$ the $70$ algebraic conjugacy classes of \gdh{}s $g$ with $\rk((V_\Lambda^g)_1)>0$ map to the $11$ algebraic conjugacy classes in $\O(\Lambda)$ with cycle shapes $1^{24}$, $1^82^8$, $1^63^6$, $2^{12}$, $1^42^24^4$, $1^45^4$, $1^22^23^26^2$, $1^37^3$, $1^22^14^18^2$, $2^36^3$ and $2^210^2$.
\end{thm}
This shows that we recover the decomposition of the genus of the Moonshine module described by Höhn in \cite{Hoe17}. There is probably an automorphic proof of this result (cf.\ \cite{Sch17}).


\section*{Appendix. List of \GDH{}s}
We give a complete list of the \gdh{}s of $V_\Lambda$.

First, we describe the $70$ \gdh{}s $g=\phi_{\eta}(\nu)\sigma_h\in\Aut(V_\Lambda)$ from \autoref{thm:main3} that give cyclic orbifold constructions of the $70$ \strathol{} \voa{}s of central charge $24$ with $V_1\neq\{0\}$. Recall that many of their properties are listed in \autoref{table:70}.

We shall explain our notation. Let $\mathcal{G}_{24}\subseteq\mathbb{F}_2^{24}$ be the \emph{extended binary Golay code}, a doubly-even, self-dual binary code in dimension $24$. Let $\Omega:=\{1,2,\ldots,24\}$ denote the coordinates of $\mathbb{F}_2^{24}$ (and of $\R^{24}$). The automorphism group of $\mathcal{G}_{24}$, i.e.\ the group of permutations of $\Omega$ mapping $\mathcal{G}_{24}$ to itself, is the Mathieu group $\operatorname{M}_{24}\subseteq\mathcal{S}_\Omega$. Following Curtis' \emph{Miracle Octad Generator} we shall arrange vectors in $\mathbb{F}_2^{24}$ (and in $\R^{24}$) in $(4\times 6)$-arrays. For definiteness, let $\mathcal{G}_{24}$ be the binary code generated by the standard basis (see, e.g., (5.35) in \cite{Gri98b}).

Let $\R^{24}$ be equipped with the standard scalar product $(\cdot,\cdot)$ so that $(e_i,e_j)=\delta_{i,j}$ for all $i,j\in\Omega$. The \emph{Leech lattice} $\Lambda\subseteq\R^{24}$ is a positive-definite, even, unimodular lattice of rank $24$. It can be constructed from $\mathcal{G}_{24}$ as the $\Z$-span of
\begin{equation*}
\frac{1}{\sqrt{2}}x,\;x\in\mathcal{G}_{24},\quad\sqrt{2}(e_i\pm e_j),\;i,j\in\Omega,\quad\frac{1}{\sqrt{8}}(-3,1,\ldots,1)
\end{equation*}
where we consider the entries of $x\in\mathbb{F}_2^{24}$ as $0,1\in\R$. An explicit basis of $\Lambda$ is given, for example, in \cite{CS99}.
The automorphism group of the Leech lattice is the Conway group $\Co_0$. $\operatorname{M}_{24}$ can be considered as a subgroup of $\Co_0$, acting on $\Lambda$ by permuting the coordinates.

In \autoref{table:11expl} (see also \autoref{table:11}) we list representatives $\nu$ of the $11$ algebraic conjugacy classes of automorphisms in $\O(\Lambda)\cong\Co_0$ that the $70$ \gdh{}s project to. For all cycle shapes but $2^36^3$ these classes have representatives in $\operatorname{M}_{24}$, allowing us to write them as permutations of the coordinates $\Omega$ of $\R^{24}$. We list particularly nice representatives, partly taken from \cite{Gri98b}, (5.38) and \cite{CS99}, Chapter~11. The cycle shape $2^36^3$ may be realised as a signed permutation. In \autoref{table:11expl} arrows denote permutations of the corresponding coordinates and a hollow node indicates a sign change.

In \autoref{table:70expl}, for each automorphism $\nu\in\O(\Lambda)$ in \autoref{table:11expl} (separated by a horizontal line), we list elements $h\in\Lambda_\C$ such that $g=\phi_{\eta}(\nu)\sigma_h\in\Aut(V_\Lambda)$ are \gdh{}s, amounting to a total of $70$ \gdh{}s. To improve legibility the least common denominator of $h$ is factored out and written under the vector. The first $23$ entries are representatives of the classes of deep holes of $\Lambda$ as they are given in \cite{CPS82}.

\begin{table}
\caption{Representatives of $11$ conjugacy classes in $\O(\Lambda)$.}
\renewcommand{\arraystretch}{1.2}
\begin{tabular}{ccc}
\begin{tikzpicture}
\draw[step=1cm] (0,0) grid (3,2);
\fill(0.25,0.25) circle(1.5pt);
\fill(0.75,0.25) circle(1.5pt);
\fill(1.25,0.25) circle(1.5pt);
\fill(1.75,0.25) circle(1.5pt);
\fill(2.25,0.25) circle(1.5pt);
\fill(2.75,0.25) circle(1.5pt);
\fill(0.25,0.75) circle(1.5pt);
\fill(0.75,0.75) circle(1.5pt);
\fill(1.25,0.75) circle(1.5pt);
\fill(1.75,0.75) circle(1.5pt);
\fill(2.25,0.75) circle(1.5pt);
\fill(2.75,0.75) circle(1.5pt);
\fill(0.25,1.25) circle(1.5pt);
\fill(0.75,1.25) circle(1.5pt);
\fill(1.25,1.25) circle(1.5pt);
\fill(1.75,1.25) circle(1.5pt);
\fill(2.25,1.25) circle(1.5pt);
\fill(2.75,1.25) circle(1.5pt);
\fill(0.25,1.75) circle(1.5pt);
\fill(0.75,1.75) circle(1.5pt);
\fill(1.25,1.75) circle(1.5pt);
\fill(1.75,1.75) circle(1.5pt);
\fill(2.25,1.75) circle(1.5pt);
\fill(2.75,1.75) circle(1.5pt);
\end{tikzpicture}
&
\begin{tikzpicture}
\draw[step=1cm] (0,0) grid (3,2);
\fill(0.25,0.25) circle(1.5pt);
\fill(0.75,0.25) circle(1.5pt);
\fill(1.25,0.25) circle(1.5pt);
\fill(1.75,0.25) circle(1.5pt);
\fill(2.25,0.25) circle(1.5pt);
\fill(2.75,0.25) circle(1.5pt);
\fill(0.25,0.75) circle(1.5pt);
\fill(0.75,0.75) circle(1.5pt);
\fill(1.25,0.75) circle(1.5pt);
\fill(1.75,0.75) circle(1.5pt);
\fill(2.25,0.75) circle(1.5pt);
\fill(2.75,0.75) circle(1.5pt);
\fill(0.25,1.25) circle(1.5pt);
\fill(0.75,1.25) circle(1.5pt);
\fill(1.25,1.25) circle(1.5pt);
\fill(1.75,1.25) circle(1.5pt);
\fill(2.25,1.25) circle(1.5pt);
\fill(2.75,1.25) circle(1.5pt);
\fill(0.25,1.75) circle(1.5pt);
\fill(0.75,1.75) circle(1.5pt);
\fill(1.25,1.75) circle(1.5pt);
\fill(1.75,1.75) circle(1.5pt);
\fill(2.25,1.75) circle(1.5pt);
\fill(2.75,1.75) circle(1.5pt);
\draw[thick] (1.25, 1.75) --  (1.25,1.25);
\draw[thick] (1.75, 1.75) --  (1.75,1.25);
\draw[thick] (1.25, 0.75) --  (1.25,0.25);
\draw[thick] (1.75, 0.75) --  (1.75,0.25);
\draw[thick] (2.25, 1.75) --  (2.25,1.25);
\draw[thick] (2.75, 1.75) --  (2.75,1.25);
\draw[thick] (2.25, 0.75) --  (2.25,0.25);
\draw[thick] (2.75, 0.75) --  (2.75,0.25);
\end{tikzpicture}
&
\begin{tikzpicture}
\draw[step=1cm] (0,0) grid (3,2);
\fill(0.25,0.25) circle(1.5pt);
\fill(0.75,0.25) circle(1.5pt);
\fill(1.25,0.25) circle(1.5pt);
\fill(1.75,0.25) circle(1.5pt);
\fill(2.25,0.25) circle(1.5pt);
\fill(2.75,0.25) circle(1.5pt);
\fill(0.25,0.75) circle(1.5pt);
\fill(0.75,0.75) circle(1.5pt);
\fill(1.25,0.75) circle(1.5pt);
\fill(1.75,0.75) circle(1.5pt);
\fill(2.25,0.75) circle(1.5pt);
\fill(2.75,0.75) circle(1.5pt);
\fill(0.25,1.25) circle(1.5pt);
\fill(0.75,1.25) circle(1.5pt);
\fill(1.25,1.25) circle(1.5pt);
\fill(1.75,1.25) circle(1.5pt);
\fill(2.25,1.25) circle(1.5pt);
\fill(2.75,1.25) circle(1.5pt);
\fill(0.25,1.75) circle(1.5pt);
\fill(0.75,1.75) circle(1.5pt);
\fill(1.25,1.75) circle(1.5pt);
\fill(1.75,1.75) circle(1.5pt);
\fill(2.25,1.75) circle(1.5pt);
\fill(2.75,1.75) circle(1.5pt);
\draw[thick] (0.25, 1.25) --  (0.25,0.75);
\draw[thick,->] (0.25, 0.75) --  (0.25,0.29);
\draw[thick] (0.75, 1.25) --  (0.75,0.75);
\draw[thick,->] (0.75, 0.75) --  (0.75,0.29);
\draw[thick] (1.25, 1.25) --  (1.25,0.75);
\draw[thick,->] (1.25, 0.75) --  (1.25,0.29);
\draw[thick] (1.75, 1.25) --  (1.75,0.75);
\draw[thick,->] (1.75, 0.75) --  (1.75,0.29);
\draw[thick] (2.25, 1.25) --  (2.25,0.75);
\draw[thick,->] (2.25, 0.75) --  (2.25,0.29);
\draw[thick] (2.75, 1.25) --  (2.75,0.75);
\draw[thick,->] (2.75, 0.75) --  (2.75,0.29);
\end{tikzpicture}
\\
\end{tabular}
\renewcommand{\arraystretch}{1}
\label{table:11expl}
\end{table}

\begin{table}
\caption{Representatives of $70$ orbits of the action of $C_{\O(\Lambda)}(\nu)$ on $\pi_\nu(\Lambda_\C)/\pi_\nu(\Lambda)$ with $\nu$ from \autoref{table:11expl} (continued on next page)}
\tiny
\setlength{\tabcolsep}{1pt}
\renewcommand{\arraystretch}{1.1}
%
\renewcommand{\arraystretch}{1}
\label{table:70expl}
\end{table}
\addtocounter{table}{-1}
\begin{table}
\caption{(continued)}
\tiny
\setlength{\tabcolsep}{1pt}
\renewcommand{\arraystretch}{1.09}
%
\renewcommand{\arraystretch}{1}
\end{table}

Finally, we list the $38$ \gdh{}s of $V_\Lambda$ whose orbifolds give the Moonshine module $V^\natural$.

\begin{table}
\caption{The $38$ \gdh{}s of $V_\Lambda$ whose corresponding orbifold constructions are isomorphic to $V^\natural$ (and their class in $\Co_1=\O(\Lambda)/\{\pm1\}$ \cite{CCNPW85}).}
\renewcommand{\arraystretch}{1.2}
\begin{tabular}{rlll}
$n$ & Cycle shape & $\rho(V_\Lambda(g))-1$ & $\Co_1$ \\\hline
 2 & $1^{-24}2^{24}$                                    & $1/2$  &  1A  \\
 3 & $1^{-12}3^{12}$                                    & $1/3$  &  3A  \\
 4 & $1^{-8}4^{8}$                                      & $1/4$  &  4A  \\
 5 & $1^{-6}5^{6}$                                      & $1/5$  &  5A  \\
 6 & $1^{-6}2^{6}3^{-6}6^{6}$                           & $1/6$  &  3B  \\
 6 & $1^{-4}2^{-4}3^{4}6^{4}$                           & $1/6$  &  6A  \\
 6 & $1^{-5}2^{1}3^{-1}6^{5}$                           & $1/6$  &  6D  \\
\end{tabular}
\renewcommand{\arraystretch}{1}
\label{table:38}
\end{table}

\FloatBarrier


\enlargethispage{\baselineskip}

\bibliographystyle{alpha_noseriescomma}
\bibliography{quellen}{}

\end{document}